\documentclass[11pt]{article}
\usepackage[margin=0.86in]{geometry}
\usepackage{amsmath,amssymb,amsthm}
\usepackage{array,booktabs,longtable}
\usepackage{framed}
\renewenvironment{frame}{\begin{framed}}{\end{framed}}
\usepackage[colorlinks=true,linkcolor=blue,citecolor=green,urlcolor=black,
            hypertexnames=false]{hyperref}
\usepackage{microtype}
\usepackage{tikz}
\usetikzlibrary{shapes.geometric}
\newtheorem{theorem}{Theorem}[section]
\newtheorem{proposition}[theorem]{Proposition}
\newtheorem{lemma}[theorem]{Lemma}
\newtheorem{corollary}[theorem]{Corollary}
\newtheorem{remark}[theorem]{Remark}
\newtheorem{assumption}{Assumption}
\theoremstyle{definition}
\newtheorem{definition}{Definition}
\newtheorem{heuristic}{Heuristic}
\newtheorem{example}{Example}
\theoremstyle{plain}

\newcommand{\Hc}{\mathcal H}
\newcommand{\E}{\mathbb E}
\newcommand{\Pp}{\mathbb P}
\newcommand{\Eenv}{\mathbb E_{\rm env}}
\newcommand{\Penv}{\mathbb P_{\rm env}}
\newcommand{\Ew}{\mathbb E^{\omega}}
\newcommand{\Pw}{\mathbb P^{\omega}}
\newcommand{\one}{\mathbf 1}
\newcommand{\dd}{\,d}
\DeclareMathOperator{\Var}{Var}

\title{Quenched First-Passage Asymptotics for Branching Random Walk on a Hamming Cube}
\author{Jose Blanchet\thanks{Department of Management Science and Engineering,
Stanford University. Email: \texttt{jose.blanchet@stanford.edu}}
\and
Zhenyuan Zhang\thanks{Laboratory for Information and Decision Systems, Massachusetts Institute of Technology. Email:
\texttt{zzyzzy@mit.edu}}}
\date{{\today}}

\begin{document}
\maketitle

\begin{abstract}
We study continuous-time supercritical branching random walks in space-inhomogeneous random environment on the Hamming cube $\{0,1,\dots,b-1\}^d$, where the reproduction laws at each site are i.i.d. sampled. This serves as an idealized model for RNA sequence evolution with mutation--selection balance. The reproduction and mutation events are decoupled, and we assume that the reproduction law has essential supremum strictly below the deterministic mutation rate. Our main results provide tight asymptotics of the first-passage times for the model, uniformly in the origin--target distance $1\le m\le d$ as $d\to\infty$, conditional upon survival. We prove both quenched tightness in environment probability and annealed tightness around a deterministic center. Moreover, we identify the leading order of the deterministic center and develop an expansion in the sparse-distance regime $m=o(d)$. Our proof technique derives quantitative approximations of the model using an independent-visit variant, where revisits still resample the environment. As an application, we show that on a macroscopic scale, increasing the reproduction law in convex order decreases the first-passage times, and increasing the mutation rate increases the first-passage times for the sparse regime.
\end{abstract}

\tableofcontents

\section{Introduction}

RNA viruses generate genetic variation through mutations
during replication.  Classic quasispecies theory describes the resulting clouds of genotypes through the interplay between mutation and selection~\cite{Eigen1971}.  This diversity provides variants that
can differ in replication ability, host adaptation, or resistance to
antiviral drugs~\cite{DomingoSheldonPerales2012}.  In this paper, we develop an idealized yet fundamental probabilistic model on when a
growing population first produces a prescribed mutant of concern.

We model genotypes as sequences of length $d$ over an alphabet of
$b\ge2$ symbols, with $b=4$ for RNA. Typically $d$ is large, and our analytical results concern the regime $d\to\infty$. The genotype sequences form a Hamming cube $\{0,1,\dots,b-1\}^d$, where 
two vertices are neighbours if they differ in exactly one coordinate.
Each particle mutates at total rate one, choosing a neighbour uniformly,\footnote{For $b=4$, the mutation walk coincides with the classic 
Jukes--Cantor model~\cite{JukesCantor1969}.  See also~\cite{Durrett2008} for background on
probabilistic models of sequence evolution.}
and, at a genotype with net growth rate $\lambda$, splits at rate
$\max(\lambda,0)$ and dies at rate $\max(-\lambda,0)$.\footnote{Here, at each genotype, mutation, replication, and death use
independent exponential clocks.  A mutation changes the particle's
genotype without branching, whereas a split produces two particles with
the parent's genotype.  In other words, mutations are modeled as separate events,
rather than copying errors at replication.}  This defines a branching random walk (BRW) on the Hamming cube~\cite{BZ}.

To incorporate selection, we allow the net growth rate to vary with the
genotype, in contrast to the common rate in~\cite{BZ}.  We sample the
vertex rates independently from a given probability distribution and keep them fixed
throughout the evolution process. Consequently, all particles carrying the same genotype share the net growth rate. The resulting process is called a branching random walk in random environment (BRWRE) on the Hamming cube. This i.i.d.~sampling is the classical house-of-cards idealization of a
fitness landscape~\cite{Kingman1978,KauffmanLevin1987}: it neglects fitness
correlations even between genotypes differing at a single site.\footnote{An experiment on tobacco etch virus found weak correlations between the
fitness values of neighboring genotypes in landscapes formed by combinations
of five mutations.  The measured landscapes were described as intermediate
between the house-of-cards model and a model retaining
correlations~\cite{CerveraLalicElena2016}.  We use the uncorrelated
landscape as a tractable approximation.} Probabilities conditional on the landscape
are called \emph{quenched}; averaging over it gives the \emph{annealed} law.

Let $m$ denote the Hamming distance from the initial
genotype to the target, the number of coordinates in which they differ.
In this paper, we attempt to understand the following two fundamental questions, whose answers depend on $m$:
\begin{itemize}
    \item On what time scale does the target genotype first appear? How large are the fluctuations of this first-appearance time?
    \item Can increasing mutation, while keeping the net growth rates
    fixed, delay its first appearance? If yes, by how much?
\end{itemize}
 These questions naturally connect to \textit{first-passage times} (FPT) of BRWRE on the hypercube, which by definition refers to the first time any descendant reaches
the prescribed target. 
The second question, in particular, is motivated by antiviral mutagenesis. Molnupiravir, for example, increases errors during viral RNA replication
to suppress viral growth~\cite{GordonEtAl2021}.  The possibility that
mutagenic treatment also promotes the emergence of unwanted viral variants
has motivated studies of its evolutionary
safety~\cite{LobinskaPilpelNowak2023,LobinskaEtAl2024}. The first appearance of a prescribed mutant has also been studied using multitype branching processes \cite{NicholsonAntal2019,NicholsonCheekAntal2023}.

We address these questions by studying asymptotics of FPT of BRWRE on the hypercube uniformly in $m$ as $d\to\infty$.  
Here, the net growth rate distribution is fixed as $d$ grows and we assume that  {its essential supremum is strictly below} the mutation rate, which is also fixed as $d\to\infty$.\footnote{The difference from \cite{NicholsonAntal2019,NicholsonCheekAntal2023} is that their asymptotic results concern increasingly rare mutations on a fixed genotype structure; ours concern increasing size of the sequence space with a fixed total mutation rate.}  This assumption mirrors the slow-branching regime 
in~\cite{BZ}, where the FPT remains of order $d$ even for targets at
bounded positive Hamming distances.

\paragraph{Main contribution.}
We briefly summarize our main results as follows. Section~\ref{sec:main-results} contains their precise statements.
\begin{itemize}

\item We show that the first-passage time is tight under the
quenched law conditional on survival, in environment probability, around a deterministic center $\widehat\theta_{d,m}$, uniformly for distances $1\leq m\leq d$ (Theorem~\ref{thm:quenched-all-distance}); the same center gives annealed tightness conditional on survival.
In other words, for typical environments, fluctuations under the law conditional on the environment and survival remain on a bounded time scale. 

\item We characterize the deterministic center $\widehat\theta_{d,m}$ (Theorem~\ref{thm:sparse-expansion}). Specifically, we find an explicit function $f(d,m)$ such that uniformly in $m$, $|\widehat\theta_{d,m}-f(d,m)|=O(\log d)$.  Moreover, in the sparse-distance
regime $m=o(d)$, we derive the first-order expansion 
$\widehat\theta_{d,m}=c_1d+c_2m+o(m)+O(\log d)$ for some $c_1,c_2>0$. 

\item We study comparative statics of the deterministic center $\widehat\theta_{d,m}$ against the net growth and mutation rates. Clearly, a higher net growth rate decreases the unconditioned first-passage time by a coupling argument. Nevertheless, we show that a mean-preserving spread of the net growth-rate law also 
decreases $\widehat\theta_{d,m}$ 
(Corollary~\ref{cor:rate-convex-order}).  For sparse distances $m=o(d)$, increasing the mutation rate between two fixed values in the same supercritical interval specified by \eqref{eq:mutation-supercritical}, while keeping the net growth rates fixed, delays first passage by order $d$, with each process conditioned on its own survival (Corollary~\ref{thm:mutation-rate}).

\item 
Hinging on the connection between branching random walks and the Fisher--KPP equation, we discuss applications to the Fisher–KPP equation on the Hamming cube, which describes diffusion and random local growth limited by saturation. In the absence of deaths, we show that the nonlinear solution and its linearization reach any fixed level in $(0,1)$ at times differing by only $O_{\mathbb P_{\mathrm{env}}}(1)$, uniformly over the prescribed targets at all Hamming distances (Corollary~\ref{cor:kpp-arrival}).
\end{itemize}

\paragraph{Related literature.}
\label{sec:literature}
Branching random walks on Hamming cubes (with homogeneous environments), as well as their coalescing variants, have been studied in many different aspects. Examples include first-passage times \cite{BZ}, partial cover times~\cite{BalelliMilisicWainrib2016}, and cover times \cite{ZhangCover2026}. 
For the mean-field setting, the survey \cite{KoenigSurvey2021} studies branching random walks in static random environments and
their connection to the parabolic Anderson equation, which describes
their quenched first moments.
For the parabolic Anderson equation on the binary hypercube, see
\cite{AvenaGunHesse2020}.

First-passage times serve as a prototypical example for analyzing extremes of spatial branching processes, such as on Euclidean spaces~\cite{ZhangMohantyBlanchetCai2024,BlanchetCaiMohantyZhang2026,BlanchetZhangFPT2026}. In these works, precise asymptotics of the first-passage times are established. Their approach is based on the analysis of 
the genealogy near the one-dimensional extreme, which does not apply to the hypercube.

For BRW in space-inhomogeneous random environments on $\mathbb Z$,
\cite{KriechbaumSpace2024} proves annealed tightness of
first-passage times and the maximum  {after environment-dependent centering} under suitable assumptions, based on tilted passage
times across ordered levels and a modified second-moment argument.
Related works include subsequential tightness~\cite{KriechbaumSubseq2021}, functional CLT \cite{CernyDrewitz2020}, and almost-sure quenched tightness  {around quenched medians} for branching Brownian motion \cite{CernyDrewitzOswald2025} and for branching random walk on
$\mathbb Z$ \cite{CernyDalessi2026}.  {These tightness results use environment-dependent centers. On the Hamming cube, we obtain tightness around a deterministic center (see Remark~\ref{rem:fluctuation}).} For time-inhomogeneous environments, we refer to \cite{KriechbaumTime2025,MalleinMilos2019,Fang2012}.   For earlier studies on non-extremal behavior such as limiting shapes and growth rates of related models, including multi-dimensional settings, see \cite{CometsPopov2007,CometsPopovShape2007,CometsYoshida2011}.

\paragraph{Paper outline.}
Section~\ref{sec:main-results} defines the model and states the main results.
Section~\ref{sec:first-moment-analysis} develops the first-moment estimates and
deterministic centering.  Sections~\ref{sec:bridge-fractional-second-moment}  and \ref{sec:overlap-transfer} prepare tools for proving the tightness result.  Section~\ref{sec:main-results-proofs} collects the final proofs of main results and their applications.  Appendix~\ref{sec:notations} contains a table of frequently used notation.  Appendices~\ref{sec:linearprimer}
and~\ref{sec:auxiliary} collect elementary tools of use.

\paragraph{Notation conventions.}
We use $c,C\in(0,\infty)$ to denote deterministic constants whose values
may vary from line to line, where the lower case $c$ is understood as a small number and vice versa.  The notation
$a_d=O(b_d)$ means $|a_d|\le C|b_d|$ for all large $d$, while
$a_d=o(b_d)$ means $a_d/b_d\to0$ whenever the ratio is defined.  For
nonnegative quantities, $a_d\asymp b_d$ means
$cb_d\le a_d\le Cb_d$.  Each estimate inherits the uniformity stated in its
surrounding sentence or display.  We also write $a_d\ll b_d$, or
$b_d\gg a_d$, for $a_d\le Cb_d$, and we write $a_d\sim b_d$ for $a_d/b_d\to1$, or equivalently
$a_d=(1+o(1))b_d$; subscripts record permitted dependence
of the implicit constant.  

For sets in a Euclidean space, $K\Subset U$
means that $K$ is compact and contained in the interior of $U$.
For vertices, $x\sim y$ denotes adjacency.
The notation $\mathcal L(W)$ denotes the probability law
of a random variable $W$.  The symbol $\one_A$ is the indicator of an event
$A$; without a subscript, $\one$ denotes the all-ones vector of the dimension
dictated by context.
For $a\in\mathbb R$, put $(a)_+=\max(a,0)$.  We use
$\mathbb N=\{1,2,\ldots\}$ and $\mathbb N_0=\{0,1,2,\ldots\}$. For locally integrable functions on $[0,\infty)$, write
$(f*g)(t)=\int_0^tf(s)g(t-s)\dd s$ and let $h^{*N}$ denote the
$N$-fold convolution of $h$.

\section{Model setup and main results}
\label{sec:main-results}

\subsection{Model and assumptions}
\label{sec:model}

Fix an integer $b\ge2$ and consider the Hamming cube
$$
  \Hc_d^{(b)}=\{0,\ldots,b-1\}^d.
$$
For $x,y\in\Hc_d^{(b)}$, their Hamming distance is
$d_{\mathrm H}(x,y)=\#\{1\le j\le d:x_j\ne y_j\}$.
We write $x\sim y$ when $d_{\mathrm H}(x,y)=1$.
The continuous-time simple random walk (we call it \textit{mutation walk}) jumps at rate one and, at each jump, chooses one of its
$(b-1)d$ Hamming neighbours uniformly. The walk's one-step matrix and continuous-time generator are
\begin{equation*}
 P(x,y)=\bigl((b-1)d\bigr)^{-1}\one_{\{x\sim y\}},\qquad P-I,
\end{equation*}
respectively.

We study continuous-time branching random walks on the Hamming cube, where particles perform simple random walks at rate one and independently split or die at rates determined by their sites. The net growth rates depend on the specific sites on the Hamming cube and are sampled i.i.d.~from a given distribution, which we call the environment. 
We impose the following upper-bound and supercriticality assumptions on the environment.

\begin{assumption}
\label{ass:environment}
The environment is a family
$\omega=(\Lambda_v)_{v\in\Hc_d^{(b)}}$ of i.i.d. random variables with a fixed probability law
$P=\mathcal L(\Lambda_{\mathbf0})$ on $\mathbb R$ satisfying\footnote{We use $P(d\lambda)$ for the rate law; $P(x,y)$ and $P^n$ continue
to denote the one-step mutation kernel and its powers.}
\begin{equation}\label{eq:Ass1}
 \lambda_+:=\operatorname*{ess\,sup}\Lambda_{\mathbf0}<1\qquad\text{and}
 \qquad \int\frac{P(d\lambda)}{1-\lambda}>1.
\end{equation}
\end{assumption}

Note that we allow for negative rates and a rate $\Lambda_v=\lambda<0$ indicates particles at $v$ die at rate $|\lambda|$. Intuitively, the first condition in \eqref{eq:Ass1} implies that every family of particles initiated at a site produces on average a finite number of \(\int_0^\infty e^{-(1-\lambda)t}\,dt=\frac1{1-\lambda}\) many mutants, and such number is uniformly bounded from above. The second condition in \eqref{eq:Ass1} is the supercriticality condition for the independent-visit model defined below.

We write $\Penv$ and $\Eenv$ for probability and expectation under the environment law in Assumption~\ref{ass:environment}. Denote the mean splitting rate by
\begin{equation}
 \bar\lambda=\Eenv[(\Lambda_{\mathbf0})_+]
 =\int(\lambda)_+\,P(d\lambda)>0.
 \label{eq:atom-summary}
\end{equation}

\begin{definition}
The \emph{branching random walk in a static random environment} (BRWRE)
starts with one particle at $x$ in an environment $\omega$ sampled under
Assumption~\ref{ass:environment} and kept fixed throughout the evolution.
Each particle moves as the mutation walk and, while at $v$, is replaced by
two particles at rate $(\Lambda_v)_+$ and dies at rate $(-\Lambda_v)_+$.
At each split, conditional on
$\omega$ and the split location, the two descendant systems evolve
independently.  We also call this the \emph{static model}.
\end{definition}

We now introduce the following proxy process for the static model, called the independent-visit model. The key (and only) difference is that, for the static model, every particle at $v$ uses the
same rate, including on later visits; for the independent-visit model, an independent net growth rate is sampled every time a particle mutates, including returns to previous vertices.

\begin{definition}
The \emph{independent-visit branching random walk} starts with one particle
at $x$ whose signed rate is sampled from $P$.  Each particle moves as
the mutation walk.  At each mutation, it draws a new rate from $P$
independently of the past.  
A particle with current rate $i$ is replaced by two particles at rate
$(i)_+$ and dies at rate $(-i)_+$.  Both children inherit that rate and retain it until their
respective next mutations.  Conditional on the location and rate at a
split, the two descendant systems evolve independently.
\end{definition}

Finally, we define our key statistics, the first-passage time. We consider a starting vertex $x\in \Hc_d^{(b)}$ aiming to hit the target 
$\mathbf 0=(0,\ldots,0)$. 
For the BRWRE, write $N_t(y)$ for
the number of particles at $y$ at time $t$; in particular,
$N_t(\mathbf 0)$ is the number at the target.  
The \textit{first-passage time} from $x$ is defined by  $$\tau_x^\omega=\inf\{t\ge0:N_t(\mathbf 0)>0\}.$$
For a BRWRE starting at $x$, we denote the quenched law and
expectation by $\Pw_x$ and $\Ew_x$.

\subsection{Heuristics}
\label{sec:heuristics}
Let $(X_t)_{t\ge0}$ denote the mutation walk defined above, and let
$\Pp_x$ and $\E_x$ denote probability and expectation for that walk
started from $x$. 
The (quenched) first-moment of the particle count is thus given by 
\begin{equation*}
 K_t^\omega(x,y)
 =\Ew_x[N_t(y)]
 =\E_x\left[
   \exp\left(\int_0^t\Lambda_{X_s}\dd s\right)
   \one_{\{X_t=y\}}
 \right].
\end{equation*}
Write
$$
 \mathbf{x}_k=\bigl(\underbrace{1,\ldots,1}_{k},0,\ldots,0\bigr),
 \qquad 0\le k\le d.
$$
Let $\widehat K_{d,m}(t)$ denote the expected number of particles at
$\mathbf0$ at time $t$ in the independent-visit model started from
$\mathbf{x}_m$. By the symmetry of the Hamming cube, this fully characterizes the FPT.

In this section, we derive heuristically formulas and observations that motivate the statements of the main results, including:
\begin{itemize}
    \item Why the independent-visit model serves as a good proxy for the static model;
    \item Why the FPT for the independent-visit model stays concentrated;
    \item How to identify the leading linear term for the FPT in $d$;
    \item How to compute the change in FPT after a small perturbation of $m$ where $m=o(d)$.
\end{itemize}

\paragraph{Static and independent-visit models.}
Recall that in the static model, every visit to a vertex uses the same rate.  In the
independent-visit model, each mutation samples a fresh rate, even on a
return. 
There are two intuitive reasons on why the independent-visit model serves as a reasonable proxy for the static model:
\begin{itemize}
    \item The paths that 
contribute most to this expected population have few repeated vertices. This is because detours waste time, and each detour may select up to an order $d$ of neighbors which makes revisits rare.
\item Two independently chosen such paths typically have few intersections apart
from their prescribed common endpoints. This is because in a high-dimensional Hamming cube, there are many paths from the origin to the target with similar lengths.
\end{itemize}
 This suggests that the expected particle counts of the static model can be well-approximated by the independent-visit proxy, up to bounded multiplicative factors. Near the predicted arrival time, these factors should lead to only a bounded-in-probability time shift, conditional on survival.
Returns and overlaps may still occur, but their weighted contribution can be controlled quantitatively.

\begin{frame}
\begin{heuristic}
\label{heur:static-independent}
Conditional on survival, the independent-visit mean should predict the
FPT in the static model up to a bounded-in-probability additive error.
\end{heuristic}
\end{frame}

\paragraph{Concentration of the passage time near a deterministic center.}
On the Hamming cube in high dimensions, different routes from the origin to the target can sample largely different parts of the environment. Since first-passage intuitively describes the earliest arrival among many branching families, we expect that the limited overlap between those families yields enough independence between their earliest arrival times. Thus, the FPT comprises of the minimum of quantities described by many independent noises, and hence it is reasonable that the FPT is concentrated in a window of bounded size.\footnote{ By
contrast, a nearest-neighbour path to a distant target on {$\mathbb{Z}$}
must cross every intervening site; the common environment can then
produce a diffusive shift of the front~\cite{CernyDrewitz2020}; see Remark \ref{rem:fluctuation}.}

Next, we propose a candidate for the {deterministic centering} of the FPT. 
For the independent-visit model, define the unit-window mean and its level-one crossing time by
\begin{equation}
 \widehat M_t(\mathbf{x}_m)
 =\int_t^{t+1}\widehat K_{d,m}(s)\,\dd s,
 \qquad
 \widehat\theta_{d,m}
 =\inf\{t\ge0:\widehat M_t(\mathbf{x}_m)\ge1\}.\label{eq:thetahat definition}
\end{equation}
Thus, $\widehat\theta_{d,m}$ is deterministic once the rate law $P$ is fixed. Intuitively, $\widehat\theta_{d,m}$ is the threshold at which the unit-window averaged expected particle count approaches one. 
Near this level-one crossing, exponential
growth of the mean number of particles intuitively suggests, for fixed $A>0$ and sufficiently large $d$,
$$
 \widehat M_{\widehat\theta_{d,m}-A}(\mathbf{x}_m)\ll e^{-cA},
 \qquad
 \widehat M_{\widehat\theta_{d,m}+A}(\mathbf{x}_m)\gg e^{cA}.
$$
Therefore, a bounded multiplicative error in the mean can be
compensated by a bounded change in time, leading to a reasonable candidate $\widehat\theta_{d,m}$.

\begin{frame}
\begin{heuristic}
\label{heur:passage-concentration}
Conditional on survival, the FPT is concentrated in a window of bounded size near the deterministic center 
$\widehat\theta_{d,m}$.
\end{heuristic}
\end{frame}

\paragraph{Identifying the linear order of $\widehat\theta_{d,m}$.}
To locate the leading order in $d$ of the deterministic center $\widehat\theta_{d,m}$, we compute the independent-visit mean particle count $\widehat K_{d,m}(t)$. The idea is to look at the number of mutations and identify a range for its main contribution. 
Let $(S_n)_{n\ge0}$ be the embedded discrete jump process of $X$, and write
$p_{d,m}(n)=\Pp_{\mathbf{x}_m}(S_n=\mathbf0)$.
Along an ancestral path, $n$ mutations divide $[0,t]$ into $n+1$
holding intervals.  On an
interval of length $r$, the no-mutation weight $e^{-r}$ times the
net growth weight $e^{\lambda r}$, averaged over the rate distribution $P$, is given by
$$
 h(r)=\Eenv\bigl[e^{-(1-\Lambda_{\mathbf 0})r}\bigr]
 =\int e^{-(1-\lambda)r}\,P(d\lambda).
$$
Integrating the product over durations summing to $t$ gives
$h^{*(n+1)}(t)$. Accounting for the discrete paths ending at the target then gives
the exact identity
\begin{equation}
\widehat K_{d,m}(t)
 =\sum_{n\ge0}p_{d,m}(n)h^{*(n+1)}(t).                   \label{eq:freshK}
\end{equation}
To evaluate $\widehat K_{d,m}(t)$ at macroscopic scales, write $\theta=m/d$, $u=t/d$,
and $\nu=n/d>\theta$, with $u,\nu$ bounded above and away from zero. We separately deal with the two factors in \eqref{eq:freshK}.

\begin{itemize}
\item \emph{The spatial factor $p_{d,m}(n)$.}
For a power series $F$, let $[z^n]F(z)$ denote the coefficient of $z^n$.
The rate-one mutation-walk transition probability is given
in~\cite[Lemma~6]{BZ}:
\begin{equation}
 \Pp_{\mathbf{x}_m}(X_t=\mathbf 0)
 =e^{-t}A_b\Bigl(\frac{t}{d}\Bigr)^m
 C_b\Bigl(\frac{t}{d}\Bigr)^{d-m},
 \label{eq:mutation-transition}
\end{equation}
where for $z\ge0$, we set
\begin{equation}
 A_b(z)=\frac{e^z-e^{-\frac{z}{b-1}}}{b},\qquad
 C_b(z)=\frac{e^z+(b-1)e^{-\frac{z}{b-1}}}{b}.          \label{eq:AC}
\end{equation}
On the other hand, conditioning on the Poisson number of mutations by
time $t$ gives
\begin{equation}
 \Pp_{\mathbf{x}_m}(X_t=\mathbf 0)
 =e^{-t}\sum_{n\ge0}\frac{t^n}{n!}p_{d,m}(n).
 \label{eq:jump-count-conditioning}
\end{equation}
Comparing \eqref{eq:mutation-transition} and \eqref{eq:jump-count-conditioning}
and writing the generating-function
variable as $z$ yields
\[
 \sum_{n\ge0}\frac{(dz)^n}{n!}p_{d,m}(n)
 =A_b(z)^mC_b(z)^{d-m}.
\]
Coefficient extraction gives the exact identity
\begin{equation}
 p_{d,m}(n)=n!d^{-n}[z^n]A_b(z)^mC_b(z)^{d-m}.
                                                        \label{eq:exact-pdm}
\end{equation}

Stirling's formula then gives, for every $z>0$,
\[
 \frac1d\log p_{d,m}(n)
 \le\nu\log\nu-\nu+\theta\log A_b(z)+(1-\theta)\log C_b(z)
       -\nu\log z+O\!\left(\frac{\log d}{d}\right).
\]
Minimizing over $z$ gives the function $\Psi_b$, defined for
$\theta\in[0,1]$ and $\nu\ge0$ by\footnote{The convention is $0\log0=0$, and a term with zero coefficient is omitted
before evaluating a logarithm. The formula gives $\Psi_b(\theta,\nu)=-\infty$ when
$0\le\nu<\theta$, so its feasible count range is $\nu\ge\theta$.}
\begin{equation}
 \Psi_b(\theta,\nu)
 =\nu\log\nu-\nu+
   \inf_{z>0}\bigl(
      \theta\log A_b(z)+(1-\theta)\log C_b(z)-\nu\log z
   \bigr),                                              \label{eq:Psi}
\end{equation}
so we expect that roughly\footnote{Here, we restrict to admissible mutation counts, meaning
$p_{d,m}(n)>0$; for $b=2$, this requires $n\equiv m\pmod2$.}
\begin{equation}
   p_{d,m}(n)\approx e^{d\Psi_b(\theta,\nu)}.
   \label{eq:heuristic-pdm}
\end{equation}

\item \emph{The holding-time factor $h^{*(n+1)}(t)$.}
Define the Laplace transform
\begin{equation}
 H(\kappa)=\int_0^\infty e^{-\kappa r}h(r)\dd r
 =\int\frac{P(d\lambda)}{\kappa+1-\lambda},
 \qquad \kappa>\lambda_+-1.                               \label{eq:H}
\end{equation}
Laplace tilting gives the exact identity
\begin{align}
     h^{*(n+1)}(t)
 =e^{\kappa t}H(\kappa)^{n+1}
  \left(\frac{e^{-\kappa\,\cdot}h}{H(\kappa)}\right)^{*(n+1)}(t),
 \qquad \kappa>\lambda_+-1.\label{eq:exact-hstar}
\end{align}
The last factor is a probability density. 
{
Choosing a proper tilt shows that the leading exponent is therefore $\kappa u+\nu\log H(\kappa)$.\footnote{If an interior tilt satisfies
$
 -{H'(\kappa)}/{H(\kappa)}={t}/{(n+1)},
$
then the corresponding sum has mean $t$.  A local limit estimate then gives a density of order $d^{-1/2}$ at $t$, which does not affect the exponent.}
}
Minimizing over $\kappa$ gives the function $\chi$, defined for
$\nu\ge0$ and $u>0$ by
\begin{equation}
 \chi(\nu,u)
 =\inf_{\kappa>\lambda_+-1}\bigl(\kappa u+\nu\log H(\kappa)\bigr),\label{eq:chi}
\end{equation}
and we expect that
\begin{equation}
    h^{*(n+1)}(t)\approx e^{d\chi(\nu,u)}.
    \label{eq:heuristic-convolution}
\end{equation}
\end{itemize}

Finally, combining \eqref{eq:heuristic-pdm} and \eqref{eq:heuristic-convolution}
and summing the nonnegative contributions in \eqref{eq:freshK}
selects the largest combined exponent.  For $\theta\in[0,1]$ and $u>0$,
we therefore maximize over the mutation-count density $\nu$:\footnote{The deterministic functions $\Psi_b$, $\chi$, and $\widehat F$ depend only on $b$ and the
rate law in Assumption~\ref{ass:environment}; none is quenched or random.}
\begin{equation}
 \widehat F(\theta,u)
 =\sup_{\nu\ge0}\bigl(\Psi_b(\theta,\nu)+\chi(\nu,u)\bigr).       \label{eq:F}
\end{equation}
In other words,
\begin{equation}
    \widehat K_{d,m}(t)\approx e^{d\widehat F(\theta,u)}.
\end{equation}
  When $\widehat F<0$, this mean is exponentially
small; when $\widehat F>0$, it is exponentially large.  Its positive
zero therefore gives the leading passage-time density. 
Therefore, for every $\theta\in[0,1]$, we let $u_*(\theta)$ denote the unique positive zero
of $u\mapsto\widehat F(\theta,u)$.\footnote{Its existence, uniqueness, and uniform
stability are established in Lemma~\ref{lem:global-root-model-only} in
Section~\ref{sec:root-curve}.}  

\begin{frame}
\begin{heuristic}
\label{heur:variational-profile}
We have $\widehat\theta_{d,m}\approx du_*(m/d)+o(d)$.
\end{heuristic}
\end{frame}

Heuristic \ref{heur:variational-profile} also reveals an important philosophy we will follow for the proofs: when dealing with moment-type quantities like $\widehat K_{d,m}(t)$, the typical approach is to separate out a group of $n$ with major contribution, whose complement contributes only of a lower exponential order. Then, on this compact set of major $n$'s, estimates arising from saddle analysis and local limit theorems become more convenient to work with. In the following, we call the collection of such major $n$ the \textit{retained mutation counts}, and its complement the \textit{omitted mutation counts}.

\paragraph{The sparse-distance slope.}
We have derived the macroscopic location of $\widehat\theta_{d,m}$ (i.e., on the order of $d$), and it remains to understand the microscopic perturbations (i.e., on orders smaller than $d$). We demonstrate this using sparse distances $m=o(d)$ as $d\to\infty$. Put
$$
 G(\theta,u,\nu)=\Psi_b(\theta,\nu)+\chi(\nu,u).
$$
 For this heuristic only, assume for simplicity that
at $(0,u_*(0))$ there is a unique triple $(\nu,z,\kappa)$ attaining
the supremum in \eqref{eq:F} and the infima in \eqref{eq:Psi}
and \eqref{eq:chi}.
Stationarity in $z$ and $\nu$ in
\eqref{eq:Psi}, \eqref{eq:chi}, and \eqref{eq:F}, using $C_b'=A_b$, gives
\[
 \nu=z\frac{C_b'(z)}{C_b(z)}=z\frac{A_b(z)}{C_b(z)},\qquad
 0=\log\nu-\log z+\log H(\kappa),\qquad z=\nu H(\kappa).
\]
Thus, $A_b(z)/C_b(z)=1/H(\kappa)$.  Since $z$ and $\kappa$ minimize
their respective objectives in \eqref{eq:Psi} and \eqref{eq:chi},
\begin{equation}
 \partial_\theta G=\log\frac{A_b(z)}{C_b(z)}=-\log H(\kappa),
 \qquad \partial_uG=\kappa.
 \label{eq:heuristic-root-derivatives}
\end{equation}
To find the sparse-distance slope, formally set $u=u_*(0)+r\theta$
for small $\theta>0$.  Using stationarity in $\nu$, the chain rule and
\eqref{eq:heuristic-root-derivatives} give
\[
 \widehat F\bigl(\theta,u_*(0)+r\theta\bigr)
 =\theta\bigl(\kappa r-\log H(\kappa)\bigr)+o(\theta).
\]
Since $\kappa>0$, keeping the first-order exponent at zero gives the slope 
$r=\log H(\kappa)/\kappa$.  
\begin{frame}
\begin{heuristic}
\label{heur:sparse-slope}
For $m=o(d)$, we expect $\widehat\theta_{d,m}\approx u_*(0)d+(\log H(\kappa)/\kappa)m$.
\end{heuristic}
\end{frame}

\subsection{Statements of main results}

In this section, we present our main results. We start with a quantitative description of the deterministic center $\widehat\theta_{d,m}$, followed by the quenched and annealed tightness of the FPT.

\begin{definition}
\label{def:maximizing-saddle}
A \emph{maximizing saddle}\footnote{The term saddle refers to this max--min optimization: maximize over $\nu$
after minimizing over $(z,\kappa)$. } at $(\theta,u)$ is a triple $(\nu,z,\kappa)$
such that $\nu$ maximizes $G(\theta,u,\cdot)$, $z$ minimizes the spatial
objective in \eqref{eq:Psi}, and $\kappa$ minimizes the clock objective
in \eqref{eq:chi}, at this value of $\nu$.\footnote{A clock minimum at the
boundary is interpreted through its limiting value.} 
\end{definition}

Let $\mathcal S_*$ be the set of maximizing saddles at
$(0,u_*(0))$, and let
$$
 \mathcal K_*=\{\kappa:(\nu,z,\kappa)\in\mathcal S_*\}.
$$

\begin{theorem}
\label{thm:sparse-expansion}
Under Assumption~\ref{ass:environment}, there is a deterministic
$C_{\rm ctr}<\infty$ such that, for all sufficiently large $d$,
\begin{equation}
 \sup_{1\le m\le d}
 \left|\widehat\theta_{d,m}
      -d u_*\!\left(\frac md\right)\right|
 \le C_{\rm ctr}\log d.                              \label{eq:all-distance-center-log}
\end{equation}
In addition,
\begin{equation}
u_*(\theta)=u_*(0)+r_+\theta+o(\theta)
 \quad(\theta\downarrow0),
 \qquad\text{where }~~~
 r_+=\min_{\kappa\in\mathcal K_*}\frac{\log H(\kappa)}{\kappa}>0.                            \label{eq:sparseright}
\end{equation}
Consequently, if $1\le m_d\le d$ is a deterministic integer sequence with
$m_d=o(d)$, then
\begin{equation}
 \widehat\theta_{d,m_d}=u_*(0)d+r_+m_d+o(m_d)+O(\log d).
 \label{eq:sparse-center-expansion}
\end{equation}
\end{theorem}

Denote the survival event and its quenched probability by
\begin{equation}
 \mathsf S=\left\{\sum_vN_t(v)>0\text{ for every }t\ge0\right\},
 \qquad s^\omega(x)=\Pw_x(\mathsf S).
 \label{eq:survival-probability}
\end{equation}
 {When $s^\omega(x)=0$, we define
$\Pw_x(\cdot\mid\mathsf S):=\Pw_x(\cdot)$ by convention.}
With environment probability tending to one,
$s^\omega(x)>0$ at every vertex, by Lemma~\ref{lem:uniform-survival}.

\begin{theorem}
\label{thm:quenched-all-distance}
Under Assumption~\ref{ass:environment}, uniformly over
$1\le m\le d$,
\begin{equation*}
 \tau_{\mathbf{x}_m}^{\omega}
 =\widehat\theta_{d,m}
  +O_{\mathbb P_{\mathbf{x}_m}^{\omega}(\cdot\mid\mathsf S)}(1)
 \quad\text{in $\Penv$-probability}.
\end{equation*}
More precisely, for every $\varepsilon>0$,
\begin{equation}
 \lim_{A\to\infty}\limsup_{d\to\infty}
 \Penv\!\left(
  \sup_{1\le m\le d}
  \Pw_{\mathbf{x}_m}\!\left(
   \bigl|\tau_{\mathbf{x}_m}^{\omega}-\widehat\theta_{d,m}\bigr|>A
   \,\middle|\,\mathsf S\right)>\varepsilon
 \right)=0.                                           \label{eq:all-distance-window}
\end{equation}
The same center gives uniform annealed tightness conditional on survival:
\begin{equation}
 \lim_{A\to\infty}\limsup_{d\to\infty}
 \sup_{1\le m\le d}
 \frac{\Eenv\!\left[
  \Pw_{\mathbf{x}_m}\!\left(
   \bigl|\tau_{\mathbf{x}_m}^{\omega}-\widehat\theta_{d,m}\bigr|>A,
   \,\mathsf S\right)\right]}{\Eenv[s^\omega(\mathbf{x}_m)]}=0.
                                                        \label{eq:all-distance-annealed}
\end{equation}
\end{theorem}

\begin{remark}\label{rem:fluctuation}
Although $\widehat\theta_{d,m}$ is constructed from the independent-visit
{first moment $\widehat M_t(\mathbf x_m)$}, Theorem~\ref{thm:quenched-all-distance} shows that it also centers
the first-passage time in the static environment conditional on survival.
Moreover, the environment-dependent location of this conditional law fluctuates only
on the order-one scale about a deterministic center.  This is 
different from the one-dimensional static random-environment models, where the
medium displaces the front on the diffusive $\sqrt n$ scale about its
deterministic linear order and an environment-dependent center is required
to recover order-one tightness; see~\cite{CernyDrewitz2020,KriechbaumSpace2024}. 
\end{remark}

\subsection{Applications}

\subsubsection{Comparative statics}

As a first application of our main results, we investigate comparative statics: how does the FPT change if we modify the underlying parameters? Here, we consider the deterministic mutation rate and the branching rate distribution. 

\paragraph{Mutation rate.}
We study how the FPT asymptotics change when we change the mutation
rate to $\varrho>0$ (which is a constant not depending on $d$) while keeping the branching rate law $\Lambda_v$ fixed.
Corollary~\ref{thm:mutation-rate}
shows that, in the sparse-distance regime, an increase in $\varrho$ leads to an increase in the FPT of order $d$. This parallels 
\cite[Corollary~3]{BZ} 
for the homogeneous BRW model (i.e., with constant branching rate). Since our main results assume a unit mutation rate, we need to rescale in order to apply these results. For
rates satisfying
\begin{equation}
 \varrho>\lambda_+,\qquad
 \varrho\int\frac{P(d\lambda)}{\varrho-\lambda}>1,
 \label{eq:mutation-supercritical}
\end{equation}
let $\tau_x^{\omega,\varrho}$ denote the corresponding first-passage time
and $u_*^{[\varrho]}$ the root $u_*$ for the rate law
$\mathcal L(\Lambda_{\mathbf0}/\varrho)$.  Dividing
$u_*^{[\varrho]}$ by $\varrho$ gives the leading passage-time coefficient (at order $d$) 
in physical units.

\begin{corollary}
\label{thm:mutation-rate}
Under Assumption~\ref{ass:environment}, let $I$ be a compact interval of
mutation rates satisfying \eqref{eq:mutation-supercritical}.
For any fixed $\varrho_1<\varrho_2$ in $I$, there is
$\varepsilon=\varepsilon(\varrho_1,\varrho_2)\in(0,1]$ such that
$$
 \frac{u_*^{[\varrho_2]}(\theta)}{\varrho_2}
 >\frac{u_*^{[\varrho_1]}(\theta)}{\varrho_1},
 \qquad 0\le\theta\le\varepsilon.
$$
Thus, at these distances the larger mutation rate has a larger leading
passage-time coefficient (at order $d$).
\end{corollary}

\paragraph{Reproduction rate.}
{By a standard coupling argument, stochastically larger net growth-rate laws give
stochastically smaller unconditioned first-passage times under the annealed law.}
Thus, the nontrivial comparative statics is with respect to the convex order. We write $P_1\le_{\rm cx}P_2$ if the two laws have the same finite mean
and $\int f(\lambda)P_1(d\lambda)\le\int f(\lambda)P_2(d\lambda)$
for every convex function $f$.  In the following, we use subscripts $1,2$ for the
deterministic quantities associated with the two reproduction rate distributions $P_1,P_2$.

\begin{corollary}
\label{cor:rate-convex-order}
Suppose that the mutation walk is fixed with unit rate. Let $P_1,P_2$ be two fixed rate
laws satisfying Assumption~\ref{ass:environment}.  If
$P_1\le_{\rm cx}P_2$, then
\[
 \widehat\theta_{2,d,m}\le\widehat\theta_{1,d,m}
 \quad(1\le m\le d),\qquad
 u_{*,2}(\theta)\le u_{*,1}(\theta)
 \quad(0\le\theta\le1).
\]
If $P_1\ne P_2$, the root inequality is strict.  Thus, for deterministic
integers $1\le m_d\le d$, the leading passage-time
coefficient $u_{*,i}(m_d/d)$ is strictly smaller for $i=2$, conditional on survival
under each law.
\end{corollary}

\subsubsection{Fisher--KPP level times}
\label{sec:kpp}

Throughout this section, assume in addition that $\Lambda_v\ge0$, i.e., there are no deaths.  
On $\Hc_d^{(b)}$, the Fisher--KPP equation combines mutation with
logistic growth at rate $\Lambda_x$:
$$
 \partial_t w^\omega(t,x)
 =(P-I)w^\omega(t,x)
   +\Lambda_xw^\omega(t,x)\bigl(1-w^\omega(t,x)\bigr),
 \qquad w^\omega(0,x)=\one_{\{x=\mathbf0\}}.
$$
Its linearization at zero is the parabolic Anderson equation:
$$
 \partial_t v^\omega(t,x)
 =(P-I)v^\omega(t,x)+\Lambda_xv^\omega(t,x),
 \qquad v^\omega(0,x)=\one_{\{x=\mathbf0\}}.
$$
See~\cite[Section~2.3]{KoenigSurvey2021} for its branching interpretation
and~\cite{AvenaGunHesse2020} for the model on the binary hypercube.
First-event conditioning and the first-moment formula give
\[
 w^\omega(t,x)=\Pw_x\bigl(N_t(\mathbf0)>0\bigr),\qquad
 v^\omega(t,x)=K_t^\omega(x,\mathbf0).
\]
This is the usual branching representation of Fisher--KPP; see, for
example, \cite[Proposition~2.1]{DrewitzSchmitz2022}
for its random-medium form.  The branching representation gives
$w^\omega\le v^\omega$ in every environment.  
For a level $0<\alpha<1$ and either solution $f$, write
\[
 \mathcal T_\alpha[f](x)=\inf\bigl\{t\ge0:f(t,x)\ge\alpha\bigr\}.
\]
These times are finite whenever $\sum_v\Lambda_v>0$, by
Proposition~\ref{prop:finite-state-tools} below. 

\begin{corollary}
\label{cor:kpp-arrival}
Under Assumption~\ref{ass:environment} and $P([0,\lambda_+])=1$, for every fixed
$\alpha\in(0,1)$,
\[
 \lim_{A\to\infty}\limsup_{d\to\infty}
 \Penv\!\left(
  \max_{1\le m\le d}\max_{f\in\{w^\omega,v^\omega\}}
  \left|\mathcal T_\alpha[f](\mathbf x_m)
       -\widehat\theta_{d,m}\right|>A
 \right)=0.
\]
In addition,
$$
 \lim_{A\to\infty}\limsup_{d\to\infty}
 \Penv\!\left(
 \begin{gathered}
  \sum_v\Lambda_v>0,\quad
  \mathcal T_\alpha[w^\omega](\mathbf x_m)
       >\mathcal T_\alpha[v^\omega](\mathbf x_m)+A\\
  \text{for some }1\le m\le d
 \end{gathered}
 \right)=0.
$$
\end{corollary}

In other words, the nonlinear delay remains bounded
in probability, not merely negligible on the order-$d$ scale. 
In one-dimensional random media, \cite[Theorem~1.5]{DrewitzSchmitz2022}
gives an $O(\log t)$ bound on the distance between nonlinear and linear
spatial fronts; a related bound for branching random walks appears
in~\cite[Theorem~2.4]{CernyDrewitz2020}.  The difference from our result is that on the Hamming cube, we compare times to reach
a fixed level at prescribed vertices as $d\to\infty$, rather than
spatial front positions as $t\to\infty$.

\subsection{Proof strategy and technical highlights}\label{sec:proofstrategy}

In this section, we start from a high-level concise outline of the proofs of main theorems, which underlines the key technical challenges. Then, we illustrate these challenges and their underlying intuition in further detail.

\subsubsection{Proof outline}

\paragraph{Theorem~\ref{thm:sparse-expansion}: locating the deterministic center $\widehat\theta_{d,m}$.}
Two steps from the heuristics in Section~\ref{sec:heuristics} remain to be made rigorous. 
\begin{itemize}
    \item {First, Heuristic~\ref{heur:variational-profile} identifies only the exponential rate, and hence the linear order in $d$ of the FPT. To make the exponential approximation more precise, we need \underline{local limit estimates on discrete bridges}, which identify the convolution structures within \eqref{eq:exact-pdm} and \eqref{eq:exact-hstar} with a sum of independent random variables after a proper tilting. This gives}
{
\[
 \log\widehat K_{d,m}(t)
 =d\widehat F\left(\frac{m}{d},\frac{t}{d}\right)+O(\log d)
 \quad\text{near }t=du_*\left(\frac{m}{d}\right).
\]
}
Since the exponent increases at a rate bounded away from zero there,
this locates $\widehat\theta_{d,m}$ to within $O(\log d)$
(Proposition~\ref{prop:global-oneobject} and
Lemma~\ref{lem:independent-center-location}), which gives \eqref{eq:all-distance-center-log}.

\item 
Second, {the calculation leading to Heuristic~\ref{heur:sparse-slope}
assumes} a unique maximizing saddle, but the sparse-distance calculation must allow \underline{several maximizing
saddles}; comparing their first-order responses gives the minimum
defining $r_+$ in \eqref{eq:sparseright}
(Lemma~\ref{lem:sparsederivative}).
\end{itemize}

\paragraph{Theorem~\ref{thm:quenched-all-distance}: quenched and annealed tightness of the FPT.}
There are three steps.
\begin{enumerate}
\item \emph{Compare the static and independent-visit models.}
{Heuristic~\ref{heur:static-independent}} indicates that the independent-visit model serves as a technically tractable proxy for the static model we are interested in.\footnote{Recall that each repeated visit resamples a reproduction rate for the independent-visit model but not for the static model.}  
We \underline{control their difference quantitatively} by counting revisits of the paths and their holding
times.  {Roughly speaking, this gives the conditional-mean comparison
\[
 \Eenv[K_t^\omega(\mathbf x_m,\mathbf0)
       \mid\Lambda_{\mathbf x_m},\Lambda_{\mathbf0}]
 \asymp_L\widehat K_{d,m}(t)  \qquad\text{if }\Lambda_{\mathbf x_m},\Lambda_{\mathbf0}\ge-L,
\]
and bounds the fluctuations of $K_t^\omega(\mathbf x_m,\mathbf0)
       \mid\Lambda_{\mathbf x_m},\Lambda_{\mathbf0}$ about this mean
(Lemma~\ref{lem:abstract-environment-concentration}).}

\item \emph{Obtain a positive hitting probability via fractional Paley--Zygmund.}
{ The goal is to use the fractional Paley--Zygmund inequality to give a lower bound on the probability of the first-passage event. 
For a small fixed $q\in(0,1)$, we need an upper bound on the $(1+q)$-moment
of the target count.
For the independent-visit model, we write $\mathbf K_r^P(x;(v,i))P(di)$ for
the expected number of particles at $v$ at time $r$ with current rate in
$di$, starting from one particle at $x$ with a fresh rate of law $P$.
Write $\mathbf k_s((v,i),\mathbf0)$ for the expected number of particles
at $\mathbf0$ at time $s$, starting from one particle at $v$ with initial
rate $i$.
Consider a time $t$ near the deterministic center $\widehat\theta_{d,m}$. For $r\in[0,t]$, define the fractional
branching sum by\footnote{The precise form of \eqref{eq:B} is motivated by an upper bound of fractional moments of particle counts; see \eqref{eq:moment-connection}. }
}
\begin{align}
     \widehat{\mathcal B}_{t,r}^{(1+q)}(x)
 =\sum_v\int \mathbf K_r^P(x;(v,i))\,i_+\,
       \mathbf k_{t-r}((v,i),\mathbf0)^{1+q}\,P(di).\label{eq:B}
\end{align}
{Theorem~\ref{thm:globalfractionaldiagram} gives the \underline{upper bound for the fractional branching sum}
$$
 \int_0^t\widehat{\mathcal B}_{t,r}^{(1+q)}(\mathbf x_m)\dd r
 \ll\widehat K_{d,m}(t)+\widehat K_{d,m}(t)^{1+q}.
$$
{\underline{Transferring this bound to the static model and averaging over the
environment} gives an annealed $(1+q)$-moment estimate.  Fractional Paley--Zygmund
then bounds the environmental mean of {\color{black}the probability of having a particle at the target at time $t$}
away from zero.  We then prove an annealed concentration of that probability directly by showing
that resampling one rate can matter only through a successful ancestral
path that visits the resampled vertex
(Propositions~\ref{prop:quenched-fractional-diagram}
and~\ref{prop:quenched-barrier}).}
}

\item \emph{Bootstrap and control early hits.}
Conditional on survival, we obtain many independent descendants 
within a bounded $O(1)$ time and $O(1)$ distance from the origin 
(Lemma~\ref{lem:survival-preliminary-family}).
Their positive success probabilities then \underline{lift the positive hitting probability from fractional Paley--Zygmund} to a probability close enough to one, leading to an upper bound of the FPT.
For the lower bound, we apply the first moment method and control the expectation of the total occupation before the deterministic center
(Proposition~\ref{prop:global-center}).
\end{enumerate}

\subsubsection{Main technical challenges}

{The estimates below concern $1\le m\le d$ and
$t=du_*(m/d)+O(\log d)$; implicit constants may depend on the fixed constant here.}

\paragraph{Discrete bridges and local limit theorems.}
Although our process is continuous-time, we need to condition on the number of mutations $n$ in \eqref{eq:freshK} and analyze the discrete random walk bridges, for two purposes.
\begin{itemize}
\item It gives matching estimates for the two factors in \eqref{eq:freshK}.
With tilts chosen to match $n$ and $t$, their normalized forms satisfy\footnote{{These estimates hold uniformly for retained mutation counts. The exponential-rate calculation in Heuristic~\ref{heur:variational-profile} identifies these counts; Proposition~\ref{prop:global-oneobject} proves that the omitted counts contribute an exponentially small fraction of the first moment.}}
\[
 \frac{(dz)^np_{d,m}(n)}{n!A_b(z)^mC_b(z)^{d-m}}\asymp d^{-1/2},
 \qquad
 \frac{e^{-\kappa t}h^{*(n+1)}(t)}{H(\kappa)^{n+1}}\asymp d^{-1/2},
\]
which follow from local limit theorems 
(Lemmas~\ref{lem:spatialLCLT} and~\ref{lem:clockLCLT}).

\item It decouples the time and the geometry of the mutation processes.
Let $(S_0,\ldots,S_n)$ denote an $n$-step bridge, which is a mutation path 
conditioned on its starting and ending vertices.
When quantifying the difference between the static and independent-visit models,
it is essential to control repeated visits to vertices. For discrete-time
random walks on the Hamming cube this is easier, because conditioned on
the total number of steps, the bridge path law does not depend on the
random environment. 
For example, we establish
\begin{align}
    \Pp\bigl(S_{k+1}=w\mid S_0,\ldots,S_k,S_n=\mathbf0\bigr)
 \ll\frac1{n-k}.\label{eq:bridge1}
\end{align}
for $n\asymp d$ 
(Lemma~\ref{lem:edgespread}). This is crucial for controlling intersection probabilities of discrete paths of order $O(d^{-1})$.
\end{itemize}

\paragraph{Several maximizing saddles.}
As indicated in the paragraph leading to Heuristic \ref{heur:variational-profile}, the main idea of the saddle analysis is to identify the {retained mutation counts}, depending on the distance $m$ and time $t$, through a rough and macroscopic-level first-moment computation.\footnote{However, note that this does not reveal the geometry of the mutations even if conditioned on the environment. The computation is performed for the independent-visit model, where all routes on the hypercube {with the same mutation count} contribute the same weight for the first moment.}

Near $u=u_*(\theta)$, we retain the admissible mutation counts $n$ whose
exponent $G(\theta,u,n/d)$ differs from $\widehat F(\theta,u)$ by at most
a sufficiently small fixed amount. {The other counts contribute at most
$e^{-cd}\widehat K_{d,m}(t)$
(Proposition~\ref{prop:global-oneobject}).} With $\nu=n/d$, these counts and
their tilts satisfy, for fixed $c,C>0$,
\[
 \nu-\theta\ge c,\qquad \nu\le C,\qquad c\le z,\kappa\le C.
\]
Thus, $n-m\ge cd$, and both tilts are bounded above and away from zero
(Lemmas~\ref{lem:variational-regularity} and~\ref{lem:global-root-model-only}).\footnote{As explained above, these bounds keep the local limit estimates uniform:
the tilted probabilities and densities at their means are comparable to
$d^{-1/2}$, with constants independent of the retained count and $m$.}
As the exponent approaches its maximum and $u$ approaches $u_*(\theta)$,
these parameters approach the set of maximizing saddles along that curve,
which may not be necessarily a single saddle (Lemma~\ref{lem:root-active-stability}).

{For the sparse-distance expansion at $m=o(d)$, we apply Danskin's theorem
\cite[Theorem~4.1]{BonnansShapiro1998}, which intuitively identifies the first-order response of
a maximum as the largest first-order response among its current maximizers.}
Combining this with the derivatives in
\eqref{eq:heuristic-root-derivatives} gives, for fixed $r$,
\[
 \widehat F\bigl(\theta,u_*(0)+r\theta\bigr)
 =\theta\max_{\kappa\in\mathcal K_*}
       \bigl(r\kappa-\log H(\kappa)\bigr)+o(\theta).
\]
Consequently, the first saddle to recover exponent zero determines the new root:
\[
 \max_{\kappa\in\mathcal K_*}
       \bigl(r_+\kappa-\log H(\kappa)\bigr)=0
 \quad\Longleftrightarrow\quad
 r_+=\min_{\kappa\in\mathcal K_*}\frac{\log H(\kappa)}{\kappa},
\]
which is the goal of 
(Lemma~\ref{lem:sparsederivative}).

\paragraph{Quantitative approximation of the static model using the independent-visit model.}
The two models differ if a vertex is visited repeatedly.
{Consider one
particle at rate $\lambda$ and remove each descendant when it mutates.
The expected number remaining at time $r$ is
$e^{\lambda r}e^{-r}=e^{-(1-\lambda)r}$, accounting for branching,
death, and the absence of mutation along the ancestral path.
Averaging over the rate gives
}
\[
 h(r)=\Eenv\bigl[e^{-(1-\Lambda_{\mathbf0})r}\bigr].
\]
In other words, $h(r)$ is the expected size of the family that has not yet mutated,
averaged over the rate at its starting vertex. 
Two visits to $v$ with holding times $r_1,r_2$ use the same rate in
the static model, so their environmental mean is
$$
 \Eenv\!\left[e^{-(1-\Lambda_v)r_1}e^{-(1-\Lambda_v)r_2}\right]
 =h(r_1+r_2).
$$
In the independent-visit model, the rates are independent, giving
$h(r_1)h(r_2)$ instead. 
With the normalization
$$
 Y_v(r)=\frac{e^{-(1-\Lambda_v)r}}{h(r)},
 \qquad \Eenv[Y_v(r)]=1,
$$
their ratio admits the probabilistic representation
$$
  \frac{h(r_1+r_2)}{h(r_1)h(r_2)}=\Eenv\bigl[Y_v(r_1)Y_v(r_2)\bigr].
$$
To bound alike ratios, fix $\varepsilon>0$ and put
$p_\varepsilon=\Penv(\Lambda_{\mathbf0}\ge\lambda_+-\varepsilon)>0$.
By definition,
\begin{align}
    h(r)\ge p_\varepsilon e^{-(1-\lambda_++\varepsilon)r},
 \qquad Y_v(r)\le p_\varepsilon^{-1}e^{\varepsilon r},
 \qquad r\ge0.\label{eq:Hh}
\end{align}
If $v$ is visited $k$ times, we bound all but the first factor and use
$\Eenv[Y_v(r_1)]=1$ to obtain the trivial bound
$$
 \Eenv\!\left[\prod_{\ell=1}^kY_v(r_\ell)\right]
 =\frac{h(r_1+\cdots+r_k)}{\prod_{\ell=1}^kh(r_\ell)}
 \Eenv[Y_v(r_1)]\le p_\varepsilon^{-(k-1)}
       \exp\!\left(\varepsilon\sum_{\ell=2}^kr_\ell\right).
$$
Thus, intuitively, each repeated visit of duration $r$ after the first contributes
at most $p_\varepsilon^{-1}e^{\varepsilon r}$.

In general, for a path $\pi=(v_0,\ldots,v_n)$ with holding times $(r_0,\ldots,r_n)$,
the independence of the environment yields a total ratio of
\[
 \Eenv\!\left[\prod_{j=0}^nY_{v_j}(r_j)\right]
 =\prod_{v:\,v\in\pi}
   \frac{h\!\left(\sum_{j:v_j=v}r_j\right)}
        {\prod_{j:v_j=v}h(r_j)}.
\]
Denote the local times by $L_v(\pi)=\#\{j:v_j=v\}$, with the corresponding holding times
$r_{v,1},\ldots,r_{v,L_v(\pi)}$ in order.
{For notational convenience, let $\mathbb Q_{d,m,t}$ denote
the joint law of a path and its holding times obtained by normalizing
their first-moment weights in \eqref{eq:freshK} by $\widehat K_{d,m}(t)$. 
More precisely, the count $n$ has probability
$p_{d,m}(n)h^{*(n+1)}(t)/\widehat K_{d,m}(t)$.
Given $n$, the path is an ordinary $n$-step mutation bridge from
$\mathbf x_m$ to $\mathbf0$, independent of the holding times, whose
density on $r_0+\cdots+r_n=t$ is proportional to $\prod_{j=0}^nh(r_j)$.}
\begingroup
Our next goal is to show that
\begin{equation}
 \frac{\Eenv[K_t^\omega(\mathbf x_m,\mathbf0)]}{\widehat K_{d,m}(t)}
 =\E_{\mathbb Q_{d,m,t}}\!\left[
      \Eenv\!\left[\prod_{j=0}^nY_{v_j}(r_j)\right]\right]\ll1.
 \label{eq:ts1}
\end{equation}
For a retained mutation count $n$, the holding times have density
$f_\kappa(r)=e^{-\kappa r}h(r)/H(\kappa)$, conditioned on their sum
being $t$, with $-H'(\kappa)/H(\kappa)=t/(n+1)$.
Before conditioning, integration of one holding ratio gives\footnote{Here we have omitted conditioning on the sum of holding times to equal to $t$, but it affects only a constant factor due to local limit theorems.}
\begin{align}
\int_0^\infty f_\kappa(r)Y_v(r)\,dr
&=\int_0^\infty
\frac{e^{-\kappa r}h(r)}{H(\kappa)}
\frac{e^{-(1-\Lambda_v)r}}{h(r)}\,dr=\frac1{H(\kappa)}
\int_0^\infty e^{-(\kappa+1-\Lambda_v)r}\,dr=\frac1{H(\kappa)(\kappa+1-\Lambda_v)}.\label{eq:1.5}
\end{align}
This has environmental mean one and is uniformly bounded for the
retained tilts.  Thus, each revisit costs at most a bounded constant factor
after integration. 
\endgroup 
This motivates us to to control exponential moments of $\sum_v(L_v(\pi)-1)_+$. The key argument consists of the following two observations.
\begin{itemize}
\item \emph{A prescribed next neighbour is rarely hit.}
{For $0\le k<n/2$,\footnote{To estimate the second half $n/2<k\leq n$, we time-reverse the random walk. } the bridge estimate \eqref{eq:bridge1} gives}
\begin{align}
    \Pp\bigl(S_{k+1}\in\{S_0,\ldots,S_k\}
       \mid S_0,\ldots,S_k,S_n=\mathbf0\bigr)
 \ll\frac1d\sum_{j=0}^k\one_{\{S_j\sim S_k\}};\label{eq:return}
\end{align}
see (Lemma~\ref{lem:edgespread}). 
Thus, the return probability is controlled by the number of earlier
occurrences $S_j$ that are adjacent to the current vertex $S_k$.

\item \emph{Neighbourhood occupation is small.}
Two distinct vertices of the Hamming cube $\Hc_d^{(b)}$ have at most
$\max\{2,b-2\}$ common neighbours.  From a vertex distinct from $v$, a move
into the neighbour set of $v$ therefore has probability $O(d^{-1})$. Over $O(d)$ steps, this gives
$$
 \E_{\mathbf x_m}\!\left[
 \exp\!\left(a\sum_{j=0}^n\one_{\{S_j\sim v\}}\right)
 \,\middle|\,S_n=\mathbf0\right]\ll_a1
 \qquad\text{for every fixed }a>0;
$$
see Lemma~\ref{lem:bridgestar}.
\end{itemize}
{
Conditional on $S_k=u$ and $S_n=\mathbf0$, the prefix
$(S_0,\ldots,S_k)$ is a $k$-step bridge from $\mathbf x_m$ to $u$.
The neighbourhood bound, uniformly in $u$, therefore gives
$$
 \E_{\mathbf x_m}\!\left[
 \exp\!\left(a\sum_{j=0}^k\one_{\{S_j\sim S_k\}}\right)
 \,\middle|\,S_n=\mathbf0\right]\ll_a1.
$$
Combining this with \eqref{eq:return} controls repeats within
each half; the same argument controls visits between the two halves.
Thus, for fixed $a>0$,}
$$
 \E_{\mathbf x_m}\!\left[
 \exp\!\left(a\sum_v(L_v(\pi)-1)_+\right)
 \,\middle|\,S_n=\mathbf0\right]\ll_a1,
 \qquad \pi=(S_0,\ldots,S_n);
$$
see Proposition~\ref{prop:collisionmgf}.
{Together with the clock integration argument above, this proves \eqref{eq:ts1}.}

{Write $\widehat K_t(x,y)=\widehat K_{d,d_{\mathrm H}(x,y)}(t)$.}
Once \eqref{eq:ts1} is established, we need to show that for $y$ close to $\mathbf{x}_m$,
\begin{align*}
    \frac{K_t^\omega(y,\mathbf0)}{\widehat K_{t}(y,\mathbf0)}\qquad\text{is concentrated near}\qquad\frac{{\Eenv[K_t^\omega(y,\mathbf0)\mid\Lambda_y,\Lambda_{\mathbf0}]}}{\widehat K_{t}(y,\mathbf0)}.
\end{align*}
{
To see this concentration, we apply a martingale resampling argument.
Fix the endpoint rates $\Lambda_y,\Lambda_{\mathbf0}$ and replace one
other rate $\Lambda_z$ by an independent copy.  In this way, only paths visiting $z$
have their weights changed. Revealing the rates one at a time yields a Doob martingale, and a standard martingale inequality bounds the
total fluctuation by combining the squared sizes of these changes.
  A key step is 
$$
 \sum_{z\notin\{y,\mathbf0\}}\Pp(L_z(\pi)>0)^2
 =\E\!\left[
   \sum_{z\notin\{y,\mathbf0\}}
   \one_{\{L_z(\pi)>0,\ L_z(\pi')>0\}}
 \right]\ll d^{-1},
$$
where $\pi,\pi'$ be independent paths sampled from the normalized
independent-visit weights from $y$ to $\mathbf0$. 
To see the $d^{-1}$ intuitively, note that for the retained
counts $n\asymp d$, \eqref{eq:bridge1} bounds the probability of choosing
any prescribed first neighbour by $O(d^{-1})$.\footnote{Two independent bridges
therefore choose the same first neighbour with probability $O(d^{-1})$;
time reversal gives the same explanation near the target.  Farther
from the endpoints, their positions spread out, and the sum over all
possible meetings remains $O(d^{-1})$
(Proposition~\ref{prop:range}).}    Finally, the holding-time integration
in \eqref{eq:1.5} and the repeated-visit bound just proved control the
weights of visiting paths, leading to the desired concentration (Lemma~\ref{lem:abstract-environment-concentration}).}

As a consequence, near the deterministic center, the conditional mean satisfies
\begin{align}
    \frac{\Eenv[K_t^\omega(y,\mathbf0)\mid\Lambda_y,\Lambda_{\mathbf0}]}
{\widehat K_t(y,\mathbf0)}
\asymp_L1
\qquad\text{if }\Lambda_y,\Lambda_{\mathbf0}\ge-L.
\label{eq:ts2}
\end{align}
Moreover, the upper bound part does not rely on $\Lambda_y,\Lambda_{\mathbf0}\ge-L$.
% Consequently, with high environment probability,
% \[
% K_t^\omega(y,\mathbf0)\asymp_L\widehat K_t(y,\mathbf0).
% \]

\paragraph{The fractional branching sum for the independent-visit model.}
Recall that we need to control \eqref{eq:B} to apply fractional Paley--Zygmund. For a split at time $r$, call the ancestral path before
the split the \emph{trunk} and each descendant path an \emph{arm}; see Figure \ref{fig:sketch-shared-rate} below.
The bound in Theorem~\ref{thm:globalfractionaldiagram} takes the form
\begin{align}
    \widehat{\mathcal B}_{t,r}^{(1+q)}(\mathbf x_m)
 \ll_q
 \widehat K_{d,m}(t)^{1+q}e^{-c_qr}
 +\widehat K_{d,m}(t)e^{-c_q(t-r)}
 \qquad(0\le r\le t).\label{eq:B2}
\end{align}
Conditioning on the location $v$ and current rate $i$ at the split gives
$$
 \sum_v\int\mathbf K_r^P(\mathbf x_m;(v,i))
        \mathbf k_{t-r}((v,i),\mathbf0)P(di)=\widehat K_{d,m}(t).
$$
% Dividing these weights by $\widehat K_{d,m}(t)$ gives a probability
% distribution on $(v,i)$.  This is the distribution used below when
% we speak of typical or exceptional splits.

The main difficulty is to establish the bound \eqref{eq:B2} uniformly over $1\le m\le d$. 
We apply the decomposition technique of the branching-time integral in
\cite[Appendix~C.1.2, proof of Proposition~24]{BZ}.
For a sufficiently small fixed $\delta>0$, we treat three ranges separately.
\begin{itemize}
\item \emph{Early splits: $0\le r\le\delta t$.} Our goal is to control the quantity
\[\frac{\widehat{\mathcal B}_{t,r}^{(1+q)}(\mathbf x_m)}
     {\widehat K_{d,m}(t)^{1+q}}
=
\sum_v\int
\mathbf K_r^P(\mathbf x_m;(v,i))\,i_+
\left(
\frac{\mathbf k_{t-r}((v,i),\mathbf0)}
     {\widehat K_{d,m}(t)}
\right)^{1+q}P(di),\]
for $r$ small. Recall that $\widehat K_{d,m}(t)$ and $\mathbf k_{t-r}((v,i),\mathbf0)$ counts the respective first moments, where compared to the former, the count for $\mathbf k_{t-r}((v,i),\mathbf0)$ has $r$ less time and a new distance of $d_{\mathrm H}(v,\mathbf0)$ instead of $m$. Next, we {restrict the reference counts to the retained set} and for each such $n$, pick $z_n$ and $\kappa_n$ as the optimizers in \eqref{eq:Psi} and \eqref{eq:chi}.
{Lemma~\ref{lem:global-retained-arm} makes the calculation
in Heuristic~\ref{heur:variational-profile} precise and retains the
dependence on the initial rate $i$:
\[
\frac{\mathbf k_{t-r}((v,i),\mathbf0)}
     {\widehat K_{d,m}(t)}
\ll
\sum_{n\text{ retained}}
\frac{\pi_n e^{-\kappa_n r}}{\kappa_n+1-i}
\left(\frac{A_b(z_n)}{C_b(z_n)}\right)^{d_{\mathrm H}(v,\mathbf0)-m}
+\text{lower-order contribution},
\]
where using Jensen's inequality with the probability weights
\[
\pi_n=\frac{p_{d,m}(n)h^{*(n+1)}(t)}
           {\widehat K_{d,m}(t)},
\]
we obtain
\[
\begin{aligned}
&\frac{\widehat{\mathcal B}_{t,r}^{(1+q)}(\mathbf x_m)}
     {\widehat K_{d,m}(t)^{1+q}}\\
&\ll
\sum_{n\text{ retained}}\pi_n
\Bigg(
e^{-(1+q)\kappa_n r}
\sum_v\int
\mathbf K_r^P(\mathbf x_m;(v,i))
\frac{i_+}{(\kappa_n+1-i)^{1+q}}
\left(\frac{A_b(z_n)}{C_b(z_n)}\right)^{(1+q)(d_{\mathrm H}(v,\mathbf0)-m)}P(di)
\Bigg)\\
&\qquad+\text{lower-order contribution}.
\end{aligned}
\]
Theorem~\ref{thm:globalfractionaldiagram} then bounds the parenthesized
expression by $\ll_q e^{-c_qr}$, uniformly in $n$, through delicate saddle analysis. 
Intuitively, the loss
from having $r$ less time overwhelms the gain from getting closer to
the target.}

\item \emph{Middle splits: $\delta t\le r\le(1-\delta)t$.} Consider $t=\widehat\theta_{d,m}+O(\log d)$. 
In this case, $r,t-r\asymp d$. 
We first group $v$ by their distances from both endpoints, so that
{the expected target count is
constant within each group.}
Depending on $r$, we may classify the groups as typical or exceptional, where at a given time \(r\), a typical $v$ has about as many coordinates matching the origin and the target as one would expect from a path with the given mutation count and conditioned to reach the target at time \(t\);\footnote{This is similar to Heuristic~\ref{heur:variational-profile} and will be 
made precise in Lemma~\ref{lem:setvalued-fractional-gaps}. } the exceptional groups have an exponentially small contribution.

\begin{itemize}
\item {\emph{Typical groups.}}
One can show that for a typical group $G$, the expected number of particles located in \(G\) at time \(r\) satisfies $\sum_{w\in G}\widehat K_r(\mathbf{x}_m,w)\geq e^{cd}$ for some small $c>0$.
We claim that 
\begin{align}
   \underbrace{\left(\sum_{w\in G}\widehat K_r(\mathbf{x}_m,w)\right)}_{\text{expected number arriving in group }G}\times
\underbrace{\widehat K_{t-r}(v,\mathbf0)}_{\text{expected target descendants from one particle}}
\le \widehat K_t(\mathbf{x}_m,\mathbf0).\label{eq:see}
\end{align}
Indeed, note that both $\mathbf K_r^P(\mathbf x_m;(v,i))$ and
$\mathbf k_{t-r}((v,i),\mathbf0)$ increase with $i$, so
$$
 \widehat K_r(\mathbf x_m,v)\widehat K_{t-r}(v,\mathbf0)
 \le\int\mathbf K_r^P(\mathbf x_m;(v,i))
         \mathbf k_{t-r}((v,i),\mathbf0)P(di).
$$
Summing the right-hand side over $v$ gives $\widehat K_t(\mathbf{x}_m,\mathbf0)$;
keeping only the chosen group proves \eqref{eq:see}.
By \eqref{eq:see}, {this expected target count} is $\widehat K_{t-r}(v,\mathbf0)\ll d^Ce^{-cd}$.
By \eqref{eq:Hh},
$$
 \mathbf k_{t-r}((v,i),\mathbf0)
 \le p_\varepsilon^{-1}e^{\varepsilon(t-r)}
      \widehat K_{t-r}(v,\mathbf0),
 \qquad i\le\lambda_+.
$$
Choosing $\varepsilon$ sufficiently small preserves the exponential decay, i.e., $\mathbf k_{t-r}((v,i),\mathbf0)
 \le e^{-cd}$.

\item {\emph{Exceptional groups.}} 
We apply the trivial bound 
 $\mathbf k_{t-r}((v,i),\mathbf0)\leq e^{\lambda_+(t-r)}\le e^{Cd}$.
\end{itemize}

Applying the above bounds on $\mathbf k_{t-r}((v,i),\mathbf0)^q$ gives
$$
 \frac{\widehat{\mathcal B}_{t,r}^{(1+q)}(\mathbf x_m)}
      {\widehat K_{d,m}(t)}
 \ll
 \underbrace{e^{-cqd}}_{\text{typical splits}}
 +\underbrace{e^{-cd+qCd}}_{\text{exceptional splits}}.
$$
The constants are uniform in $m$, so picking $q$ small enough 
makes both terms decay uniformly in $m$.

\item \emph{Late splits: $(1-\delta)t\le r\le t$.}
Put $s=t-r$.  The population growth is bounded by $e^{\lambda_+s}$, while  the mutation-walk probability of being at a
specified vertex is bounded by $e^{-s+Cs^2/d}$.  Since $\lambda_+<1$, for $\delta$ small enough,
$$
 \sup_{v,i}\mathbf k_s((v,i),\mathbf0)
 \le\exp\!\left(-(1-\lambda_+)s+C\frac{s^2}{d}\right)
 \le e^{-(1-\lambda_+)s/2}.
$$
 Bounding the extra $q$-th power of $\mathbf k_s((v,i),\mathbf0)$ by the supremum
and the splitting rate by $\lambda_+$ yields 
$\widehat{\mathcal B}_{t,t-s}^{(1+q)}(\mathbf x_m)
\ll\widehat K_{d,m}(t)e^{-q(1-\lambda_+)s/2}$, uniformly in $m$. Note that here we use the crucial assumption $\lambda_+<1$.
\end{itemize}

\paragraph{Averaging the branching bound over the static environment.}
{The preceding first-moment comparison \eqref{eq:ts2} concerns starts
close to each $\mathbf x_m$.} However, this does not suffice to directly
transfer the fractional branching sum \eqref{eq:B} between static
and independent-visit models because it would require uniform
estimates over all $v\in\Hc_d^{(b)}$.
{
Our goal is then to prove
\begin{equation}
 \Eenv\!\left[
 \int_0^t\sum_v K_r^\omega(x,v)(\Lambda_v)_+
       K_{t-r}^\omega(v,\mathbf0)^{1+q}\,\dd r
 \,\middle|\,\Lambda_x,\Lambda_{\mathbf0}\right]
 \ll \widehat K_t(x,\mathbf0)
       +\widehat K_t(x,\mathbf0)^{1+q},
 \label{eq:sketch-static-branching-bound}
\end{equation}
uniformly in the conditioned endpoint rates, for the starts and
times under consideration.
}
The late-split estimate from the previous paragraph holds for every
static environment and gives a contribution
$\ll K_t^\omega(x,\mathbf0)$ after integration in $r$.
{Its conditional environmental mean is therefore
$\ll\widehat K_t(x,\mathbf0)$ by \eqref{eq:ts2}.}
For the remaining ranges, we need transfer estimates from the independent-visit model to the static model by comparing them. The arguments consist of two parts.
\begin{enumerate}
\item \emph{Replace the fractional power by one or two arms.}
For fixed $(r,v)$, {write $s=t-r$.}
The static integrand (analogue of the fractional branching sum
\eqref{eq:B}) equals
\[
\begin{aligned}
&K_r^\omega(x,v)(\Lambda_v)_+
                  K_s^\omega(v,\mathbf0)^{1+q}=(\bar\lambda\,\widehat K_r(x,v)
                     \widehat K_s(v,\mathbf0)^{1+q})\,
\underbrace{\frac{K_r^\omega(x,v)}{\widehat K_r(x,v)}
 \frac{(\Lambda_v)_+}{\bar\lambda}
 \underbrace{\left(
 \frac{K_s^\omega(v,\mathbf0)}{\widehat K_s(v,\mathbf0)}
 \right)^{1+q}}_{=:a^{1+q}}}_{\text{correction term}}.
\end{aligned}
\]
The factor $\bar\lambda\,\widehat K_r(x,v)
\widehat K_s(v,\mathbf0)^{1+q}$ is our reference weight.
Positive association and Jensen's inequality
(Lemma~\ref{lem:scalar-type-domination}) give
\[
 \bar\lambda\sum_v\widehat K_r(x,v)
       \widehat K_{t-r}(v,\mathbf0)^{1+q}
 \le\widehat{\mathcal B}_{t,r}^{(1+q)}(x).
\]
Thus, the preceding bound \eqref{eq:B2} controls the total reference
weight; the remaining task is to bound
{the conditional environmental mean of the correction,
averaged with these weights.}

Using $a^{1+q}\le a+a^2$, we reduce this task to analyzing correction
terms introduced by configurations with a trunk with one or two
arms ($a$ and $a^2$, respectively). In either case, we call the
resulting configuration a \emph{diagram}
(Figure~\ref{fig:sketch-shared-rate}).
{For either diagram, we need to extend the preceding estimates
for an ordinary bridge.}

\item {\emph{Control repeated uses of the same rate.}
Consider the one-arm and two-arm diagrams in
Figure~\ref{fig:sketch-shared-rate}. Choose $(r,v)$ using the normalized reference weights, then sample
the trunk and arms independently under their bridge laws.
For each visit of duration $u$ at $z$, replacing the independent-visit
factor $h(u)$ by the static factor $e^{-(1-\Lambda_z)u}$ introduces
$Y_z(u)$. The correction weight of the sampled diagram is therefore
\[
 \frac{(\Lambda_v)_+}{\bar\lambda}
 \prod_{\text{visits }(z,u)\text{ in the diagram}}Y_z(u),
\]
where the product includes every holding occurrence, including
repeated visits to the same vertex.
For fixed $(r,v)$, averaging over the sampled trunk and arms gives
\[
 \underbrace{\frac{K_r^\omega(x,v)}
                  {\widehat K_r(x,v)}}_{\text{trunk}}
 \underbrace{\frac{(\Lambda_v)_+}
                  {\bar\lambda}}_{\text{splitting factor}}
 \underbrace{\left(
 \frac{K_{t-r}^\omega(v,\mathbf0)}
      {\widehat K_{t-r}(v,\mathbf0)}
 \right)^j}_{\text{one or two arms}},
 \qquad j\in\{1,2\}.
\]

Bound the splitting factor by $\lambda_+/\bar\lambda$.
As in the repeated-visit estimate following \eqref{eq:Hh},
we bound all but one holding ratio at each vertex and use
$\Eenv[Y_z(u)]=1$ for the remaining ratio.
At the conditioned endpoints $x,\mathbf0$, we bound the first
ratio as well. 
The holding-time bounds then have two forms:
\begin{itemize}
\item \emph{Arms.}
For retained counts, the clock-integration argument in \eqref{eq:1.5} gives
a constant cost per bounded ratio. The omitted counts have an
exponentially small contribution.

\item \emph{Trunk.}
The trunk may have too few mutations for those clock estimates.
Use instead
\[
 Y_z(u)\le p_\varepsilon^{-1}e^{\varepsilon u}.
\]
The trunk holding times sum to $r$, so these exponential factors
multiply to at most $e^{\varepsilon r}$.
\end{itemize}
Each repeated occurrence thus costs a constant factor, with an
additional factor $e^{\varepsilon r}$.
Proposition~\ref{prop:collisionmgf} bounds exponential moments of
the repetition count, and
Proposition~\ref{prop:fractional-short-trunk} gives an exponential
tail for $r$.
For sufficiently small $\varepsilon$, Hölder's inequality combines
these bounds to control the correction averaged over the diagram
and the environment, conditional on the endpoint rates
(Lemma~\ref{lem:abstract-environment-concentration}(ii)).
Together with the reference-weight bound and the late-split estimate,
this proves \eqref{eq:sketch-static-branching-bound}.
}
\end{enumerate}

% Keep the existing Figure~\ref{fig:sketch-shared-rate} unchanged.

\begin{figure}[htbp]
\centering
\begin{tikzpicture}[
 x=1cm,y=1cm,font=\small,line width=.6pt,
 endpoint/.style={circle,fill=black,inner sep=0pt,minimum size=3pt},
 splitting/.style={circle,draw,fill=black!20,inner sep=0pt,minimum size=6pt}]
 \node at (2.3,1.35) {One arm: $a$};
 \draw (0,0)--node[above] {trunk, $r$} (2.1,0)
              --node[above] {arm, $t-r$} (4.6,0);
 \node[endpoint,label=below:{$x$}] at (0,0) {};
 \node[splitting,label=below:{$v$}] at (2.1,0) {};
 \node[endpoint,label=below:{$\mathbf0$}] at (4.6,0) {};
 \begin{scope}[xshift=7cm]
  \node at (2.3,1.35) {Two arms: $a^2$};
  \draw (0,0)--node[above] {trunk, $r$} (2.1,0);
  \draw (2.1,0)--node[above,sloped] {arm, $t-r$} (4.6,.65);
  \draw (2.1,0)--node[below,sloped] {arm, $t-r$} (4.6,-.65);
  \node[endpoint,label=below:{$x$}] at (0,0) {};
  \node[splitting,label=below:{$v$}] at (2.1,0) {};
  \node[endpoint,label=right:{$\mathbf0$}] at (4.6,.65) {};
  \node[endpoint,label=right:{$\mathbf0$}] at (4.6,-.65) {};
 \end{scope}
\end{tikzpicture}
\caption{Both diagrams contain one trunk and one splitting factor
$(\Lambda_v)_+/\bar\lambda$ at $v$.  The two arms on the right end at
the same target vertex $\mathbf0$.}
\label{fig:sketch-shared-rate}
\end{figure}
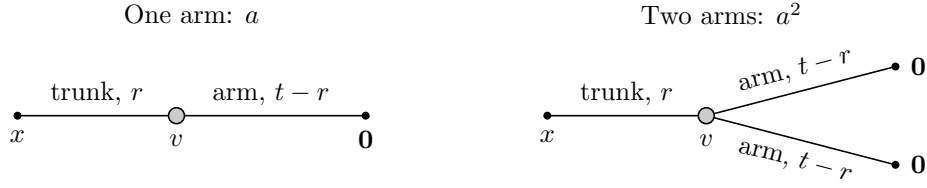

\paragraph{From moments to first-passage times.}
At a split in a fixed environment, let $Z,Z'$ be the target-hitting counts
from the two i.i.d.\ descendant families.
Subadditivity and Jensen's inequality give
\[
0\le\E\bigl[(Z+Z')^{1+q}-Z^{1+q}-(Z')^{1+q}\bigr]
\le2\E[Z]\E[Z^q]\le2(\E[Z])^{1+q}.
\]
Fix $L\ge0$, and suppose
$\Lambda_y,\Lambda_{\mathbf0}\ge-L$.
{First-event conditioning gives}
\begin{align}
\Ew_y[N_t(\mathbf0)^{1+q}]
&\le K_t^\omega(y,\mathbf0)
+2\int_0^t\sum_v K_r^\omega(y,v)(\Lambda_v)_+
K_{t-r}^\omega(v,\mathbf0)^{1+q}\,\dd r.
\label{eq:moment-connection}
\end{align}
Averaging over the environment conditional on the endpoint rates,
\eqref{eq:sketch-static-branching-bound} and
\eqref{eq:ts2} give
\[
\Eenv[\Ew_y[N_t(\mathbf0)^{1+q}]
       \mid\Lambda_y,\Lambda_{\mathbf0}]
\ll \widehat K_t(y,\mathbf0)
   +\widehat K_t(y,\mathbf0)^{1+q},
\]
whereas \eqref{eq:ts2} also implies
\[
\Eenv[K_t^\omega(y,\mathbf0)
       \mid\Lambda_y,\Lambda_{\mathbf0}]
\gg_L\widehat K_t(y,\mathbf0).
\]Fractional Paley--Zygmund now leads to
{
\[
\Eenv[\Pw_y(N_t(\mathbf0)>0)
       \mid\Lambda_y,\Lambda_{\mathbf0}]
\gg_L\frac{\widehat K_t(y,\mathbf0)}
              {1+\widehat K_t(y,\mathbf0)}.
\]
At a suitable fixed time after the deterministic center,
$\widehat K_t(y,\mathbf0)\ge1$.
Proposition~\ref{prop:quenched-barrier} concentrates the success
probability about its conditional environmental mean.
}
Thus, for some $c_L>0$,
{with high environment probability,}
\[
\Pw_y(\tau_y^\omega\le t)
\ge\Pw_y(N_t(\mathbf0)>0)\ge c_L.
\]

Then, we apply a bootstrapping argument to lift the lower bound $c_L$
to a lower bound close to one. 
Conditional on the survival event $\mathsf S$, we run the static
process for a finite time, until we find $Q$ particles at rates at
least $-L$ within bounded time and distance, with failure probability
$\eta>0$ (Lemma~\ref{lem:survival-preliminary-family}, with $L$ chosen large enough).
We keep particles whose descendants survive forever.
Conditional on the history of these surviving lineages, their
descendant systems are independent, each conditioned on its own
survival. Then, we evolve these particles independently and require
one of them to hit the target by the desired time. 
The total failure probability is then of the form
\[
\Pw_{\mathbf x_m}\!\left(
\tau_{\mathbf x_m}^\omega>\widehat\theta_{d,m}+A
\,\middle|\,\mathsf S\right)
\le\eta+(1-c_L)^Q.
\]
Choosing $Q$ large and $\eta$ small leads to the upper bound of the FPT.
The lower bound follows from the first-moment method by counting the
total occupation times before $\widehat\theta_{d,m}-A$.

\section{First-moment analysis and deterministic centering}
\label{sec:first-moment-analysis}

\subsection{First-moment estimates for the independent-visit model}
\label{sec:freshprofile}

Section~\ref{sec:heuristics} derived the independent-visit representation
\eqref{eq:freshK} and the spatial coefficient \eqref{eq:exact-pdm}.
In arbitrary-endpoint notation, \eqref{eq:freshK} becomes
\[
 \widehat K_t(x,y)=\sum_{n\ge0}P^n(x,y)h^{*(n+1)}(t).
\]
By spatial transitivity, the independent-visit kernel depends only on the
distance between its endpoints:
\begin{equation}
 \widehat K_t(x,y)=\widehat K_{d,d_{\mathrm H}(x,y)}(t),
 \qquad
 \widehat M_t(y):=\int_t^{t+1}\widehat K_s(y,\mathbf 0)\dd s.           \label{eq:fresh-unit-window}
\end{equation}

\subsubsection{Local limit estimates}
In this section, we estimate the spatial and holding-time factors in each fixed-count term of
\eqref{eq:freshK} separately.  Lemma~\ref{lem:spatialLCLT} treats the spatial factor
$p_{d,m}(n)$ through \eqref{eq:tilted-pdm}, and
Lemma~\ref{lem:clockLCLT} treats the holding-time factor
$h^{*(n+1)}(t)$ after the $\kappa$-tilt.  Proposition~\ref{prop:global-oneobject}
combines the two estimates into the fixed-count exponent $G$ and sums over $n$ to
obtain $\widehat F$.

For $z>0$, define integer-valued tilted coordinate variables by
\[
 \Pp_z(X^A=k)=\frac{\bigl([w^k]A_b(w)\bigr)z^k}{A_b(z)},\qquad
 \Pp_z(X^C=k)=\frac{\bigl([w^k]C_b(w)\bigr)z^k}{C_b(z)}.
\]
Under the product law of these variables, write $\E_z$ and
$\Var_z$ for expectation and variance. Their means are
\[
 \E_z[X^A]=\frac{zA_b'(z)}{A_b(z)},\qquad
 \E_z[X^C]=\frac{zC_b'(z)}{C_b(z)}.
\]
For $\theta\in[0,1]$ and $z>0$, define
\begin{equation}
 J_\theta(z)=\theta\frac{zA_b'(z)}{A_b(z)}
 +(1-\theta)\frac{zC_b'(z)}{C_b(z)}.                    \label{eq:Jtheta}
\end{equation}
For $0\le\ell\le d$, let $\Sigma_{d,\ell,z}$ denote the sum of
$\ell$ independent copies of $X^A$ and $d-\ell$ independent copies of
$X^C$, all with tilt $z$. Then
\[
 \frac1d\E_z[\Sigma_{d,\ell,z}]=J_{\ell/d}(z).
\]
Recall \eqref{eq:exact-pdm}. 
Expressing the coefficient under this tilted law gives, for every $z>0$,
\begin{equation}
 p_{d,m}(n)
 =\frac{n!}{(dz)^n}A_b(z)^mC_b(z)^{d-m}
   \Pp_z(\Sigma_{d,m,z}=n).                             \label{eq:tilted-pdm}
\end{equation}
Thus, choosing $z=z_n$ through
$J_{m/d}(z_n)=n/d$ makes $\E_{z_n}[\Sigma_{d,m,z_n}]=n$ and turns
coefficient extraction into a central local-limit problem.

\begin{lemma}\label{lem:spatialLCLT}
Let $K_z\Subset(0,\infty)$ be compact and fix $C_0<\infty$.  Uniformly for
$m/d\in[0,1]$, $z\in K_z$, and admissible integers $j$ satisfying
$$
 \bigl|j-\E_z[\Sigma_{d,m,z}]\bigr|\le C_0\log d,
$$
one has
\begin{equation}
 \Pp_z(\Sigma_{d,m,z}=j)
 \sim_{K_z,C_0}(1+\one_{\{b=2\}})
   \frac{1}{\sqrt{2\pi\Var_z(\Sigma_{d,m,z})}}
   \exp\left(
      -\frac{\bigl(j-\E_z[\Sigma_{d,m,z}]\bigr)^2}
             {2\Var_z(\Sigma_{d,m,z})}
   \right),                                             \label{eq:uniformspatialLCLT}
\end{equation}
where, in the binary case, $j\equiv m\pmod2$.  Moreover,
\begin{equation}
 \sup_j\Pp_z(\Sigma_{d,m,z}=j)\ll_{K_z}d^{-\frac12}.     \label{eq:spatialconcentration}
\end{equation}
Both conclusions remain uniform after changing $m$ and $j$ by
$O(\log d)$, provided the appropriate lattice is respected.
\end{lemma}

\begin{proof}
We apply the quantitative lattice local limit theorem
in~\cite[Corollary~1.11]{GiulianoWeber2017}.
\begin{itemize}
\item For $b>2$, the $A$ law has
uniformly positive masses at $1,2$, and the $C$ law at $2,3$, for $z\in K_z$.
\item For $b=2$, apply the theorem to $(X^A-1)/2$ and $X^C/2$ instead.  Both
transformed laws have uniformly positive masses at $0,1$, and their sum is
$(\Sigma_{d,m,z}-m)/2$.
\end{itemize}
Thus all the summands, after this binary rescaling,
are integer-valued and have a uniformly positive Bernoulli part: for each
such variable $X$,
\[
 \sum_{k\in\mathbb Z}
 \min\bigl\{\Pp_z(X=k),\Pp_z(X=k+1)\bigr\}\ge c>0.
\]
This remains true when all coordinates have the same type.

The original moment generating functions are
$A_b(ze^t)/A_b(z)$ and $C_b(ze^t)/C_b(z)$.  They are uniformly bounded
for $z\in K_z$ and $t$ in a fixed neighbourhood of zero.  Consequently,
the third absolute moments of the span-one summands are uniformly bounded,
and their variances are bounded above and away from zero.  In the cited
corollary, take the moment function $|x|^3$ and a fixed positive Bernoulli
extraction parameter smaller than the bound just obtained.  The total
variance and total extraction parameter are both of order $d$, and the
third-moment ratio is $O(d^{-1/2})$.  The corollary therefore gives an
additive error $O(\sqrt{\log d}/d)$ in its Gaussian approximation.
Its central-window condition holds because the squared standardized
displacement is $O(\log^2d/d)$.

Undoing the rescaling gives, uniformly in the window of the lemma,
\[
 \Pp_z(\Sigma_{d,m,z}=j)
 =(1+\one_{\{b=2\}})
 \frac{\exp\!\left(
 -\dfrac{\bigl(j-\E_z[\Sigma_{d,m,z}]\bigr)^2}
 {2\Var_z(\Sigma_{d,m,z})}\right)}
 {\sqrt{2\pi\Var_z(\Sigma_{d,m,z})}}
 +O\!\left(\frac{\sqrt{\log d}}d\right).
\]
For $b=2$, the variance of the transformed sum is
$\Var_z(\Sigma_{d,m,z})/4$, which gives the factor two; the original sum
belongs to $m+2\mathbb Z$.  Since its variance is of order $d$, the
Gaussian term in the stated window is comparable to $d^{-1/2}$.
This proves \eqref{eq:uniformspatialLCLT}.

For the global bound \eqref{eq:spatialconcentration}, use the same uniformly positive adjacent masses
to extract from each span-one summand a translated fair Bernoulli variable
with a fixed positive probability.  The extraction indicators are
independent, so a Chernoff bound shows that at least $cd$ Bernoulli
variables are extracted except on an event of probability $e^{-c'd}$.
Conditional on the indicators and the remaining variables, their sum is
a translate of a fair binomial variable.  With at least $cd$ trials its
largest point mass is $O(d^{-1/2})$; the exceptional event has
exponentially small probability.  This proves \eqref{eq:spatialconcentration}.

All bounds are uniform in the numbers of $A$ and $C$ summands.  Changing $m$ by
$O(\log d)$ changes the mean by
$O(\log d)(\E_z[X^A]-\E_z[X^C])=O(\log d)$, uniformly on $K_z$.
Thus the stated changes of $m$ and $j$ remain within the same order of
central window, with the appropriate parity in the binary case.
\end{proof}

For $\kappa>\lambda_+-1$, put
$f_\kappa(r)=e^{-\kappa r}h(r)/H(\kappa)$.
Let $R$ have density $f_\kappa$.
Let $\widetilde f_{N,\kappa}$ denote the density of the sum of $N$
independent variables with density $f_\kappa$.

\begin{lemma}\label{lem:clockLCLT}
Let $K_\kappa\Subset(\lambda_+-1,\infty)$ be compact and fix
$C_0<\infty$ and $0<c_N<C_N<\infty$. Uniformly for
$\kappa\in K_\kappa$, $c_Nd\le N\le C_Nd$, and $|a|\le C_0\log d$,
\begin{align}
 \widetilde f_{N,\kappa}(N\mu_\kappa+a)
 &\sim_{K_\kappa,C_0,c_N,C_N}\frac1{\sqrt{2\pi N\sigma_\kappa^2}}
   \exp\left(-\frac{a^2}{2N\sigma_\kappa^2}\right),
                                                        \label{eq:uniformclockLCLT}\\
 \sup_x\widetilde f_{N,\kappa}(x)
 &\ll_{K_\kappa}N^{-\frac12},                            \label{eq:clockconcentration}
\end{align}
where $\mu_\kappa=-H'(\kappa)/H(\kappa)$ and
$\sigma_\kappa^2=\Var_{f_\kappa}(R)$. In particular, the density at its
mean is comparable to $N^{-\frac12}$, and the ratio of the global upper bound
in \eqref{eq:clockconcentration} to this mean density is bounded.
\end{lemma}

\begin{proof}
The density $f_\kappa$ is a mixture of exponential densities:
\[
 f_\kappa(r)
 =\int\frac{P(d\lambda)}
 {(\kappa+1-\lambda)H(\kappa)}
 (\kappa+1-\lambda)e^{-(\kappa+1-\lambda)r}.
\]
On $K_\kappa$, all component rates are bounded away from zero, so the
mixture has a uniform exponential moment.  Its variance is positive
for every $\kappa$ and continuous in $\kappa$ by dominated convergence.
Compactness therefore bounds the variance away from zero.  Thus
\begin{equation}
 \sup_{\kappa\in K_\kappa}\E_{f_\kappa}\bigl[e^{cR}\bigr]\le C,
 \qquad c\le\Var_{f_\kappa}(R)\le C.                       \label{eq:clockuniformmoments}
\end{equation}
Also $\|f_\kappa\|_\infty=H(\kappa)^{-1}$ is uniformly bounded.
It follows that $(R-\mu_\kappa)/\sigma_\kappa$ has both a uniformly
bounded density and a uniformly bounded third absolute moment.
The density local limit bound in~\cite[Theorem~1]{BobkovGotze2025}
therefore applies uniformly in $\kappa$.  After undoing the normalization,
it gives the global additive estimate
\[
 \sup_{x\in\mathbb R}
 \left|\widetilde f_{N,\kappa}(x)
 -\frac{1}{\sqrt{2\pi N\sigma_\kappa^2}}
 \exp\!\left(-\frac{(x-N\mu_\kappa)^2}{2N\sigma_\kappa^2}\right)
 \right|\ll_{K_\kappa}N^{-1}.
\]
When $|x-N\mu_\kappa|\le C_0\log d$ and $N\asymp d$, the Gaussian
term is comparable to $N^{-1/2}$, so the relative error is
$O(N^{-1/2})$.  This proves \eqref{eq:uniformclockLCLT} and the lower
bound at the mean.  Taking the supremum in the additive estimate proves
\eqref{eq:clockconcentration}.
\end{proof}

\subsubsection{Estimates on the first moments and related quantities}
Recall \eqref{eq:H}. 
Write $\kappa_\partial=\lambda_+-1<0$ and
$H_\partial=\lim_{\kappa\downarrow\kappa_\partial}H(\kappa)$,
allowing $H_\partial=\infty$.  When $H_\partial<\infty$, the
clock objective is also evaluated at $\kappa_\partial$ by continuity.
An infimum attained at this added endpoint is called a boundary clock
minimum.  To apply the local-limit estimates uniformly, we need the nearly
maximizing count densities and their tilts to stay bounded and away from
the boundaries of their domains.  The next lemma verifies this near the
zero set of $\widehat F$ and identifies the regular region used in
Proposition~\ref{prop:global-oneobject}.

\begin{lemma}\label{lem:variational-regularity}
The function $G$ is continuous on
$\{(\theta,u,\nu):0\le\theta\le1,\ u>0,\ \nu\ge\theta\}$,
with its limiting values at the feasible boundary.  The function
$\widehat F$ is continuous on $[0,1]\times(0,\infty)$.
On every compact positive-time parameter set, its maximizing count
sets are nonempty and contained in a common bounded interval.
No count $\nu=\theta>0$ is maximizing.

For $\nu>0$, the clock minimum has a unique location
$\kappa^c(u,\nu)\in[\kappa_\partial,\infty)$, with the endpoint allowed
only when $H_\partial<\infty$.  The map $\kappa^c$ is continuous.
Define the regular region
\[
 \mathcal R=\left\{(\theta,u):
 \text{every maximizing count }\nu\text{ satisfies }\nu>\theta,\quad
 \kappa^c(u,\nu)>\kappa_\partial
 \right\}.
\]
This region is relatively open.  On every compact
$\mathcal K\subset\mathcal R$, there are
$\varepsilon_{\mathcal K},\delta_{\mathcal K}>0$,
$0<\nu_-<\nu_+<\infty$, and compact tilt intervals
$Z_{\mathcal K}\Subset(0,\infty)$ and
$K_{\mathcal K}\Subset(\kappa_\partial,\infty)$ such that
\begin{equation}
 G(\theta,u,\nu)\ge\widehat F(\theta,u)-2\varepsilon_{\mathcal K}
 \ \Longrightarrow\
 \nu_-\le\nu\le\nu_+,\qquad
 \nu-\theta\ge\delta_{\mathcal K},\quad
 z(\theta,\nu)\in Z_{\mathcal K},\quad
 \kappa^c(u,\nu)\in K_{\mathcal K}.                 \label{eq:UO14}
\end{equation}
Every point with $\widehat F(\theta,u)\ge0$ belongs to $\mathcal R$,
and every maximizing clock tilt at such a point is positive.
\end{lemma}

\begin{proof}
For $\nu>\theta$, the spatial minimizer is unique: its mean map
$J_\theta$ increases from $\theta$ to infinity, and
\[
 zJ_\theta'(z)=\theta\Var_z(X^A)+(1-\theta)\Var_z(X^C)>0.
\]
For $\nu>0$, the clock objective has strictly positive second derivative
$\nu\operatorname{Var}_{f_\kappa}(R)$.  Its derivative is
$u-\nu\mu_\kappa$, where $\mu_\kappa=-H'(\kappa)/H(\kappa)$ strictly
decreases to zero.  Write
$\mu_\partial=\lim_{\kappa\downarrow\kappa_\partial}\mu_\kappa
\in(0,\infty]$.
\begin{itemize}
\item If $u/\nu<\mu_\partial$, the minimum is the unique
solution of $\mu_\kappa=u/\nu$.
\item If $u/\nu\ge\mu_\partial<\infty$, it is $\kappa_\partial$.
\end{itemize}
Finiteness of $\mu_\partial$ implies
$H_\partial<\infty$: otherwise the exponential mixture defining
$f_\kappa$ puts asymptotically all its mixing mass on arbitrarily small
rates, forcing $\mu_\kappa\to\infty$.  These descriptions prove
continuity of $\kappa^c$ and of $\chi$ for $\nu,u>0$.  They also give
\[
 \kappa^c(u,\nu)\le\kappa_\partial+\nu/u,
\]
because $\mu_\kappa\le(\kappa-\kappa_\partial)^{-1}$.
When the clock minimum is at $\kappa_\partial$, the derivative of
$\chi$ with respect to $\nu$ is $\log H_\partial$.
Where both minima are interior, these positive variances and the
implicit function theorem give smooth dependence of the two tilts.
Differentiating their optimized values, the tilt derivatives cancel by
stationarity, giving
\[
 \partial_\theta G=\log\frac{A_b(z)}{C_b(z)},\qquad
 \partial_uG=\kappa^c(u,\nu).
\]
The same time derivative holds at a boundary clock minimum by its
uniqueness and continuity.  Derivatives in $\theta$ at $0,1$ are
understood from the feasible side.

At $\nu=0$, Lemma~\ref{lem:smallnuexpansions} gives
$\chi(\nu,u)=\kappa_\partial u+
O(\nu(1+|\log\nu|))$, uniformly on compact time intervals.
The spatial objective has its usual continuous extension at
$\nu=\theta$, and at the corner
\begin{equation}
 \bigl|\Psi_b(\theta,\nu)\bigr|\le C\nu\bigl(1+|\log\nu|\bigr),
 \qquad 0\le\theta\le\nu\downarrow0.                 \label{eq:UO13a}
\end{equation}
To verify this bound, put $r=\nu-\theta$.  The expansions of $A_b,C_b$
give $r\asymp\theta z+z^2$ for $b>2$, and $r\asymp z^2$ for $b=2$,
uniformly near this corner.  Thus $cr\le z\le C\sqrt r$; integrating
$\partial_\nu\Psi_b=\log(\nu/z)$ from the geodesic boundary $\nu=\theta$ proves
\eqref{eq:UO13a}.  In particular,
\begin{equation}
 G_\partial(\theta,u)
 =\theta\log\theta-\theta-\theta\log(b-1)+\chi(\theta,u)
 \quad(\theta>0),\qquad G_\partial(0,u)=\kappa_\partial u,
                                                               \label{eq:UO7}
\end{equation}
and $G_\partial(\theta,u)\to\kappa_\partial u$ as $\theta\downarrow0$.

The elementary bounds $A_b(z),C_b(z)\le e^z$ and
$H(\kappa)\le(\kappa-\kappa_\partial)^{-1}$ imply
\begin{equation}
 G(\theta,u,\nu)\le
 \kappa_\partial u+\nu\left(1+\log\frac u\nu\right).
                                                               \label{eq:UO12}
\end{equation}
This tends uniformly to $-\infty$ as $\nu\to\infty$ on compact time
intervals.  We may therefore restrict the maximum to $[\theta,M]$
for a fixed sufficiently large $M$.  This is a continuous compact-valued
feasible correspondence.  Berge's maximum theorem
\cite{FeinbergKasyanovVoorneveld2014} gives continuity of $\widehat F$,
attainment, and upper semicontinuity of the maximizing count sets.
For its usual product-domain formulation, extend $G$ continuously by
replacing an infeasible count $\nu<\theta$ with $\theta$.
For $\theta>0$, $\chi$ is locally Lipschitz in $\nu$ near
$\theta$, whereas
$\partial_\nu\Psi_b=\log(\nu/z)\to+\infty$ as
$\nu\downarrow\theta$.  Hence
\begin{equation}
 \widehat F(\theta,u)>G_\partial(\theta,u),\qquad\theta>0.
                                                               \label{eq:UO9}
\end{equation}

We will also use the following compactness consequence.  On a compact
set of feasible triples $(\theta,u,\nu)$, the continuous deficit
$\widehat F-G$ has a strictly positive minimum outside any open
neighbourhood of the maximizing triples, whenever this complement is
nonempty.  Indeed, the deficit has no zero on that compact complement.
At a regular point the maximizing counts and their clock tilts stay
away from the boundaries of their domains.  Upper
semicontinuity preserves this separation nearby, proving that
$\mathcal R$ is open.  Applying the deficit observation on a compact
subset of $\mathcal R$ proves \eqref{eq:UO14}; the spatial tilt bounds
follow from the continuity of the inverse of $J_\theta$.

Finally suppose $\widehat F(\theta,u)\ge0$.  A maximizing count cannot
be zero, since the only feasible zero count has $\theta=0$ and value
$\kappa_\partial u<0$.  By \eqref{eq:UO9}, it also cannot equal
$\theta>0$.  Put $\ell_\theta(z)=
\theta\log A_b(z)+(1-\theta)\log C_b(z)$.
The envelope identity, valid also for a boundary clock minimum, is
\begin{equation}
 \partial_\nu G=\log\left(\frac{\nu H(\kappa^c)}z\right).
                                                               \label{eq:UO8}
\end{equation}
At a maximizing count this vanishes, while
$\nu=z\ell_\theta'(z)$.  Consequently
\[
 \widehat F(\theta,u)=
 \kappa^c u-\bigl(z\ell_\theta'(z)-\ell_\theta(z)\bigr).
\]
Lemma~\ref{lem:positive-log-derivatives} makes the parenthesis strictly
positive.  Thus $\kappa^c>0$, proving the final assertion.
\end{proof}

The regular region $\mathcal R$ need not contain every positive-time
point.  A finite boundary mean allows the clock infimum
to occur at $\kappa_\partial$, outside the range of the preceding
clock local-limit estimates.  All points of the zero graph, and hence a fixed neighbourhood
of that graph, lie in $\mathcal R$.  We retain a global upper estimate
outside this region and state the two-sided estimates only where their
saddles are regular.

\begin{definition}
\label{def:retained-counts}
Fix $J\Subset(0,\infty)$ and $\varepsilon>0$.  For $0\le m\le d$ and
$t>0$, the \emph{retained mutation counts} are
\begin{equation*}
 \mathcal I_d^{J,\varepsilon}(m,t)
 =\left\{n\ge0:\ p_{d,m}(n)>0,\ \frac nd\in J,\quad
 G\left(\frac md,\frac td,\frac nd\right)
 \ge \widehat F\left(\frac md,\frac td\right)-\varepsilon\right\}.
\end{equation*}
In other words, these are the possible mutation counts in the chosen density range whose
exponential rates are within $\varepsilon$ of the maximum.
\end{definition}

The following proposition bounds $\widehat K_{d,m}(t)$ and shows
that the retained counts contribute all but an exponentially small
fraction of this sum.
Its final assertion controls the distant count tails even when each
summand is multiplied by $e^{an}$ for fixed $a>0$.
\begin{proposition}
\label{prop:global-oneobject}
Under Assumption~\ref{ass:environment}, every compact
$\mathcal K\subset[0,1]\times(0,\infty)$ satisfies the global upper bound
\begin{equation}
 \widehat K_{d,m}(t)\ll_{\mathcal K}
 d^{3/2}e^{d\widehat F(m/d,t/d)},\qquad \left(\frac md,\frac td\right)\in \mathcal{K}.                    \label{eq:global-oneobject-upper}
\end{equation}
If in addition $\mathcal K\subset\mathcal R$, there are
$0<\nu_-<\nu_+<\infty$ and $\varepsilon_{\mathcal K}>0$ with the
following properties.  All estimates below are uniform, for all sufficiently
large $d$, over $0\le m\le d$ and $t>0$ such that
$(m/d,t/d)\in\mathcal K$. 

\begin{enumerate}
\item The independent-visit kernel \eqref{eq:freshK} satisfies
\begin{equation}
 d^{-\frac12}
 e^{d\widehat F\left(\frac md,\frac td\right)}
 \ll_{\mathcal K} \widehat K_{d,m}(t)
 \ll_{\mathcal K}d^{\frac12}
 e^{d\widehat F\left(\frac md,\frac td\right)}.       \label{eq:UO1}
\end{equation}
\item Put $J_{\mathcal K}=[\nu_-,\nu_+]$.  For every
$0<\varepsilon\le\varepsilon_{\mathcal K}$ there is
$c_{\mathcal K,\varepsilon}>0$ such that
\begin{equation}
 \sum_{n\notin\mathcal I_d^{J_{\mathcal K},\varepsilon}(m,t)}
 p_{d,m}(n)h^{*(n+1)}(t)
 \le e^{-c_{\mathcal K,\varepsilon}d}\widehat K_{d,m}(t).
                                                        \label{eq:UO2}
\end{equation}
\item For every fixed $a>0$, there are
$J_{\mathcal K,a}\Subset(0,\infty)$ and
$c_{\mathcal K,a}>0$ such that
\begin{equation}
 \sum_{n:\,\frac nd\notin J_{\mathcal K,a}}
 e^{an}p_{d,m}(n)h^{*(n+1)}(t)
 \le e^{-c_{\mathcal K,a}d}\widehat K_{d,m}(t).
                                                        \label{eq:UO3}
\end{equation}
\end{enumerate}
\end{proposition}

\begin{remark}
The bounds in \eqref{eq:UO1} differ by a factor of $d$, or $\log d$
after taking logarithms.  Near the root, the clock tilts are bounded
away from zero by Lemma~\ref{lem:global-root-model-only}, so this gives
the $O(\log d)$ error in $\widehat\theta_{d,m}$ in
\eqref{eq:all-distance-center-log}.
\end{remark}

\begin{proof}
Write $\theta=m/d$, $u=t/d$, and $\nu=n/d$, and let $u_-,u_+$ bound
the second coordinate of $\mathcal K$.  We first prove the global
upper bound, then work on a compact subset of $\mathcal R$.

\emph{Step 1: a global upper bound and proof of \eqref{eq:global-oneobject-upper}.}
For every $\kappa>\kappa_\partial$ and $n\ge0$,
convolution with a probability density cannot increase the supremum norm.
Since $\|f_\kappa\|_\infty=H(\kappa)^{-1}$, exponential tilting gives
\[
 h^{*(n+1)}(du)
 =e^{\kappa du}H(\kappa)^{n+1}\widetilde f_{n+1,\kappa}(du)
 \le e^{\kappa du}H(\kappa)^n.
\]
Taking the infimum over $\kappa$ gives
$h^{*(n+1)}(du)\le e^{d\chi(n/d,u)}$.
Since $A_b(w)^mC_b(w)^{d-m}$ has nonnegative coefficients,
\eqref{eq:exact-pdm} gives, for every $z>0$ and $n\ge0$,
\[
 p_{d,m}(n)\le\frac{n!}{(dz)^n}A_b(z)^mC_b(z)^{d-m}.
\]
Stirling's bound with prefactor $C\sqrt{n+1}$, followed by the infimum
over $z>0$, gives
$p_{d,m}(n)\le C\sqrt{n+1}\,e^{d\Psi_b(\theta,n/d)}$,
with the convention $0\log0=0$.
Thus, on every fixed interval $0\le n/d\le R$,
\begin{equation}
 p_{d,m}(n)h^{*(n+1)}(du)
 \le C_R\sqrt d\,e^{dG(\theta,u,n/d)}.               \label{eq:UO16}
\end{equation}
To control the large-count tail, use $p_{d,m}(n)\le1$ and
$h(r)\le e^{\kappa_\partial r}$.  The latter bound gives
\[
 h^{*(n+1)}(du)
 \le\bigl(e^{\kappa_\partial\,\cdot}\bigr)^{*(n+1)}(du)
 =e^{\kappa_\partial du}\frac{(du)^n}{n!},
\]
and hence
\begin{equation}
 p_{d,m}(n)h^{*(n+1)}(du)
 \le e^{\kappa_\partial du}\frac{(du)^n}{n!}.         \label{eq:UO17}
\end{equation}
Continuity bounds $\widehat F$ below on $\mathcal K$.  For sufficiently
large fixed $R$, the tail in \eqref{eq:UO17} is exponentially smaller
than $e^{d\widehat F}$.  Indeed, for every fixed $a\ge0$,
\[
 \sum_{n>Rd}e^{an}\frac{t^n}{n!}
 \le e^{-Rd}\sum_{n\ge0}\frac{(e^{a+1}t)^n}{n!}
 =\exp\left(-Rd+e^{a+1}t\right).
\]
Take $a=0$ and use $t\le u_+d$.  Summing
\eqref{eq:UO16} over $O(d)$ remaining counts proves
\eqref{eq:global-oneobject-upper}.

\emph{Step 2: estimates in compact intervals.}
Now assume $\mathcal K\subset\mathcal R$.  Lemma~\ref{lem:variational-regularity}
supplies \eqref{eq:UO14}.  In particular the spatial saddle
$J_\theta(z_n)=n/d$ stays in a compact subset of $(0,\infty)$.
The exact clock equation
$\mu_{\kappa_n}=t/(n+1)$ also has a solution in a common compact
subset of $(\kappa_\partial,\infty)$: it is a uniformly small perturbation
of the interior solution with mean $u/\nu$.
Exponential tilting, Stirling's formula, and
Lemmas~\ref{lem:spatialLCLT} and \ref{lem:clockLCLT} therefore give,
uniformly for admissible retained counts,
\begin{equation}
 p_{d,m}(n)h^{*(n+1)}(du)
 \asymp d^{-1/2}e^{dG(\theta,u,\nu)}.                 \label{eq:UO15}
\end{equation}
Indeed, the spatial local probability contributes $d^{-1/2}$ and the
Stirling prefactor contributes $d^{1/2}$, while the clock density
contributes $d^{-1/2}$.  Replacing $(n+1)/d$ by $n/d$ in the clock
objective changes the exponent by a bounded amount, since $\log H$
is bounded on the retained compact interval.

\emph{Step 3: summation and discarded counts.}
A maximizing count lies in \eqref{eq:UO14}; choose an admissible integer
within two of its multiple by $d$.  The derivatives of $G$ are bounded
on a fixed neighbourhood of the retained set, so this rounding loses
only a bounded amount in the exponent.  Keeping this summand in
\eqref{eq:UO15} proves the lower bound in \eqref{eq:UO1}.
The $O(d)$ counts with $G\ge\widehat F-2\varepsilon_{\mathcal K}$
give the upper bound $Cd^{1/2}e^{d\widehat F}$.
Every other feasible count in a fixed bounded interval satisfies
$G(\theta,u,n/d)<\widehat F(\theta,u)-2\varepsilon_{\mathcal K}$.
By \eqref{eq:UO16}, the total contribution of these counts is at most
$Cd^{3/2}e^{d(\widehat F-2\varepsilon_{\mathcal K})}$.
The large-count factorial tail has a fixed gap by Step 1.
This proves the upper bound in \eqref{eq:UO1}.

Take $J_{\mathcal K}=[\nu_-,\nu_+]$ from \eqref{eq:UO14}.
Every feasible count outside this interval has
$G<\widehat F-2\varepsilon_{\mathcal K}$, and every count inside it
but outside $\mathcal I_d^{J_{\mathcal K},\varepsilon}(m,t)$ has
$G<\widehat F-\varepsilon$.  Combining \eqref{eq:UO16} on a bounded
count interval with the tail estimate from Step 1 and the lower bound
in \eqref{eq:UO1} proves \eqref{eq:UO2}, absorbing polynomial factors
into the exponential.

\emph{Step 4: distant count tails with exponential weights.}
For fixed $a>0$, choose
$0<r<\min\{\nu_-,\varepsilon_{\mathcal K}/a\}$.
By \eqref{eq:UO14}, every feasible count $n<rd$ has
$G<\widehat F-2\varepsilon_{\mathcal K}$.  Hence \eqref{eq:UO16} gives
\[
 \sum_{n<rd}e^{an}p_{d,m}(n)h^{*(n+1)}(t)
 \ll_{\mathcal K,a}d^{3/2}e^{d(\widehat F(\theta,u)-\varepsilon_{\mathcal K})}.
\]
The estimate in Step 1 for the tail weighted by $e^{an}$, together with
\eqref{eq:UO17}, gives the
same exponential gap for $n>Rd$ when $R$ is sufficiently large,
depending on $a$.  Divide both tails by the lower bound in
\eqref{eq:UO1} and absorb the polynomial factor into the gap.
This proves \eqref{eq:UO3} with $J_{\mathcal K,a}=[r,R]$.
\end{proof}

\begin{proposition}
\label{prop:global-oneobject-local}
Under Assumption~\ref{ass:environment}, let
$\mathcal K\subset\mathcal R$ be compact
(endpoints $\theta=0,1$ are allowed).  There are
$0<\nu_-<\nu_+<\infty$, $\varepsilon_{\mathcal K}>0$, and compact intervals
$$
 Z_{\mathcal K}\Subset(0,\infty),\qquad
 K_{\mathcal K}\Subset(\lambda_+-1,\infty),
$$
with the following properties.  All estimates are uniform, for all
sufficiently large $d$,
over $0\le m\le d$ and $t>0$ such that
$(m/d,t/d)\in\mathcal K$.  Put $J_{\mathcal K}=[\nu_-,\nu_+]$.
The interval $J_{\mathcal K}$ and the level
$\varepsilon_{\mathcal K}$ may be chosen to be the same as in
Proposition~\ref{prop:global-oneobject}.

If
$n\in\mathcal I_d^{J_{\mathcal K},\varepsilon_{\mathcal K}}(m,t)$,
then the equations
\begin{equation}
 J_{\frac{m}{d}}(z_n)=\frac nd,
 \qquad
 -\frac{H'(\kappa_n)}{H(\kappa_n)}=\frac{t}{n+1}
                                                        \label{eq:UO4}
\end{equation}
have unique solutions with
$z_n\in Z_{\mathcal K}$ and $\kappa_n\in K_{\mathcal K}$.
With the variables defined immediately before \eqref{eq:uniformspatialLCLT},
\begin{align}
 \sup_{0\le\ell\le d}\sup_{j\in\mathbb Z}
 \Pp_{z_n}(\Sigma_{d,\ell,z_n}=j)
 &\ll_{\mathcal K}d^{-\frac12},
 &
 \Pp_{z_n}(\Sigma_{d,m,z_n}=n)
 &\gg_{\mathcal K}d^{-\frac12}.                 \label{eq:UO5}
\end{align}
The lower bound is understood on the admissible lattice (and membership
in the retained set includes $p_{d,m}(n)>0$).

Fix $B<\infty$ and let $N$ be an integer with $|N-(n+1)|\le B$.
Uniformly for $|\Delta|\le B$,
\begin{equation}
 \widetilde f_{N,\kappa_n}(t+\Delta)
 \asymp_{\mathcal K,B}d^{-\frac12},
 \qquad
 \sup_x\widetilde f_{N,\kappa_n}(x)
 \ll_{\mathcal K,B}d^{-\frac12}.  \label{eq:UO6}
\end{equation}
\end{proposition}

\begin{proof}
\begingroup
Lemma~\ref{lem:variational-regularity} and Step 2 of
Proposition~\ref{prop:global-oneobject} give the retained interval and
level, and the unique exact tilts in the stated compact intervals.
Lemma~\ref{lem:spatialLCLT} gives the uniform spatial upper bound in
\eqref{eq:UO5}.  At $\ell=m$ the tilted mean is
$dJ_{m/d}(z_n)=n$, so its lower bound applies on the admissible lattice.
For the clocks, $N\asymp d$ and
$$
 |t+\Delta-N\mu_{\kappa_n}|
 \le |\Delta|+|n+1-N|\mu_{\kappa_n}\ll_{\mathcal K,B}1.
$$
Lemma~\ref{lem:clockLCLT} at this bounded displacement from the mean,
together with \eqref{eq:clockconcentration}, proves \eqref{eq:UO6}.
\endgroup
\end{proof}

We will need to compare first-moment kernels at nearby observation times.
Applying the preceding polynomial bounds separately at the two times
would lose an unnecessary power of the dimension.  The next proposition
compares the ratio directly: bounded time shifts change the kernel by
bounded factors, and logarithmic shifts by at most polynomial factors.
We use these comparisons to control the kernel \eqref{eq:freshK} within a logarithmic
distance of $du_*(m/d)$ in Corollary~\ref{cor:polynomial-root-band}
{and to bound the occupation integral before the center in
Proposition~\ref{prop:global-center}.}

\begin{proposition}
\label{prop:global-oneobject-shifts}
Under Assumption~\ref{ass:environment}, let
$\mathcal K\subset\mathcal R$ be compact
(endpoints $\theta=0,1$ are allowed).  There exists a compact interval
$K_{\mathcal K}\Subset(\lambda_+-1,\infty)$ such that the following
estimates hold.  Write
$\underline\kappa_{\mathcal K}=\min K_{\mathcal K}$ and
$\overline\kappa_{\mathcal K}=\max K_{\mathcal K}$.  For every fixed
$B<\infty$, uniformly for all sufficiently large $d$, over
$0\le m\le d$ and $t>0$ with $(m/d,t/d)\in\mathcal K$, and for
$0\le s\le B\log d$,
\begin{equation}
 e^{\underline\kappa_{\mathcal K}s}
 \ll_{\mathcal K,B}\frac{\widehat K_{d,m}(t+s)}{\widehat K_{d,m}(t)}
 \ll_{\mathcal K,B}e^{\overline\kappa_{\mathcal K}s},
                                                               \label{eq:UO6a}
\end{equation}
whereas, uniformly for $-B\log d\le s\le0$,
\begin{equation}
 e^{\overline\kappa_{\mathcal K}s}
 \ll_{\mathcal K,B}\frac{\widehat K_{d,m}(t+s)}{\widehat K_{d,m}(t)}
 \ll_{\mathcal K,B}e^{\underline\kappa_{\mathcal K}s}.
                                                               \label{eq:UO6b}
\end{equation}
For large $d$, positivity of $t+s$ is automatic.  In particular, for
$|s|\le B$ the ratio is bounded above and away from zero.  The same
assertions hold with $\widehat K_{d,m}(t)$ replaced by
$\int_t^{t+1}\widehat K_{d,m}(r)\dd r$.
\end{proposition}

\begin{proof}
Enlarge $\mathcal K$ by a fixed compact neighbourhood inside
$\mathcal R$; this contains
$(m/d,(t+s+r)/d)$ for every $|s|\le B\log d$, $0\le r\le1$,
and all sufficiently large $d$.
Take $K_{\mathcal K}$ from
Proposition~\ref{prop:global-oneobject-local} on this enlarged set.
Restrict the denominator to a retained set at a level strictly smaller than
$\varepsilon_{\mathcal K}$.  By \eqref{eq:UO2}, this set carries at least a
$1-e^{-cd}$ fraction of $\widehat K_{d,m}(t)$.  For each retained $n$, tilt
at the exact $\kappa_n$ from \eqref{eq:UO4}.  Since $t$ is the tilted mean,
\begin{equation}
 \frac{h^{*(n+1)}(t+s)}{h^{*(n+1)}(t)}
 =e^{\kappa_ns}
   \frac{\widetilde f_{n+1,\kappa_n}(t+s)}
        {\widetilde f_{n+1,\kappa_n}(t)}.               \label{eq:UO19}
\end{equation}
Lemma~\ref{lem:clockLCLT} gives, uniformly for $(m/d,t/d)\in\mathcal K$,
retained $n$, and $|s|\le B\log d$,
$$
 \frac{\widetilde f_{n+1,\kappa_n}(t+s)}
      {\widetilde f_{n+1,\kappa_n}(t)}
 \sim_{\mathcal K,B}\exp\!\left(-\frac{s^2}{2(n+1)\sigma_{\kappa_n}^2}\right)
 \sim_{\mathcal K,B}1.
$$
By positivity, summing only the counts retained at $t$ gives
\[
 \frac{\widehat K_{d,m}(t+s)}{\widehat K_{d,m}(t)}
 \ge (1-o(1))\min_{\kappa\in K_{\mathcal K}}e^{\kappa s}.
\]
Apply this same lower bound with base time $t+s$ and displacement $-s$.
Taking reciprocals gives
\[
 \frac{\widehat K_{d,m}(t+s)}{\widehat K_{d,m}(t)}
 \le (1+o(1))\max_{\kappa\in K_{\mathcal K}}e^{\kappa s}.
\]
Both base times belong to the enlarged set, so the errors are uniform;
the retained count sets at the two times need not agree.
These inequalities prove \eqref{eq:UO6a}--\eqref{eq:UO6b}.
Apply them uniformly with base time $t+r$, $0\le r\le1$, and integrate
to obtain the unit-window version.
\end{proof}

\subsection{Locating the deterministic first-moment centering}
\label{sec:root-curve}

To turn the first-moment estimates into a time center, we need a unique
zero of $u\mapsto\widehat F(\theta,u)$ and uniform linear bounds on either side of it.  The next
lemma proves both uniformly in the distance density, allowing multiple
maximizing saddles.  For the homogeneous counterpart at small distance
densities, see~\cite[Proposition~1(i), Appendix~B.1]{BZ}.
We start from the following observation: if every function attaining
the maximum has a positive time derivative at a zero, then their
maximum increases through that zero.

\paragraph{Zeros of a compact maximum.}
Positive time derivatives at all maximizing counts make $\widehat F(\theta,\cdot)$
increase through zero.  We use this fact here to locate the zero $u_*(\theta)$ and
in Lemma~\ref{lem:sparsederivative} below to differentiate it.
Lemma~\ref{lem:mutation-profile-monotonicity} below uses the same fact to
compare the zeros at different mutation rates.
Suppose locally that $F(p,t)=\max_{\nu\in I}g(p,t,\nu)$ for a real
parameter $p$, where $I$ is compact and
$g,\partial_pg,\partial_tg$ are jointly continuous.
Suppose $F(p_0,t_0)=0$ and $\partial_tg(p_0,t_0,\nu)>0$ at every
count maximizing $g(p_0,t_0,\cdot)$.  Upper semicontinuity of the maximizers lets us restrict
the maximum to a fixed compact neighbourhood of these maximizing counts on which,
after shrinking the parameter neighbourhood,
$0<c\le\partial_tg\le C$ and $|\partial_pg|\le L$.
Integrating these derivative bounds before taking maxima gives
\begin{equation}
 c(t'-t)\le F(p,t')-F(p,t)\le C(t'-t)\quad(t'\ge t),
 \qquad \bigl|F(p,t)-F(p',t)\bigr|\le L|p-p'|.
 \label{eq:compact-max-local-bounds}
\end{equation}
Continuity preserves the opposite signs at two fixed times on either
side of $t_0$ when $p$ is close to $p_0$.  The strict time bound therefore
gives a unique local zero $t=r(p)$.  Comparing two roots gives
\begin{equation}
 c\bigl|r(p)-r(p')\bigr|\le \Bigl|F\bigl(p,r(p')\bigr)-F\bigl(p',r(p')\bigr)\Bigr|\le L|p-p'|.
 \label{eq:compact-max-root-Lipschitz}
\end{equation}
Thus, $r$ is locally Lipschitz.

Write $D_H$ for the Hadamard directional derivative: the direction may
vary convergently as the positive step size tends to zero.
For a direction $(\alpha,\beta)$ in $(p,t)$, Danskin's theorem in the Hadamard form of
\cite[Theorem~4.1]{BonnansShapiro1998}, applied to $-g$, gives
\begin{equation}
 D_HF(p_0,t_0)(\alpha,\beta)
 =\max_{\substack{\nu\in I\\g(p_0,t_0,\nu)=0}}
    \left(\alpha\partial_pg+\beta\partial_tg\right),
 \label{eq:compact-max-Hadamard}
\end{equation}
with the derivatives evaluated at $(p_0,t_0,\nu)$.
The Lipschitz bound makes the root difference quotients bounded.
Take any subsequential limit $w$ along parameter directions converging
to $\alpha$.  Danskin's formula and $F(p,r(p))=0$ give
$\max_{\nu}(\alpha\partial_pg+w\partial_tg)=0$ over the same maximizing counts.
Since every $\partial_tg>0$, this maximum is zero exactly when $w$
is the smallest of the ratios $-\alpha\partial_pg/\partial_tg$.  Hence,
\begin{equation}
 D_Hr(p_0)(\alpha)
 =\min_{\substack{\nu\in I\\g(p_0,t_0,\nu)=0}}
       \frac{-\alpha\partial_pg}{\partial_tg}.
 \label{eq:compact-max-root-Hadamard}
\end{equation}
Every subsequential limit agrees.  In particular, the right
derivative is the minimum of $-\partial_pg/\partial_tg$, and the
conventional left derivative is its maximum.  At a boundary value of
$p$, use feasible directions; the same proof applies if $g$ has a smooth
extension to a parameter neighbourhood.

For an originally moving count domain, this argument applies once all
nearby maximizers lie in a fixed compact subset of its smooth interior.
One may take a finite union of small closed intervals around the
maximizing counts, each lying in that smooth interior.

\begin{lemma}
\label{lem:global-root-model-only}
Under Assumption~\ref{ass:environment}, there are constants
$0<u_{\min}<u_{\max}<\infty$ and
$c_{\rm rt},\eta_{\rm rt},\kappa_{\rm rt}>0$ such that, for every
$\theta\in[0,1]$, the root $u_*(\theta)$ is unique,
$u_{\min}\le u_*(\theta)\le u_{\max}$, and $u_*$ is Lipschitz.  For
$0\le v\le\eta_{\rm rt}$,
\begin{equation}
 \widehat F\bigl(\theta,u_*(\theta)-v\bigr)\le-c_{\rm rt}v,
 \qquad
 \widehat F\bigl(\theta,u_*(\theta)+v\bigr)\ge c_{\rm rt}v.       \label{eqtag:R0}
\end{equation}
Every maximizing saddle with
$|u-u_*(\theta)|\le\eta_{\rm rt}$ has
$\kappa\ge\kappa_{\rm rt}$.  In addition, for every $\varepsilon>0$,
\begin{equation}
 \sup_{\substack{\theta\in[0,1],\ \varepsilon\le u\le
                         u_*(\theta)-\varepsilon}}
 \widehat F(\theta,u)<0,                               \label{eqtag:R1}
\end{equation}
where an empty supremum is ignored.
\end{lemma}

\begin{proof}
Write $\ell_\theta(z)=\theta\log A_b(z)+(1-\theta)\log C_b(z)$.
We first locate all positive zeros.  Since
$H(\kappa)\le(\kappa+1-\lambda_+)^{-1}$, minimizing this upper bound
in the clock objective gives
\begin{equation}
 \chi(\nu,u)
 \le\nu\left(1+\log\frac u\nu\right)-(1-\lambda_+)u,
 \label{eq:clock-homogeneous-sandwich}
\end{equation}
with the first term interpreted as zero when $\nu=0$.
Evaluate the spatial infimum at $z=u$ and use
\eqref{eq:clock-homogeneous-sandwich}.  The count terms cancel, yielding
\begin{equation}
 \widehat F(\theta,u)
 \le\ell_\theta(u)-(1-\lambda_+)u.
 \label{eq:global-profile-sandwich}
\end{equation}
Since $A_b<C_b$, $C_b(0)=1$, and $C_b'(0)=0$, this upper bound is
negative for all sufficiently small positive $u$, uniformly in $\theta$.

For the large-time lower bound, $H(0)>1$, $H(\kappa)\to0$ as
$\kappa\to\infty$, and strict decrease give a unique $r>0$ such that
$H(r)=1$.  Here
$0<\mu_r=-H'(r)<\infty$.  At $\nu=u/\mu_r$, the clock objective is
minimized at $r$, so $\chi(u/\mu_r,u)=ru$.
The formulas in \eqref{eq:AC} give, uniformly in $\theta\in[0,1]$,
$$
 \ell_\theta(z)=z-\log b+O\!\left(e^{-\frac b{b-1}z}\right),
 \qquad
 J_\theta(z)=z+O\!\left(ze^{-\frac b{b-1}z}\right)
 \quad(z\to\infty).
$$
Consequently, the spatial minimizer at count $\nu$ satisfies
$z=\nu+o(1)$ uniformly in $\theta$ as $\nu\to\infty$.
Substitution in \eqref{eq:Psi} gives
$\Psi_b(\theta,\nu)=-\log b+o(1)$, uniformly in $\theta$.
For all sufficiently large $u$, the count $u/\mu_r$ is feasible, and
$$
 \widehat F(\theta,u)
 \ge ru+\Psi_b\!\left(\theta,\frac u{\mu_r}\right)
 =ru-\log b+o(1).
$$
This is uniformly positive for all sufficiently large $u$.
Thus, there are $0<u_{\min}<u_{\max}<\infty$ such that
\begin{equation}
 \widehat F(\theta,u)<0\quad(0<u\le u_{\min}),\qquad
 \widehat F(\theta,u)>0\quad(u\ge u_{\max}).
 \label{eq:global-root-end-signs}
\end{equation}
Continuity from Lemma~\ref{lem:variational-regularity} gives a positive
zero for each $\theta$, and every positive zero lies in
$[u_{\min},u_{\max}]$.

At every zero, Lemma~\ref{lem:variational-regularity} confines the
maximizing counts to a compact subset of $\nu>\theta$, with interior
clock minima and positive clock tilts.  Choose a fixed compact
neighbourhood of these counts, small enough that
$\partial_uG=\kappa^c>0$ throughout it.  Upper semicontinuity puts all
nearby maximizers in this neighbourhood.  The inner minimizers are
smooth there, and
\begin{equation}
 \partial_\theta G=\log\frac{A_b(z)}{C_b(z)},
 \qquad \partial_uG=\kappa.                            \label{eq:Genvelope}
\end{equation}
The positive margin $\nu-\theta$ preserves feasibility as $\theta$
varies.  At $\theta=0,1$, the implicit function theorem extends the
spatial minimizer smoothly beyond the endpoint, since $J_\theta'(z)>0$.
Thus, \eqref{eq:compact-max-local-bounds} and
\eqref{eq:compact-max-root-Lipschitz} apply with $p=\theta$ and $t=u$:
$\widehat F(\theta,\cdot)$ increases strictly in a neighbourhood of
every zero, and the local root is Lipschitz in $\theta$.

For fixed $\theta$, the zero set is compact by
\eqref{eq:global-root-end-signs}.  Each zero is isolated, so there are
only finitely many.  Two consecutive zeros are impossible: $\widehat F(\theta,\cdot)$
would be positive immediately after the first and negative immediately
before the second, forcing another zero between them by continuity.
The positive zero is therefore unique.  Its local Lipschitz dependence
now defines the continuous function $\theta\mapsto u_*(\theta)$; the endpoint signs imply
$\widehat F(\theta,u)<0$ for $0<u<u_*(\theta)$.

The root graph over $[0,1]$ is compact.  A finite cover by the preceding
neighbourhoods gives uniform positive time slopes on
$|u-u_*(\theta)|\le\eta_{\rm rt}$ for some $\eta_{\rm rt}>0$,
proving \eqref{eqtag:R0}.  The same cover confines the maximizing
counts and inner tilts to compact subsets of the interiors of their domains and gives
$\kappa\ge\kappa_{\rm rt}>0$ throughout this neighbourhood.  A finite subdivision
of $[0,1]$ subordinate to the cover turns the local root Lipschitz bounds
into a global one.  Finally, the set in \eqref{eqtag:R1} is compact and
$\widehat F$ is strictly negative on it, so its maximum is negative.
\end{proof}

We now show that the first-moment center $\widehat\theta_{d,m}$ differs
from $d u_*(m/d)$ by $O(\log d)$.  This identifies the critical time band used in the
quenched estimates and proves the first assertion of
Theorem~\ref{thm:sparse-expansion}. Recall \eqref{eq:thetahat definition}.

\begin{lemma}
\label{lem:independent-center-location}
For all sufficiently large $d$,
$$
 \sup_{1\le m\le d}
 \left|\widehat\theta_{d,m}
 -d u_*\!\left(\frac md\right)\right|
 \ll\log d.
$$
Moreover,
\begin{equation}
 \widehat M_{\widehat\theta_{d,m}}(\mathbf{x}_m)=1
 \qquad(1\le m\le d).                                \label{eq:independent-center-equality}
\end{equation}
\end{lemma}

\begin{proof}
Choose
$$
 0<\eta<\frac14\min\{u_{\min},\eta_{\rm rt}\}.
$$
The root estimates \eqref{eqtag:R0} and the polynomial kernel bounds
\eqref{eq:UO1}, applied on a fixed compact neighbourhood of the root graph
and integrated over one time unit, give constants $B,c>0$ such that,
uniformly in $1\le m\le d$ and $1\le a\le\eta d$,
\begin{equation}
 \widehat M_{d u_*\!\left(\frac{m}{d}\right)-a}(\mathbf{x}_m)\le d^B e^{-ca},
 \qquad
 \widehat M_{d u_*\!\left(\frac{m}{d}\right)+a}(\mathbf{x}_m)\ge d^{-B}e^{ca}.
 \label{eq:deterministic-center-sign-estimates}
\end{equation}

The local estimates do not yet rule out an earlier crossing of $\widehat M_{t}(\mathbf{x}_m)$ at one.  To do so,
we first bound the kernel at small times.
Since $h(r)\le e^{-(1-\lambda_+)r}$, \eqref{eq:freshK} and the exact spatial
generating function give
$$
 \widehat K_{d,m}(s)
 \le e^{-(1-\lambda_+)s}\sum_{n\ge0}p_{d,m}(n)\frac{s^n}{n!}
 =e^{-(1-\lambda_+)s}A_b\!\left(\frac{s}{d}\right)^mC_b\!\left(\frac{s}{d}\right)^{d-m}.
$$
For small $u$, $A_b(u)\le Cu$, $A_b(u)<C_b(u)$, and
$\log C_b(u)\le Cu^2$.  Choose $\varepsilon>0$ so small that
$$
 2\varepsilon<u_{\min}-\eta,
 \qquad C(2\varepsilon)\le \frac{1-\lambda_+}{2}.
$$
Then
\begin{equation*}
 \sup_{1\le m\le d}
 \int_0^{2\varepsilon d}\widehat K_{d,m}(s)\,\dd s
 \ll\frac1d\int_0^\infty s e^{-(1-\lambda_+)s/2}\,\dd s
 \ll d^{-1}.
\end{equation*}
For all sufficiently large $d$, this is less than one.  Hence,
$\widehat M_t(\mathbf{x}_m)<1$ for
$0\le t\le2\varepsilon d-1$.

For the remaining times before $du_*(m/d)-\eta d$, decrease
$\varepsilon$, if necessary, so that $\varepsilon<\eta$.
By \eqref{eqtag:R1} there is $c_\eta>0$ such that
$\widehat F\le-c_\eta$ on
$$
 \left\{(\vartheta,u):\vartheta\in[0,1],\
          \frac{\varepsilon}{2}\le u\le u_*(\vartheta)-\frac{\eta}{2}\right\}.
$$
For all sufficiently large $d$, the global upper bound
\eqref{eq:global-oneobject-upper} therefore gives
\begin{equation*}
 \sup_{\substack{1\le m\le d\\
       \varepsilon d\le t\le d u_*\!\left(\frac{m}{d}\right)-\eta d}}
 \widehat M_t(\mathbf{x}_m)
 \ll d^{3/2}e^{-c_\eta d}=o(1).
\end{equation*}
Choose a deterministic $C_{\rm ctr}>(B+2)/c$.  The first inequality in
\eqref{eq:deterministic-center-sign-estimates} covers the interval from
$d u_*(m/d)-\eta d$ to
$d u_*(m/d)-C_{\rm ctr}\log d$.  The three intervals overlap, so
$\widehat M_t(\mathbf{x}_m)<1$ before the latter time.  After increasing
$C_{\rm ctr}$ if necessary, the second inequality in
\eqref{eq:deterministic-center-sign-estimates} gives
$\widehat M_{d u_*(m/d)+C_{\rm ctr}\log d}(\mathbf{x}_m)>1$.
This proves the location bound in \eqref{eq:all-distance-center-log}.
The initial estimate gives $\widehat M_0(\mathbf{x}_m)<1$, while the
bound to the right of the root makes the superlevel set nonempty.
Continuity shows that its infimum is attained, and the value there is
exactly one.  This proves
\eqref{eq:independent-center-equality}.
\end{proof}

\section{Estimates for the independent-visit model}
\label{sec:bridge-fractional-second-moment}

This section supplies the independent-visit branching estimates used in the transfer to the static environment. Section~\ref{sec:split-control} bounds the fractional branching sum by treating early, interior, and late splits separately. Section~\ref{sec:initial-branch} then normalizes the nonterminal branching weights and controls the trunk duration and mutation count under the resulting probability law. These controls ensure that the continuation parameters remain in the regular region with exponentially high probability.

\subsection{Fractional branching estimates}
\label{sec:split-control}

Consider a genealogy in the independent-visit comparison system, from the
source to the target, and a possible branching event along it:
\[
 \mathbf x_m\xrightarrow[\text{time }r]{\text{trunk}}(v,i)
 \xrightarrow[\text{time }t-r]{\text{continuation}}\mathbf0.
\]
At the split, both children inherit the trunk's current signed rate $i$;
their later visits receive independent labels. The fractional branching sum
$\widehat{\mathcal B}_{t,r}^{(1+q)}$, defined in
\eqref{eq:fractional-branching-sum}, weights each split by its
trunk first-moment weight with the current rate $i$ retained, branching rate,
and the $(1+q)$-power of one child's expected
target count at time $t$.  We prove \eqref{eqtag:FD2}, which bounds this sum
by two terms that decay exponentially from the initial and terminal times.

We treat early, interior, and late splits separately.  To keep the
short-trunk window fixed while choosing the moment exponent, we first
fix a small auxiliary exponent $q_*>0$ and prove a short-trunk bound on
$0\le r\le\delta_I(q_*)d$.  H\"older interpolation with the semigroup
identity at $q=0$ makes this same time window available for every
$0<q\le q_*$.  The final exponent $q$ is chosen only after this window
and the terminal threshold $\delta_T>0$ have been fixed.
Thus, a short trunk or continuation may last a small fixed fraction of
$d$, not only bounded time.
\begin{center}
\begin{tabular}{p{.35\textwidth}|p{.57\textwidth}}
$0\le r\le\delta_I(q_*)d$
 & Retained counts and an explicit upper bound on the growth of a short trunk
   (Lemmas~\ref{lem:global-retained-arm},
   \ref{lem:setvalued-initial-arm-gap}, and
   \ref{lem:marked-short-time-tail}), followed by interpolation
   (Step~1).\\
$\delta_I(q_*)d\le r\le t-\delta_Td$
 & Concentration near the mean coordinate proportions associated with retained
   counts, and the entropy bound at those proportions
   (Lemma~\ref{lem:setvalued-fractional-gaps}; Step~2).\\
$t-\delta_Td\le r\le t$
 & Direct short-time continuation bound (Step~3).
\end{tabular}
\end{center}
The step numbers refer to the proof of
Theorem~\ref{thm:globalfractionaldiagram}.  For bounded trunk durations,
the polynomial prefactors in exponential-rate estimates are too large;
the explicit upper bound on trunk growth controls
this range of trunk durations; the uniform negative bound on the
continuation exponent applies throughout the middle interval.  For splits
with $t-\delta_Td\le r\le t$, we use a separate direct argument because
the fractional power is applied only to the continuation.

The same three ranges appear in the second-moment proof
in~\cite[Proposition~24, Appendix~C.1.2]{BZ}.
The integration variable there is the continuation duration $t-r$, so
the two endpoint ranges occur in the opposite order.  Here the
{square of the expected target count} is replaced by its $(1+q)$-power for small
$q>0$, the split weight retains the current branching rate, and
the estimates cover every distance density rather than the restricted
distance range of the homogeneous result.
Related branching-process and overlap arguments appear in Martinsson's
study of unoriented first-passage percolation on the
hypercube~\cite{MartinssonUnoriented}.

We first describe the complete, unsplit first-moment kernel.
Write $\theta=m/d$ and $u=t/d$.  A count density $\nu$ represents
$n\approx\nu d$ mutations over the full time $ud$.  At a count with
interior spatial and clock minimizers, the tilts are determined by
\[
 J_\theta(z)=\nu,\qquad
 -\frac{H'(\kappa)}{H(\kappa)}=\frac{u}{\nu}.
\]
For the prescribed endpoint distance, the first equation sets the tilted
count density to $\nu$; the second sets the mean holding time per mutation
to $u/\nu$.  We then maximize $G(\theta,u,\nu)$
over $\nu$.  Together, these choices give a maximizing saddle as defined in
Definition~\ref{def:maximizing-saddle}.  We call it a \emph{full saddle}
when distinguishing the complete source-to-target path from its parts
before and after a split.  This deterministic saddle describes the paths
that dominate the first moment.
Several count densities may maximize $G$, although both interior tilts
are unique for a fixed $\nu$.  The root neighbourhood used here lies
in the regular region of Lemma~\ref{lem:variational-regularity}; away
from that region the clock infimum may occur at $\lambda_+-1$.
For the one-coordinate bridges conditioned on their prescribed initial
and terminal values, at a fraction $\alpha\in(0,1)$ of the bridge duration, the mean coordinate
proportions defined below give a trunk count density $\alpha\nu$ with
inner minimizers $\alpha z,\kappa$.

\begin{samepage}
The notation introduced below has the following roles:
\begin{center}
\begin{tabular}{p{.23\textwidth}p{.69\textwidth}}
$\mathcal R_*$ & all full saddles at the critical time $u_*(\theta)$;\\
$\mathbf r=(r_C,r_y,r_x,r_o)$ & coordinate proportions of the split vertex;\\
$\eta(\mathbf r),\xi(\mathbf r)$ & trunk and continuation distances divided by $d$;\\
$a_n$ & ratio $A_b(z_n)/C_b(z_n)$, where $z_n$ is the spatial tilt for retained count $n$;\\
$R_c$ & row sum after conjugating the mutation operator by a radial weight;\\
$\widehat{\mathcal B}_{t,r}^{(1+q)}$ & the fractional branching sum to be bounded.
\end{tabular}
\end{center}
\end{samepage}

\subsubsection{Saddle estimates at the critical time}

Decrease the fixed positive root-neighbourhood width $\eta_{\rm rt}$,
if necessary, so that the following compact set lies in the regular
region of Lemma~\ref{lem:variational-regularity}, and put
$$
 \mathcal K_{\rm rt}
 =\Bigl\{(\theta,u):\theta\in[0,1],\ \bigl|u-u_*(\theta)\bigr|\le\eta_{\rm rt}\Bigr\}.
$$

For $\nu>\theta$, let $z_{\theta,\nu}$ denote the unique spatial
minimizer, and let $\kappa_{u,\nu}$ denote the clock minimizer,
allowing the boundary value $\lambda_+-1$ when its infimum is attained
there.  On the maximizing sets considered in this regular neighbourhood,
both minimizers are interior.  Define
\begin{equation}
\begin{split}
 \mathcal A(\theta,u)
 &=\bigl\{(\nu,z_{\theta,\nu},\kappa_{u,\nu}):\nu>\theta,
              G(\theta,u,\nu)=\widehat F(\theta,u)\bigr\},\\
 \mathcal R_*
 &=\Bigl\{(\theta,\nu,z,\kappa):
              (\nu,z,\kappa)\in\mathcal A\bigl(\theta,u_*(\theta)\bigr)\Bigr\}.
\end{split}
\label{eq:root-active-graph}
\end{equation}
At $u=u_*(\theta)$ the exponential rate $\widehat F(\theta,u)$ is zero.  The set $\mathcal R_*$ collects
all maximizing saddles at these critical times, for all
$\theta$.  There may be more than one for a given $\theta$.

The estimates below are first proved at these critical saddles, but we
will apply them to retained counts at times within a logarithmic distance
of $d u_*(m/d)$.  The next lemma places the parameters of every sufficiently
near-maximizing count close to the set of critical saddles.  This lets us
use inequalities proved on that set, such as the gap in
Lemma~\ref{lem:setvalued-initial-arm-gap}, for the retained counts, provided
their exponential rates are sufficiently close to the maximum.
The final assertion accounts for the terminal holding interval in the
exact clock tilt.

\begin{lemma}
\label{lem:root-active-stability}
The set $\mathcal R_*$ is compact and lies in fixed compact subsets of
the interiors of the count, spatial-tilt, and clock-tilt domains.  For
every open neighbourhood $U$ of $\mathcal R_*$ there are
$\eta_U,\varepsilon_U>0$ such that, for
$\theta\in[0,1]$, $u>0$, and $\nu>\theta$,
\begin{equation}
 \bigl|u-u_*(\theta)\bigr|\le\eta_U,\qquad
 G(\theta,u,\nu)\ge\widehat F(\theta,u)-\varepsilon_U               
 \label{eq:root-active-stability-hyp}
\end{equation}
imply
\begin{equation}
 (\theta,\nu,z_{\theta,\nu},\kappa_{u,\nu})\in U.      
 \label{eq:root-active-stability-conclusion}
\end{equation}
The conclusion remains true with the finite-$d$ clock tilt defined using
$t/(n+1)$, with $\theta=m/d$, $u=t/d$, and $\nu=n/d$, after $d$ is
sufficiently large.

\end{lemma}

\begin{proof}
On a compact neighbourhood of the root graph contained in the regular
region, Lemma~\ref{lem:variational-regularity} and the retained-count
confinement in Proposition~\ref{prop:global-oneobject} place every
near-maximizing count and its two inner minimizers in fixed compact
interior sets.  By continuity, the set $\mathcal R_*$ defined in
\eqref{eq:root-active-graph} is closed, which proves compactness.

Suppose the implication from \eqref{eq:root-active-stability-hyp} to
\eqref{eq:root-active-stability-conclusion} failed.  There would be
sequences with $u_j-u_*(\theta_j)\to0$ and
$$
 \widehat F(\theta_j,u_j)-G(\theta_j,u_j,\nu_j)\to0
$$
whose associated tuples stay outside $U$.  Compactness gives a convergent
subsequence, and continuity places its limit in
$\mathcal R_*$, a contradiction.  The finite-$d$ tilt differs from
$\kappa_{u,n/d}$ by $O(d^{-1})$, uniformly, because it replaces
$u/(n/d)$ by $u/(n/d+d^{-1})$ in the equation
$-H'/H=\cdot$.  On the compact tilt interval,
$$
 \frac{\dd}{\dd\kappa}\left(-\frac{H'(\kappa)}{H(\kappa)}\right)
 =-\operatorname{Var}_{f_\kappa}(R)
$$
is bounded away from zero in absolute value, so the inverse map is
uniformly Lipschitz.  Apply the preceding compactness argument first to
a smaller neighbourhood whose closure lies in $U$; the $O(d^{-1})$
perturbation then remains in $U$ for all sufficiently large $d$.
\end{proof}

For a split vertex $v$, $r_C$ is the proportion, among all $d$
coordinates, on which the start and target agree but $v$ differs.
Among initially mismatching coordinates, $r_y$ records those at
the target value, $r_x$ those at the starting value, and $r_o$
those at a third value.  All four proportions are normalized by $d$.
The last class is absent when $b=2$.

\begin{figure}[htbp]
\centering
$\mathbf x_m\xrightarrow{\text{trunk}}v
 \xrightarrow{\text{continuation}}\mathbf0$
\medskip

\begin{tabular}{c|ccc|cc}
proportion & start & at $v$ & target & contributes to $\eta$ & contributes to $\xi$\\
\hline
$r_C$ & $0$ & nonzero & $0$ & yes & yes\\
$r_y$ & nonzero & $0$ & $0$ & yes & no\\
$r_x$ & nonzero & starting value & $0$ & no & yes\\
$r_o$ & nonzero & third value & $0$ & yes & yes
\end{tabular}
\caption{Coordinate classes at a split.  An initially matching
coordinate still equal to $0$ contributes to neither distance.}
\label{fig:midpoint-coordinate-classes}
\end{figure}

For $\theta\in[0,1]$, the possible coordinate proportions form
\begin{equation*}
 \mathfrak P_\theta
 =\left\{(r_C,r_y,r_x,r_o):
 \begin{array}{l}
 0\le r_C\le1-\theta,\quad r_y,r_x,r_o\ge0,\\[-2pt]
 r_y+r_x+r_o=\theta,\quad r_o=0\ \text{if }b=2
 \end{array}\right\}.
\end{equation*}
For $\mathbf r=(r_C,r_y,r_x,r_o)\in\mathfrak P_\theta$, put
\begin{gather*}
 \eta(\mathbf r)=r_C+r_y+r_o,\qquad
 \xi(\mathbf r)=r_C+r_x+r_o,\\
 \mathcal S_{b,\theta}(\mathbf r)
 =(1-\theta)h_2\!\left(\frac{r_C}{1-\theta}\right)
   +r_C\log(b-1)-\sum_{\zeta\in\{y,x,o\}}r_\zeta
                  \log\frac{r_\zeta}{\theta}
   +r_o\log(b-2),
\end{gather*}
where $h_2(s)=-s\log s-(1-s)\log(1-s)$.  Terms from a zero-size
coordinate class are omitted, and $0\log0=0$; in particular,
$r_o\log(b-2)$ is omitted when $b=2$.
Thus, $\eta(\mathbf r)=d_{\rm H}(\mathbf x_m,v)/d$ is the trunk
distance, and $\xi(\mathbf r)=d_{\rm H}(v,\mathbf0)/d$ is the
continuation distance.  This coordinate classification refines the
split-vertex count in~\cite[Lemma~23]{BZ}.
Applying Stirling's formula to the corresponding binomial and
multinomial coefficients gives $\mathcal S_{b,\theta}$, which counts
vertices with the prescribed coordinate proportions on the exponential scale.

For $0\le\delta\le1$ and $s>0$, write
$\ell_\delta(s)=\delta\log A_b(s)+(1-\delta)\log C_b(s)$.
For a tuple $\mathfrak s=(\nu,z,\kappa)$ with $z>0$ and
$0\le\alpha\le1$, define the mean coordinate proportions
$\mathbf r^\alpha(\mathfrak s)\in\mathfrak P_\theta$ by
\begin{equation}
\begin{split}
 r_C^\alpha&=(1-\theta)
       \frac{(b-1)A_b(\alpha z)A_b\bigl((1-\alpha)z\bigr)}{C_b(z)},
 \qquad
 r_y^\alpha=\theta\frac{A_b(\alpha z)C_b\bigl((1-\alpha)z\bigr)}{A_b(z)},\\
 r_x^\alpha&=\theta\frac{C_b(\alpha z)A_b\bigl((1-\alpha)z\bigr)}{A_b(z)},
 \qquad
 r_o^\alpha=\theta\frac{(b-2)A_b(\alpha z)A_b\bigl((1-\alpha)z\bigr)}{A_b(z)}.
\end{split}
\label{eqtag:MR2}
\end{equation}
These are the mean coordinate proportions at time $\alpha z$ under
independent one-coordinate bridges of duration $z$ with the prescribed
source and target symbols.  The formula depends on $\theta,z,\alpha$;
the tuple also records the count and clock tilt when these are used.
The factors $b-1$ and $b-2$ count intermediate symbols.
The one-coordinate semigroup identities show that the symbol probabilities
sum to one in each initial coordinate class, including the cases $\theta=0,1$.

For a maximizing saddle at the critical time,
$\mathfrak s=(\nu,z,\kappa)\in\mathcal A(\theta,u_*(\theta))$, put
\begin{equation}
 B_\theta(\alpha;z)
 =\ell_{\xi(\mathbf r^\alpha(\mathfrak s))}\bigl((1-\alpha)z\bigr)
  -(1-\alpha)\ell_\theta(z),\qquad 0\le\alpha<1.
 \label{eq:canonical-continuation-exponent}
\end{equation}
The quantity $B_\theta(\alpha;z)$ is the continuation log weight minus
$(1-\alpha)$ times the full log weight. The following bound supplies both the middle-range estimate and the
derivative at $\alpha=0$ needed for a short initial trunk.

\begin{lemma}
\label{lem:setvalued-fractional-gaps}
Under Assumption~\ref{ass:environment}, there is $c_0>0$ such that,
uniformly over $\mathcal R_*$ and $0\le\alpha<1$,
\begin{equation}
 B_\theta(\alpha;z)\le-c_0\alpha(1-\alpha).
 \label{eq:root-canonical-continuation-gap}
\end{equation}
For $0<\alpha<1$, the trunk count density $\alpha\nu$ has spatial
and clock minimizers $\alpha z,\kappa$, and
\begin{equation}
 \mathcal S_{b,\theta}\bigl(\mathbf r^\alpha(\mathfrak s)\bigr)
 +G\Bigl(\eta\bigl(\mathbf r^\alpha(\mathfrak s)\bigr),\alpha u_*(\theta),\alpha\nu\Bigr)
 =-B_\theta(\alpha;z).
 \label{eq:canonical-prefix-entropy}
\end{equation}
\end{lemma}

\begin{proof}
We first prove the identity \eqref{eq:canonical-prefix-entropy},
then apply the relative-entropy bound for one-coordinate bridges in
Lemma~\ref{lem:bridge-entropy-convexity} to establish
\eqref{eq:root-canonical-continuation-gap}. Under a one-coordinate bridge, the logarithm of the probability of a symbol at the split
is the sum of the two segment log weights minus the full log weight.
Taking expectations, and summing over the two initial coordinate
classes, gives
\begin{equation}
 \mathcal S_{b,\theta}\bigl(\mathbf r^\alpha(\mathfrak s)\bigr)
 +\ell_{\eta(\mathbf r^\alpha(\mathfrak s))}(\alpha z)
 +\ell_{\xi(\mathbf r^\alpha(\mathfrak s))}\bigl((1-\alpha)z\bigr)
 =\ell_\theta(z).
 \label{eq:canonical-Shannon-identity}
\end{equation}
Indeed, each symbol in the matching class has weight
$C_b(\alpha z)C_b((1-\alpha)z)/C_b(z)$ or
$A_b(\alpha z)A_b((1-\alpha)z)/C_b(z)$.
For an initially mismatching coordinate, the probability of each symbol
(the coordinate's value) at the cut is the product of the two
corresponding segment weights, divided by $A_b(z)$. In \eqref{eqtag:MR2}, the factor $\theta$ is the proportion of initially
mismatching coordinates, and $b-2$ counts the possible symbols distinct
from both the source and target symbols.
Weighting the two class entropies by $1-\theta$ and $\theta$ gives
$\mathcal S_{b,\theta}$.

Conditional on the total jump count of a constant-rate coordinate walk
and its endpoints, the jump times are uniform on the time interval.
The expected number of jumps before the cut is therefore $\alpha$ times the expected
total count.  Conditioning instead on the symbol at the split gives
\begin{equation}
 J_{\eta(\mathbf r^\alpha(\mathfrak s))}(\alpha z)
 =\alpha J_\theta(z)=\alpha\nu.
 \label{eq:canonical-prefix-count}
\end{equation}
The mean map is strictly increasing, so $\alpha z$ is the spatial
minimizer for this trunk count density.  The full clock equation gives
$-H'(\kappa)/H(\kappa)=u_*(\theta)/\nu
=(\alpha u_*(\theta))/(\alpha\nu)$, so $\kappa$ is its clock minimizer.
At the maximizing count density $\nu$, the stationarity condition
\eqref{eq:UO8} gives $z=\nu H(\kappa)$.
Substituting this identity in the definition of $G$ yields
\begin{equation}
 G\Bigl(\eta\bigl(\mathbf r^\alpha(\mathfrak s)\bigr),\alpha u_*(\theta),\alpha\nu\Bigr)
 =\ell_{\eta(\mathbf r^\alpha(\mathfrak s))}(\alpha z)
  -\alpha\nu+\alpha\kappa u_*(\theta).
 \label{eq:canonical-prefix-exponent}
\end{equation}
Combine \eqref{eq:canonical-Shannon-identity} and
\eqref{eq:canonical-prefix-exponent} with the root identity
$\ell_\theta(z)-\nu+\kappa u_*(\theta)=0$ to obtain
\eqref{eq:canonical-prefix-entropy}.

Let $D_\sigma(s;z)$ be the relative entropy of the position at time $s$
under a one-coordinate bridge of duration $z$, relative to its law
under the unconditioned walk with the same initial symbol.
Here $\sigma=b-1$ denotes matching endpoints and $\sigma=-1$ denotes
mismatching endpoints, as in Lemma~\ref{lem:bridge-entropy-convexity} below.
The coordinate transition probabilities are $\exp(-s)C_b(s)$ for equal
symbols and $\exp(-s)A_b(s)$ for different symbols.
Thus, the definition of relative entropy gives
\begin{equation}
B_\theta(\alpha;z)
 ={}(1-\theta)\bigl(D_{b-1}(\alpha z;z)-\alpha D_{b-1}(z;z)\bigr)
 +\theta\bigl(D_{-1}(\alpha z;z)-\alpha D_{-1}(z;z)\bigr).
\label{eq:canonical-coordinate-entropy}
\end{equation}
The $z$-coordinates of $\mathcal R_*$ lie in a compact subinterval of
$(0,\infty)$.  Applying Lemma~\ref{lem:bridge-entropy-convexity} below to
\eqref{eq:canonical-coordinate-entropy} proves
\eqref{eq:root-canonical-continuation-gap}, including $\alpha=0$ by
continuity.  This completes both assertions.
\end{proof}

\subsubsection{Estimates for early splits}

For $0\le r\le\delta_I(q)d$, we keep the near-optimal continuation
counts and derive an explicit upper bound on the trunk's growth.  The retained
continuation estimate is stated first; the subsequent {gap estimate} and
mutation-count tail bound with exponential weights give the short-trunk estimate in Step~1 of
Theorem~\ref{thm:globalfractionaldiagram}.

Here a type $i\in(-\infty,\lambda_+]$ is the sampled signed rate.
Throughout its visit,
a particle branches at rate $i_+$ and dies at rate $(-i)_+$.
Each mutation starts a new
visit with an independent rate drawn from the probability measure $P$, including a return to an earlier
vertex.  At a split, both children inherit the current
label $i$.  Conditional on the split data, their subsequent evolutions
are independent.
Thus, the trunk's terminal rate is $i$, the branching factor is $i_+$,
and each continuation's first interval of duration $s$ contributes the
first-moment weight $e^{-(1-i)s}$.  These factors use the same $i$ before
integration.  On
$\Hc_d^{(b)}\times(-\infty,\lambda_+]$, let $\mathbf K_t$ be the positive
first-moment kernel obtained by weighting the rate-one mutation process,
with a fresh type at each mutation, by $\exp(\int_0^t I_s\,\dd s)$.
Its generator, on its domain, acts by
\begin{equation*}
 f(x,i)\longmapsto
 \bigl((b-1)d\bigr)^{-1}\sum_{y\sim x}\int f(y,j)\,P(dj)
 -(1-i)f(x,i).
\end{equation*}
The notation $\mathbf K_t((x,i),\{y\}\times\dd j)$ denotes a measure in
the terminal rate.  Set
\begin{equation}
\begin{split}
 \mathbf K_t^P\bigl(x;(y,j)\bigr)\,P(dj)
 &=\int\mathbf K_t\bigl((x,i),\{y\}\times\dd j\bigr)\,P(di),\\
 \mathbf k_t\bigl((y,j),z\bigr)
 &=\mathbf K_t\bigl((y,j),\{z\}\times(-\infty,\lambda_+]\bigr).
\end{split}
\label{eq:typed-kernel-definitions}
\end{equation}
The kernel $\mathbf K_t^P$ averages the initial rate with respect to $P$
and retains the terminal rate as an argument.  The kernel $\mathbf k_t$
fixes the initial rate and integrates over all terminal rates.
The first line defines a density with respect to $P$, whose explicit
version below is defined for every $j\in(-\infty,\lambda_+]$.
Summing the jump-count expansion over the rate labels gives
\begin{equation*}
 \widehat K_t(x,y)
 =\int\mathbf K_t^P\bigl(x;(y,j)\bigr)\,P(dj)
 =\int\mathbf k_t\bigl((x,i),y\bigr)\,P(di).
\end{equation*}
Detailed balance gives
\begin{equation}
 P(di)\mathbf K_t\bigl((x,i),\{y\}\times\dd j\bigr)
 =P(dj)\mathbf K_t\bigl((y,j),\{x\}\times\dd i\bigr),
 \qquad
 \mathbf K_t^P\bigl(x;(y,i)\bigr)=\mathbf k_t\bigl((y,i),x\bigr).
 \label{eq:typed-detailed-balance}
\end{equation}

Hereafter, $h^{*0}$ denotes the unit point mass at zero, so
$f*h^{*0}=f$.  For $k\ge0$, set
\begin{equation*}
 g_{k,i}=e^{-(1-i)\,\cdot}*h^{*k}.
\end{equation*}
The first holding interval carries the fixed rate label $i$; the
remaining $k$ holding intervals have independent rate labels drawn
from $P$.  Averaging over these fresh rate labels gives\begin{equation}
 \mathbf k_t\bigl((v,i),y\bigr)
 =\sum_{k\ge0}p_{d,d_{\mathrm H}(v,y)}(k)g_{k,i}(t).
 \label{eq:typed-jump-count}
\end{equation}

Apply Propositions~\ref{prop:global-oneobject} and
\ref{prop:global-oneobject-local} to $\mathcal K_{\rm rt}$ using the
same count interval and threshold $\varepsilon_{\mathcal K_{\rm rt}}$ for both propositions, and denote the resulting count
interval by $J_{\rm rt}$.  Enlarge this interval, if necessary, so that
all sufficiently near-maximizing count densities on $\mathcal K_{\rm rt}$
lie a fixed positive distance from its endpoints.  Fix
$0<\varepsilon_{\rm ret}\le\varepsilon_{\mathcal K_{\rm rt}}$
sufficiently small.  The exact expansion
\[
 \widehat K_{d,m}(t)=\sum_{n\ge0}p_{d,m}(n)h^{*(n+1)}(t)
\]
decomposes the kernel by mutation count.  We retain admissible counts with
$n/d\in J_{\rm rt}$ whose exponential rates are at most
$\varepsilon_{\rm ret}$ below the maximum.
For $(m/d,t/d)\in\mathcal K_{\rm rt}$, specialize Definition~\ref{def:retained-counts} by putting
\begin{equation}
 \mathcal I_d(m,t)
 =\mathcal I_d^{J_{\rm rt},\varepsilon_{\rm ret}}(m,t)
 =\left\{n\ge0:
 \begin{array}{l}
 p_{d,m}(n)>0,\quad n/d\in J_{\rm rt},\\
 G\!\left(\frac md,\frac td,\frac nd\right)
 \ge\widehat F\!\left(\frac md,\frac td\right)-\varepsilon_{\rm ret}
 \end{array}\right\}.
\label{eqtag:GR0}
\end{equation}
Choose $\varepsilon_{\rm ret}$ so that the retained tilts lie in the compact region of
Proposition~\ref{prop:global-oneobject-local}.  Every sufficiently small
fixed positive value of $\varepsilon_{\rm ret}$ gives the estimates below.  For each fixed exponent $q$
used in a short-trunk estimate, choose $\varepsilon_{\rm ret}$ after fixing $q$ so
that, for fixed $C_0<\infty$ and all sufficiently large $d$, every retained count
with $|t-du_*(m/d)|\le C_0\log d$ satisfies
\eqref{eq:finite-initial-transform-gap} of Lemma~\ref{lem:setvalued-initial-arm-gap}.  At each such
choice, Proposition~\ref{prop:global-oneobject} bounds the omitted
contribution by an exponentially small fraction of the kernel.
For $n\in\mathcal I_d(m,t)$, let $z_n,\kappa_n$ be the unique solutions
\begin{equation}
 J_{\frac{m}{d}}(z_n)=\frac nd,
 \qquad -\frac{H'(\kappa_n)}{H(\kappa_n)}=\frac{t}{n+1},              \label{eqtag:GR0a}
\end{equation}
and put
$$
 \pi_n=\frac{p_{d,m}(n)h^{*(n+1)}(t)}
              {\widehat K_{d,m}(t)},\qquad
 a_n=\frac{A_b(z_n)}{C_b(z_n)}.
$$
The factor $a_n^{\ell-m}$ appears in the spatial coefficient ratio when
the continuation distance changes from $m$ to $\ell$ at this retained
tilt.  Equivalently, it conjugates the mutation operator by distance
from the target; this is its role in the short-trunk estimate below.
The choice of $J_{\rm rt}$ and \eqref{eq:UO2} imply
$\sum_{n\in\mathcal I_d(m,t)}\pi_n\le1$, with an exponentially small
deficit from one.  Every retained count is admissible by definition; for
$b=2$ this enforces the span-two parity lattice.  Since
$n/d$ stays bounded away from zero, all parity-adjusted counts used below
are nonnegative for sufficiently large $d$.

The next estimate separates retained continuation counts from a remainder while keeping the dependence on the initial rate label $i$ explicit. The remainder is exponentially small when the removed time and the change in distance are sufficiently small fractions of $d$. Here $s$ is the time removed from the continuation;
we will apply it with $s=r$, the trunk duration.

\begin{lemma}
\label{lem:global-retained-arm}
Under Assumption~\ref{ass:environment}, for every fixed $C_0<\infty$
there are $\delta_*,c>0$ such that, if $m/d=\theta$,
$|t-du_*(\theta)|\le C_0\log d$,
$0\le s\le\min(t,\delta_*d)$, and $\ell=d_{\mathrm H}(v,\mathbf 0)$, define
\begin{align*}
 \mathcal M_{s,\ell,i}
 &=\frac1{\widehat K_{d,m}(t)}
   \sum_{n\in\mathcal I_d(m,t)}
       p_{d,\ell}\!\left(n-\one_{\{b=2,\ \ell\not\equiv m\pmod2\}}\right)
       g_{n-\one_{\{b=2,\ \ell\not\equiv m\pmod2\}},i}(t-s),\\
 \mathcal E_{s,\ell,i}
 &=\frac1{\widehat K_{d,m}(t)}
   \sum_{\substack{k\ge0:\ k\notin
       \{n-\one_{\{b=2,\ \ell\not\equiv m\pmod2\}}:
          n\in\mathcal I_d(m,t)\}}}
       p_{d,\ell}(k)g_{k,i}(t-s).
\end{align*}
Then, uniformly for $i\le\lambda_+$,
\begin{align}
 \frac{\mathbf k_{t-s}\bigl((v,i),\mathbf 0\bigr)}{\widehat K_{d,m}(t)}
 &=\mathcal M_{s,\ell,i}+\mathcal E_{s,\ell,i},
                                                               \label{eqtag:GR4}\\
 \mathcal M_{s,\ell,i}
 &\ll\sum_{n\in\mathcal I_d(m,t)}\pi_n
 \frac{e^{-\kappa_ns}a_n^{\ell-m}}
 { (\kappa_n+1-i)
   H(\kappa_n)^{1+\one_{\{b=2,\ \ell\not\equiv m\pmod2\}}}},
                                                               \label{eqtag:GR5}\\
 \mathcal E_{s,\ell,i}
 &\le d^C\exp\Bigl(-cd+C\bigl(s+|\ell-m|\bigr)\Bigr).           \label{eqtag:GR6}
\end{align}
For every open neighbourhood $U$ of the set $\mathcal R_*$,
$\varepsilon_{\rm ret}$ may be reduced (depending on $U$) so
that, throughout $|t-du_*(m/d)|\le C_0\log d$ and for all sufficiently
large $d$, every tuple
$(m/d,n/d,z_n,\kappa_n)$ belonging to the retained set lies in $U$.
All counts removed by this reduction are absorbed into \eqref{eqtag:GR6}:
more precisely, after the reduction \eqref{eqtag:GR6} holds with positive
constants $c_U,\delta_U$ in place of $c,\delta_*$.  These constants may
depend on $U$ and are uniform after $U$ is fixed.
\end{lemma}

\begin{proof}
For a retained denominator count $n$, use its exact spatial and clock
tilts $z_n,\kappa_n$ from \eqref{eqtag:GR0a}.
The shift maps the binary span-two lattice for the denominator to the
corresponding numerator lattice.  Exact coefficient tilting gives
\begin{equation*}
 \frac{p_{d,\ell}\bigl(n-\one_{\{b=2,\ \ell\not\equiv m\pmod2\}}\bigr)}{p_{d,m}(n)}
 =\begin{cases}
   \displaystyle
   a_n^{\ell-m}
   \frac{\Pp_{z_n}(\Sigma_{d,\ell,z_n}=n)}
        {\Pp_{z_n}(\Sigma_{d,m,z_n}=n)},
   &b>2\text{ or }\ell\equiv m\pmod2,\\[3mm]
   \displaystyle
   \frac{d z_n}{n}a_n^{\ell-m}
   \frac{\Pp_{z_n}(\Sigma_{d,\ell,z_n}=n-1)}
        {\Pp_{z_n}(\Sigma_{d,m,z_n}=n)},
   &b=2\text{ and }\ell\not\equiv m\pmod2.
  \end{cases}
\end{equation*}
The factorial and degree terms cancel except for the bounded binary
parity correction.  Proposition~\ref{prop:global-oneobject-local} bounds
the numerator local probability by $Cd^{-1/2}$ and the denominator,
evaluated at its mean, below by $cd^{-1/2}$.  Their ratio is therefore
bounded independently of $d$.

Applying the Laplace-tilting identity to the clock convolution at
$\kappa_n$ gives the exact exponential factor
\begin{equation*}
 e^{-\kappa_ns}
 \frac{H(\kappa_n)^{n-\one_{\{b=2,\ \ell\not\equiv m\pmod2\}}}}
 {(\kappa_n+1-i)H(\kappa_n)^{n+1}}
 =\frac{e^{-\kappa_ns}}
 { (\kappa_n+1-i)
   H(\kappa_n)^{1+\one_{\{b=2,\ \ell\not\equiv m\pmod2\}}}}.
\end{equation*}
Under the $\kappa_n$-tilt, the first holding factor, with fixed initial type $i$, becomes the
probability density
$(\kappa_n+1-i)e^{-(\kappa_n+1-i)\,\cdot}$.
Convolution with a probability density cannot increase the supremum norm.
Consequently, for either retained numerator count $k\asymp d$,
\[
 \Bigl\|\bigl((\kappa_n+1-i)
 e^{-(\kappa_n+1-i)\,\cdot}\bigr)*f_{\kappa_n}^{*k}
 \Bigr\|_\infty
 \le\bigl\|f_{\kappa_n}^{*k}\bigr\|_\infty\ll k^{-1/2}\asymp d^{-1/2},
\]
by Proposition~\ref{prop:global-oneobject-local}.  The same proposition
bounds the denominator clock density at its mean below by a constant
times $d^{-1/2}$.
The compact retained-tilt region gives $\kappa_n+1-i\ge c>0$ and
$H(\kappa_n)\ge c>0$, uniformly over retained $n$ and $i\le\lambda_+$.
Hence, the reciprocal denominator in \eqref{eqtag:GR5} is uniformly bounded.
Multiply the spatial and clock ratios by $\pi_n$ and sum over the
retained counts to obtain \eqref{eqtag:GR5}.  The partition into retained
and discarded counts gives \eqref{eqtag:GR4}.

We prove the discarded-count bound by showing that, when the changes in time and distance are sufficiently small, every sufficiently near-optimal continuation count belongs, after the required parity adjustment, to the retained set for the full path.
First restrict to $s+|\ell-m|\le\delta_*d$, choosing $\delta_*>0$
small enough that both parameter pairs belong to $\mathcal K_{\rm rt}$.
If
$k\in\mathcal I_d^{J_{\rm rt},\varepsilon_{\rm ret}/4}(\ell,t-s)$,
the bounds in \eqref{eq:UO14} place $k/d$ a fixed
positive distance from the geodesic boundary and from the endpoints of
$J_{\rm rt}$.  Put
$n=k+\one_{\{b=2,\ \ell\not\equiv m\pmod2\}}$.
For sufficiently small $\delta_*$ and all large $d$, this count lies in
$dJ_{\rm rt}$ and is admissible for distance $m$: its distance from the
geodesic boundary remains positive, and the binary parity is correct.
Uniform continuity of $G$ on this compact interior region and of
$\widehat F$ on $\mathcal K_{\rm rt}$ gives
\[
 G\!\left(\frac md,\frac td,\frac nd\right)
 \ge\widehat F\!\left(\frac md,\frac td\right)
       -\varepsilon_{\rm ret}
\]
after a further decrease of $\delta_*$.  Thus $n\in\mathcal I_d(m,t)$.
Every count contributing to $\mathcal E_{s,\ell,i}$ therefore lies
outside the displayed retained set for $(\ell,t-s)$.

To apply the discarded-count bound \eqref{eq:UO2} for the rate-averaged
kernel, we first compare the holding factor with fixed initial rate
label $i$ to its average $h$ under $P$. Fix a small $\varepsilon>0$.  Since
$P([\lambda_+-\varepsilon,\lambda_+])>0$, for every
$i\le\lambda_+$ and $r\ge0$,
\[
 e^{-(1-i)r}
 \le \frac{e^{\varepsilon r}}{P\bigl([\lambda_+-\varepsilon,\lambda_+]\bigr)}h(r).
\]
Positivity of convolution therefore gives, uniformly in $i$,
\[
 g_{k,i}(t-s)
 \le \frac{e^{\varepsilon(t-s)}}{P\bigl([\lambda_+-\varepsilon,\lambda_+]\bigr)}
          h^{*(k+1)}(t-s).
\]
Apply \eqref{eq:UO2} with $\varepsilon_{\rm ret}/4$.
Its discarded-count gap is fixed before
$\varepsilon$ is chosen.  Because $t-s\le Cd$, taking
$\varepsilon$ sufficiently small absorbs the extra exponential factor
into half of that gap and yields
\[
 \mathcal E_{s,\ell,i}
 \le C e^{-cd}
       \frac{\widehat K_{d,\ell}(t-s)}{\widehat K_{d,m}(t)}.
\]
The two-sided estimate \eqref{eq:UO1} and uniform continuity of
$\widehat F$ bound this ratio by
$Cd\exp(d[\widehat F(\ell/d,(t-s)/d)-\widehat F(m/d,t/d)])$.
Shrink $\delta_*$ once more so that this exponential variation is less
than half the fixed gap above.  Absorbing the polynomial factor then
gives $\mathcal E_{s,\ell,i}\le e^{-cd}$ in this range.

If $s+|\ell-m|>\delta_*d$, the expected total population from any initial type is at most
$e^{\lambda_+(t-s)}$, and the lower bound in \eqref{eq:UO1} gives
$\mathcal E_{s,\ell,i}\le d^Ce^{Cd}$.  Increasing the coefficient of
$s+|\ell-m|$ in \eqref{eqtag:GR6} makes its right side dominate this
bound.  This proves \eqref{eqtag:GR6} for every $\ell$.

Finally, fix an open neighbourhood $U\supset\mathcal R_*$.  By
Lemma~\ref{lem:root-active-stability}, there are
$\eta_U,\varepsilon_U>0$ such that every
$\varepsilon_U$-near-maximizer with $|u-u_*(\theta)|\le\eta_U$ has its associated
tuple in $U$.  Replace $\varepsilon_{\rm ret}$ by
$\min\{\varepsilon_{\rm ret},\varepsilon_U\}$.  Since
$|t/d-u_*(m/d)|=O(\log d/d)$, this time condition holds for all large
$d$; Proposition~\ref{prop:global-oneobject}, specifically \eqref{eq:UO2},
gives the newly discarded compact counts a positive gap $c_U$.  Combining
the count-tail estimates with the preceding argument at this reduced
value of $\varepsilon_{\rm ret}$, and then decreasing both the exponential-decay constant
and the cutoff on the removed time if necessary, proves \eqref{eqtag:GR6} with
$c_U,\delta_U>0$.  This proves the last assertion.  If $U$ depends on $q$,
so may $c_U$ and $\delta_U$.
\end{proof}

To sum the retained continuation weights over the short trunk, we use
distance conjugation.  For $c>0$, conjugating the mutation operator by
$c^{d_{\rm H}(\cdot,\mathbf0)}$ gives the off-diagonal row sum
\begin{equation}
 R_c(\xi)=(1-\xi)c+\frac{\xi}{b-1}c^{-1}
                 +\frac{(b-2)\xi}{b-1}.              \label{eqtag:IG0}
\end{equation}
The three terms correspond to mutations increasing the distance,
decreasing it, or leaving it unchanged.
{The next inequality yields an explicit upper bound on the growth of
the conjugated trunk.  This bound is outweighed by the continuation factor
$e^{-(1+q)\kappa_n r}$ in \eqref{eqtag:GR5}.}

{
\begin{lemma}
\label{lem:setvalued-initial-arm-gap}
Under Assumption~\ref{ass:environment}, for a maximizing saddle at the critical time,
$\mathfrak s=(\nu,z,\kappa)\in\mathcal A(\theta,u_*(\theta))$, put
$a=A_b(z)/C_b(z)$.  There are $q_0\in(0,1)$,
$\varepsilon\in(0,\kappa_{\rm rt}/4]$, and $\eta>0$ such that, uniformly over
the set $\mathcal R_*$ and $0<q\le q_0$,
\begin{equation}
 R_{a^{1+q}}(\theta)H\bigl((1+q)\kappa-\varepsilon q\bigr)
 \le1-\eta q.
 \label{eq:root-initial-transform-gap}
\end{equation}
Fix $q\in(0,q_0]$ and $C_0<\infty$.  There is
$\varepsilon_{\rm ret}(q)>0$ such that, for all sufficiently large $d$,
every count retained with this choice and with $\theta=m/d$ and
$|t-du_*(\theta)|\le C_0\log d$ obeys
\begin{equation}
 R_{a_n^{1+q}}(\theta)H\bigl((1+q)\kappa_n-\varepsilon q\bigr)
 \le1-\frac{\eta q}{2},
 \qquad \kappa_n\ge\frac{3\kappa_{\rm rt}}4.
 \label{eq:finite-initial-transform-gap}
\end{equation}
The quantifier order is $q$ first, then
$\varepsilon_{\rm ret}(q)$, and finally $d$.
\end{lemma}

\begin{proof}
Put $u=u_*(\theta)$.  Stationarity gives
$R_a(\theta)H(\kappa)=1$, $H(\kappa)=z/\nu$, and
$H'(\kappa)/H(\kappa)=-u/\nu$.
We use the continuation gap to control the derivative with respect to
$q$ at $q=0$.  Differentiating \eqref{eqtag:MR2} at $\alpha=0$ gives
\[
 \left.\frac{\dd}{\dd\alpha}\xi(\mathbf r^\alpha(\mathfrak s))
 \right|_{\alpha=0}
 =z\left((1-\theta)a-\frac{\theta}{b-1}a^{-1}\right).
\]
Differentiating \eqref{eq:canonical-continuation-exponent} and using $z\ell_\theta'(z)=\nu$ and $\ell_\theta(z)-\nu+\kappa u=0$, gives
$$
 \partial_\alpha B_\theta(0;z)
 =z\log a\left((1-\theta)a-\frac{\theta}{b-1}a^{-1}\right)-\kappa u.
$$
Choose $\varepsilon\le\kappa_{\rm rt}/4$ so that
$\varepsilon u_{\max}\le c_0/2$, where $c_0$ is from
\eqref{eq:root-canonical-continuation-gap}.  Hence,
$$
 \left.\frac{\dd}{\dd q}
 \Bigl(R_{a^{1+q}}(\theta)H\bigl((1+q)\kappa-\varepsilon q\bigr)\Bigr)
 \right|_{q=0}
 =\frac{\partial_\alpha B_\theta(0;z)+\varepsilon u}{\nu}
 \le-\frac{c_0}{2\nu}.
$$
Compactness of $\mathcal R_*$ makes this derivative uniformly negative
and the second derivatives uniformly bounded.  Taylor's theorem proves
\eqref{eq:root-initial-transform-gap}.  For fixed $q$, continuity keeps
half this slack on a neighbourhood of $\mathcal R_*$; intersect it with
$\{\kappa>3\kappa_{\rm rt}/4\}$ and apply
Lemma~\ref{lem:root-active-stability} to obtain
\eqref{eq:finite-initial-transform-gap} with this choice of $\varepsilon_{\rm ret}(q)$.
\end{proof}
}

The following comparison bounds the kernel above and below by powers of $d$ for
$|t-du_*(m/d)|\le C_0\log d$, independently of the split time.

\begin{corollary}
\label{cor:polynomial-root-band}
For every fixed $C_0<\infty$, there is $C<\infty$ such that,
uniformly for $0\le m\le d$ and
$|t-du_*(m/d)|\le C_0\log d$, one has, for all sufficiently large $d$,
\begin{equation}
 d^{-C}\le\widehat K_{d,m}(t)\le d^C.                 \label{eqtag:FD5}
\end{equation}
\end{corollary}

\begin{proof}
At $t=du_*(m/d)$, Proposition~\ref{prop:global-oneobject} gives
$d^{-1/2}\ll\widehat K_{d,m}(t)\ll d^{1/2}$, uniformly in $m$.
Proposition~\ref{prop:global-oneobject-shifts} bounds the kernel ratio
under an $O(\log d)$ time shift by powers of $d$.  Combining these two
bounds proves \eqref{eqtag:FD5}.

\end{proof}

We also need an exponential tail for the mutation count of a short trunk.

\begin{lemma}
\label{lem:marked-short-time-tail}
For fixed $C,\zeta>0$ and $a\ge0$, there are $\delta,c>0$ such that,
for all sufficiently large $d$,
\begin{equation}
 e^{Cr}\sum_{j\ge\zeta d}e^{aj}\frac{(Cr)^j}{j!}
 \le e^{-cd},\qquad 0\le r\le\delta d.
 \label{eqtag:FD18b}
\end{equation}
In particular, this controls the sum of the positive jump-count terms
of a conjugated trunk that leaves a distance tube of width $\zeta d$,
even when each $j$-mutation term is multiplied by $e^{aj}$,
provided the conjugated off-diagonal row sums and diagonal potential
are uniformly bounded above.
\end{lemma}

\begin{proof}
Use $\one_{\{j\ge\zeta d\}}\le e^{j-\zeta d}$ to bound the left side by
$\exp(Cr+e^{a+1}Cr-\zeta d)$.
Choose $\delta$ so that $C(1+e^{a+1})\delta\le\zeta/2$.
Leaving the tube requires at least $\zeta d$ mutations, and a uniform
row-sum bound gives $(Cr)^j/j!$ for the $j$-mutation term.  This proves
the stated application as well.
\end{proof}

A space--time bridge records a discrete mutation path together with its
holding times.  For $n\ge0$, write
$$
 \Delta_n(t)=\bigl\{(r_0,\ldots,r_n)\in[0,\infty)^{n+1}:
                 r_0+\cdots+r_n=t\bigr\}.
$$
The independent-visit bridge law is\footnote{For $n=0$, integration over $\Delta_0(t)$ means evaluation at $r_0=t$
(the zero-dimensional simplex convention).}
\begin{equation}
\begin{split}
 \mathbb E_{\mathbb Q_{d,m,t}}[F]
 &:=\frac1{\widehat K_{d,m}(t)}
 \sum_{n\ge0}\bigl((b-1)d\bigr)^{-n}\sum_{\pi:\mathbf{x}_m\to\mathbf 0}
 \int_{\Delta_n(t)}
   \prod_{j=0}^nh(r_j)\,
   \one_{\{|\pi|=n\}}F(\pi,r_0,\ldots,r_n)\,\dd\boldsymbol r.
\end{split}
\label{eq:Qbridge}
\end{equation}
Here $\dd\boldsymbol r=\dd r_0\cdots\dd r_{n-1}$ is Lebesgue measure on
$\Delta_n(t)$.  For arbitrary endpoints $x,y$ and $t>0$, write
$\mathbb Q_{x,y,t}$ for the analogous law, normalized by
$\widehat K_{d,d_{\mathrm H}(x,y)}(t)$.  At $t=0,x=y$, this is the
zero-mutation bridge with its single holding time equal to zero.
The prescribed endpoints of a bridge are its starting and terminal
vertices, fixed by this conditioning.

\subsubsection{Upper bound for the fractional branching sum}

Here and below, unsubscripted $q\in(0,1)$ denotes an auxiliary fractional
exponent.
For $r\in[0,t]$, consider branching events in the independent-visit
system at time $r$.  The measure $\mathbf K_r^P(x;(v,i))P(di)$ is the
expected number of particles at $v$ with signed rate in $\dd i$.
Multiplication by $i_+\,\dd r$ gives the corresponding expected number of
branching events during $[r,r+\dd r]$.  Each child contributes, in
expectation, $\mathbf k_{t-r}((v,i),\mathbf0)$ target descendants at the
final time.  We therefore put
\begin{equation}
 \widehat{\mathcal B}_{t,r}^{(1+q)}(x)
 :=\sum_{v\in\Hc_d^{(b)}}\int
\mathbf K_r^P\bigl(x;(v,i)\bigr)i_+
\mathbf k_{t-r}\bigl((v,i),\mathbf 0\bigr)^{1+q}\,P(di).
\label{eq:fractional-branching-sum}
\end{equation}
All three factors use the same $i$ before integration.  Given
that common label, the two children evolve independently, so the square
of {the expected target count from one child} is the expected product of their target
counts.  The power $1+q$ is applied to the expected target count from one
child.  The square would give
the usual common-ancestor integrand for the second factorial moment,
up to the factor two for ordered children.  We use only small $q>0$;
\begingroup
after environmental averaging, the integrated branching-sum bound gives
the population moment estimate in
Proposition~\ref{prop:quenched-fractional-diagram}.
\endgroup

We also allow $q=0$ in the displayed definition.  The semigroup on
genotypes and rates gives the exact identity
\[
 \sum_v\int\mathbf K_r^P\bigl(x;(v,i)\bigr)
       \mathbf k_{t-r}\bigl((v,i),\mathbf0\bigr)\,P(di)
 =\widehat K_t(x,\mathbf0),\qquad 0\le r\le t.
\]
Since $i_+\le\lambda_+$, this gives
$\widehat{\mathcal B}_{t,r}^{(1)}(x)
\le\lambda_+\widehat K_t(x,\mathbf0)$.
The composition keeps the rate at the split and holds also at $r=0,t$.

\begin{example}
Suppose the common type at a split is equally likely to be slow or fast,
with expected target counts $1$ and $3$ per child, respectively.
The expected product of the children's counts differs from the square
of their common mean:
\[
 \tfrac12(1\cdot1)+\tfrac12(3\cdot3)=5,
 \qquad
 \bigl(\tfrac12\cdot1+\tfrac12\cdot3\bigr)^2=4.
\]
The second calculation averages each child's expected count over its
initial rate separately before multiplying, and therefore includes
mixed slow--fast pairs.  In the branching sum, the
actual split weights include the trunk first-moment weight at the shared rate
and $i_+$.  The
scalar kernel $\widehat K_{t-r}(v,\mathbf0)
=\int\mathbf k_{t-r}((v,i),\mathbf0)\,P(di)$ averages the type before
taking a power.  Both descriptions use the independent-visit model;
Lemma~\ref{lem:scalar-type-domination} compares them by an inequality.
\end{example}

\begin{theorem}
\label{thm:globalfractionaldiagram}
Under Assumption~\ref{ass:environment}, there is
$q_0>0$ such
that, for every fixed $q\in(0,q_0]$ and $C_0<\infty$, there is
$c_q\in(0,\infty)$ for which, uniformly in
$$
 0\le m\le d,\qquad
 \left|t-du_*\!\left(\frac md\right)\right|\le C_0\log d,
 \qquad 0\le r\le t,
$$
one has, for all sufficiently large $d$,
\begin{equation}
 {
 \widehat{\mathcal B}_{t,r}^{(1+q)}(\mathbf{x}_m)
 \ll_{q,C_0}
   \widehat K_{d,m}(t)^{1+q}e^{-c_qr}
   +\widehat K_{d,m}(t)e^{-c_q(t-r)}
 .}                                             \label{eqtag:FD2}
\end{equation}
Consequently,
\begin{equation*}
 \int_0^t\widehat{\mathcal B}_{t,r}^{(1+q)}(\mathbf{x}_m)\dd r
 \ll_{q,C_0}\widehat K_{d,m}(t)+\widehat K_{d,m}(t)^{1+q}.
\end{equation*}
\end{theorem}

The two terms on the right-hand side of \eqref{eqtag:FD2} correspond to
the early and late splits described in Section~\ref{sec:proofstrategy},
respectively.

Extend the law $\mathbb Q_{d,m,t}$ from \eqref{eq:Qbridge} by giving
each visit of duration $s$ an independent rate with distribution
$h(s)^{-1}\exp(-(1-i)s)P(di)$.  Write $(X_s,I_s)$ for its position
and signed rate.  This is the first-moment path law of the
typed process: its density relative to the mutation-and-rate process
with fresh initial rate is
$\widehat K_{d,m}(t)^{-1}\exp(\int_0^t I_s\,\dd s)
\one_{\{X_t=\mathbf0\}}$.
The typed semigroup therefore gives
\begin{equation}
 \mathbb Q_{d,m,t}(X_r=v,I_r\in\dd i)
 =\frac{\mathbf K_r^P\bigl(\mathbf x_m;(v,i)\bigr)
              \mathbf k_{t-r}\bigl((v,i),\mathbf0\bigr)}
             {\widehat K_{d,m}(t)}\,P(di).
 \label{eq:typed-first-moment-bridge}
\end{equation}
Consequently, \eqref{eq:typed-first-moment-bridge} gives the exact identity
\begin{equation}
 \frac{\widehat{\mathcal B}_{t,r}^{(1+q)}(\mathbf x_m)}
      {\widehat K_{d,m}(t)}
 =\mathbb E_{\mathbb Q_{d,m,t}}\!
    \Bigl[(I_r)_+\,\mathbf k_{t-r}\bigl((X_r,I_r),\mathbf0\bigr)^q\Bigr].
 \label{eq:fractional-bridge-expectation}
\end{equation}

\begin{proof}
Put $\theta=m/d$ and $u=t/d$.  We use $\rho(v)=d_{\mathrm H}(v,\mathbf 0)$.

\emph{Step 1: short initial trunks and interpolation.}
First fix any $q\in(0,1)$ below the threshold in
Lemma~\ref{lem:setvalued-initial-arm-gap}.  Use the
$q$-dependent choice $\varepsilon_{\rm ret}(q)$ supplied by that lemma in \eqref{eqtag:GR0}; we
keep the notation $\mathcal I_d(m,t)$.  Every retained count then satisfies
{\eqref{eq:finite-initial-transform-gap}} and has a positive clock tilt.
All omitted counts are included in
$\mathcal E_{s,\ell,i}$.  By the final clause of
Lemma~\ref{lem:global-retained-arm}, \eqref{eqtag:GR6} then holds with
positive constants which we denote by
$c_{\rm tail}(q),\delta_{\rm tail}(q)$.
To apply \eqref{eqtag:GR4}--\eqref{eqtag:GR6} inside the branching sum,
we use Jensen's inequality to bound the $(1+q)$-power of the retained
sum by a sum of separate contributions from the retained counts.
For the subprobability weights $\pi_n$ and nonnegative $U_n$,
\begin{equation*}
\left(\sum_n\pi_nU_n\right)^{1+q}
 \le\sum_n\pi_nU_n^{1+q}.
\end{equation*}
For $c>0$, conjugating the independent-visit first-moment operator
by $c^{\rho(x)}$ gives
$$
 (\mathcal L^cf)(x,i)
 =\frac1{(b-1)d}\sum_{y\sim x}\int
 c^{\rho(y)-\rho(x)}f(y,j)\,P(dj)-(1-i)f(x,i).
$$
At a vertex with distance density $\xi=\rho(x)/d$, the sum of the spatial
coefficients in its off-diagonal term is $R_c(\xi)$ from \eqref{eqtag:IG0}.

{
For a retained count $n$, put $c=a_n^{1+q}$,
$\psi=(1+q)\kappa_n-\varepsilon q>0$, and
$\varphi_\psi(i)=(\psi+1-i)^{-1}$.  Since $R_c$ is affine and
$c,c^{-1},H(\psi)$ are uniformly bounded,
\eqref{eq:finite-initial-transform-gap} gives a fixed $\zeta=\zeta(q)>0$
such that $R_c(\rho(x)/d)H(\psi)\le1$ on
$T=\{x:|\rho(x)-m|\le\zeta d\}$.  Direct calculation gives
$$
 (\mathcal L^c\varphi_\psi)(x,i)
 =\psi\varphi_\psi(i)+R_c(\rho(x)/d)H(\psi)-1.
$$
Let $\mathcal L_T^c$ and $\mathbf K^c_{r,T}$ be the generator and
semigroup killed on leaving $T$.  Killing removes positive jump terms,
so $\mathcal L_T^c\varphi_\psi\le \psi\varphi_\psi$.
Both $\varphi_\psi$ and $(1-i)\varphi_\psi$ are bounded for $i\le\lambda_+$,
placing $\varphi_\psi$ in the generator domain of
Appendix~\ref{sec:linearprimer}.  Positivity of the killed semigroup and the generator inequality give
\[
 e^{-\psi r}\mathbf K^c_{r,T}\varphi_\psi-\varphi_\psi
 =\int_0^r e^{-\psi\tau}\mathbf K^c_{\tau,T}
   (\mathcal L_T^c-\psi)\varphi_\psi\,\dd\tau\le0.
\]
Therefore, 
$\mathbf K^c_{r,T}\varphi_\psi\le e^{\psi r}\varphi_\psi$.

The powers of $H(\kappa_n)$ in \eqref{eqtag:GR5} are uniformly bounded,
and the remaining rate factor satisfies
$$
 \frac{i_+}{(\kappa_n+1-i)^{1+q}}\le C\varphi_\psi(i),
 \qquad i\le\lambda_+.
$$
Indeed, it vanishes for $i\le0$, and the ratio is bounded on
$[0,\lambda_+]$.  Thus, after raising \eqref{eqtag:GR5} to $1+q$,
the contribution of this count inside $T$, divided by
$\widehat K_{d,m}(t)^{1+q}$ and before multiplication by $\pi_n$, is at most
\begin{equation*}
 C e^{-(1+q)\kappa_n r}
 \int(\mathbf K^c_{r,T}\varphi_\psi)(\mathbf x_m,i)\,P(di)
 \le C H(\psi)e^{-\varepsilon qr}\le C e^{-\varepsilon qr}.
\end{equation*}
}
{For the remaining estimates, $\mathbb E$ refers to the
rate-one mutation process started at}
$\mathbf x_m$, with initial rate and subsequent visit rates drawn
independently from $P$.  Let $J_r$ count its mutations by time $r$.
Paths leaving the tube have $J_r\ge\zeta d$.
Since $|\rho(X_r)-m|\le J_r$, their weighted mass, including
$e^{-(1+q)\kappa_nr}\le1$, is at most
$\exp(-\zeta d+Cr)$ by Lemma~\ref{lem:marked-short-time-tail}.
{The bound $\frac{i_+}{(\kappa_n+1-i)^{1+q}}\le C$ holds uniformly over
retained $n$ and $i\le\lambda_+$.}
Choose $\delta_I(q)>0$ so that this is at most $e^{-\zeta d/2}$ for
$r\le\delta_I(q)d$.

Jensen's inequality with $\pi_n$ and $\sum_n\pi_n\le1$ therefore
bounds the retained contribution by
{
$$
 C\widehat K_{d,m}(t)^{1+q}\left(e^{-\varepsilon qr}+e^{-\zeta d/2}\right).
$$
}

For the discarded arm, use
$(\mathcal M+\mathcal E)^{1+q}\le2^q(\mathcal M^{1+q}+\mathcal E^{1+q})$
and $|\rho(X_r)-m|\le J_r$ in \eqref{eqtag:GR6}.
The original mutation process has rate one and growth potential at
most $\lambda_+$, so its path sum with weight $e^{AJ_r}$ satisfies
\[
 \mathbb E\!\left[e^{\int_0^r I_s\,\dd s}e^{AJ_r}\right]
 \le\exp\bigl((\lambda_++e^A-1)r\bigr)
 \qquad\text{for every fixed }A\ge0.
\]
Consequently,
\[
 \sum_v\int\mathbf K_r^P\bigl(\mathbf x_m;(v,i)\bigr)i_+
       \mathcal E_{r,\rho(v),i}^{1+q}\,P(di)
 \le d^C\exp\bigl(-(1+q)c_{\rm tail}(q)d+C'_qr\bigr).
\]
Decrease $\delta_I(q)$ so that it is smaller than
$\delta_{\rm tail}(q)$ and $u_{\min}/8$, and
$C'_q\delta_I(q)\le(1+q)c_{\rm tail}(q)/2$.
The last display is then at most $e^{-c_{\rm rem}(q)d}$ for some
$c_{\rm rem}(q)>0$.  Since $\mathcal E$ is normalized by $\widehat K_{d,m}(t)$, the discarded
contribution to the branching sum is this bound multiplied by
$\widehat K_{d,m}(t)^{1+q}$.  Combining the two contributions gives
\[
 \frac{\widehat{\mathcal B}_{t,r}^{(1+q)}(\mathbf x_m)}
      {\widehat K_{d,m}(t)^{1+q}}
 \le {C e^{-\varepsilon qr}}+C e^{-c_{\rm rem}(q)d},
 \qquad 0\le r\le\delta_I(q)d.
\]
Here $c_{\rm rem}(q)$ is decreased to include the paths leaving the tube.
For {$0<c_q\le\min\{\varepsilon q,c_{\rm rem}(q)/\delta_I(q)\}$}, this yields
\begin{equation}
 \widehat{\mathcal B}_{t,r}^{(1+q)}(\mathbf{x}_m)
 \le C\widehat K_{d,m}(t)^{1+q}e^{-c_qr},
 \qquad 0\le r\le\delta_I(q)d.                        \label{eqtag:FD19}
\end{equation}
The parameter $\varepsilon_{\rm ret}(q)$ and the time cutoff $\delta_I(q)$ were chosen on fixed compact root
neighbourhoods after $q$ was fixed.  They are independent of $C_0$:
only the threshold on $d$ changes to put $C_0\log d/d$ inside those
neighbourhoods and absorb polynomial factors.

Now fix one such exponent $q_*\in(0,1)$ and its width $\delta_I(q_*)$.
For every $0<q\le q_*$, H\"older's inequality gives
$$
 \widehat{\mathcal B}_{t,r}^{(1+q)}(\mathbf{x}_m)
 \le
 \Bigl(\widehat{\mathcal B}_{t,r}^{(1)}(\mathbf{x}_m)\Bigr)^{1-q/q_*}
 \Bigl(\widehat{\mathcal B}_{t,r}^{(1+q_*)}(\mathbf{x}_m)\Bigr)^{q/q_*}
 \le C\widehat K_{d,m}(t)^{1+q}
       \exp\left(-\frac{c_{q_*}}{q_*}qr\right),
 \qquad 0\le r\le\delta_I(q_*)d.
$$
Here H\"older is applied to the measure
$i_+\mathbf K_r^P(\mathbf{x}_m;(v,i))P(di)$, summed over $v$;
the case $q=q_*$ is immediate.  The second inequality uses the exponent-one
bound preceding the theorem and \eqref{eqtag:FD19} at $q_*$.
The power of $\widehat K_{d,m}(t)$ is exactly
$1-q/q_*+(1+q_*)q/q_*=1+q$.  Thus, the same initial width works for every
smaller exponent.

\emph{Step 2: the middle range.}

Choose $0<\delta_T<u_{\min}/8$ small enough for Step~3, and put
$\delta_B=\min\{\delta_I(q_*),\delta_T\}$.  Throughout this step,
$\delta_I(q_*)d\le r\le t-\delta_Td$ and $\alpha=r/t$.
Both $\alpha$ and $1-\alpha$ are at least
$\delta_B/(u_{\max}+1)$ for all sufficiently large $d$.
We first prove estimates independent of the final fractional exponent.

Figure~\ref{fig:middle-range-proof} summarizes the estimates used in
\eqref{eq:fractional-bridge-expectation}.

\begin{figure}[htbp]
\centering
\begin{tikzpicture}[
  x=1cm,y=1cm,font=\small,
  box/.style={draw,rounded corners=2pt,align=center,
              text width=5.5cm,minimum height=1.25cm,inner sep=5pt},
  flow/.style={->,semithick,shorten >=2pt,shorten <=2pt}
]
\node[box] (prefix) at (-3.2,0)
  {Large total trunk weight\\
   \eqref{eq:midpoint-class-lower},
   \eqref{eq:bulk-prefix-single-count}};

\node[box] (concentration) at (3.2,0)
  {Clock and coordinate concentration\\
   \eqref{eq:bulk-clock-count-concentration},
   \eqref{eq:bulk-profile-concentration}};

\node[box] (continuation) at (-3.2,-1.9)
  {Small continuation kernel\\near the coordinate means\\
   \eqref{eq:bulk-typed-continuation}};

\node[box] (exception) at (3.2,-1.9)
  {Exponentially small probability\\of discarded counts or\\large coordinate deviations\\
   \eqref{eq:bulk-bridge-good-event}};

\node[box,text width=9cm] (combine) at (0,-3.9)
  {Fractional bridge identity
   \eqref{eq:fractional-bridge-expectation}\\
   Global continuation bound and sufficiently small $q$};

\node[box,text width=9cm] (bound) at (0,-5.7)
  {Middle-range branching bound\\
   \eqref{eqtag:FD11}};

\draw[flow] (prefix) -- (continuation);
\draw[flow] (concentration) -- (exception);
\draw[flow] (continuation.south)
  -- (combine.north -| continuation.south);
\draw[flow] (exception.south)
  -- (combine.north -| exception.south);
\draw[flow] (combine) -- (bound);
\end{tikzpicture}
\caption{Estimates used in the middle-range argument.}
\label{fig:middle-range-proof}
\end{figure}
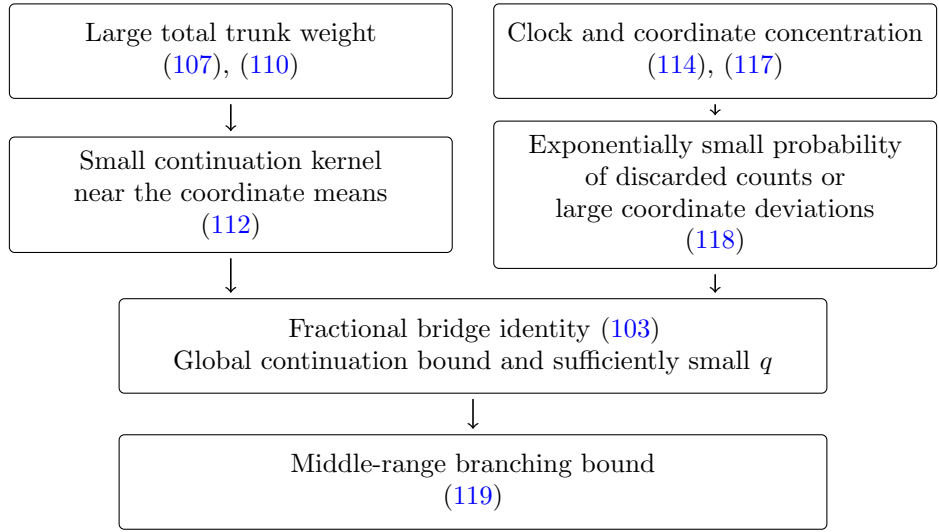

First, the scalar kernels satisfy
\begin{equation}
 \sum_v\widehat K_r(x,v)\widehat K_h(v,y)
 \le\widehat K_{r+h}(x,y),\qquad r,h\ge0.
 \label{eq:scalar-convolution-inequality}
\end{equation}
For a fixed $v$, the two functions
$i\mapsto\mathbf k_r((v,i),x)$ and
$i\mapsto\mathbf k_h((v,i),y)$ are increasing by
\eqref{eq:typed-jump-count}.  For any increasing functions $f,g$,
\begin{equation}
 \int fg\,\dd P-\int f\,\dd P\int g\,\dd P
 =\frac12\iint\bigl(f(i)-f(j)\bigr)\bigl(g(i)-g(j)\bigr)P(di)P(dj)\ge0.
 \label{eq:rate-association}
\end{equation}
Apply \eqref{eq:rate-association} at each $v$, use detailed balance
\eqref{eq:typed-detailed-balance}, and sum:
$$
 \sum_v\widehat K_r(x,v)\widehat K_h(v,y)
 \le\sum_v\int\mathbf K_r^P\bigl(x;(v,i)\bigr)
                      \mathbf k_h\bigl((v,i),y\bigr)P(di)
 =\widehat K_{r+h}(x,y).
$$
This proves \eqref{eq:scalar-convolution-inequality} with the rate
preserved in the typed composition.

For a vertex $v$, let $\mathbf r(v)$ be its coordinate proportions from
Figure~\ref{fig:midpoint-coordinate-classes}.  Every vertex $w$ with
$\mathbf r(w)=\mathbf r(v)$ has the same two endpoint distances as $v$.
The binomial and multinomial counts, followed by Stirling's bounds, give
\begin{equation}
\#\bigl\{w:\mathbf r(w)=\mathbf r(v)\bigr\}
 =\binom{d-m}{dr_C}
   \binom{m}{dr_y,dr_x,dr_o}(b-1)^{dr_C}(b-2)^{dr_o}
 \ge d^{-C}\exp\Bigl(d\mathcal S_{b,\theta}\bigl(\mathbf r(v)\bigr)\Bigr).
\label{eq:midpoint-class-lower}
\end{equation}
Here the entries on the right are those of $\mathbf r(v)$, and the
$r_o$ entry and its factor are omitted for $b=2$.  Zero-size initial
classes contribute one.  The error in the logarithmic Stirling bounds
is $O(\log d)$ uniformly over all class sizes, so
\eqref{eq:midpoint-class-lower} also holds at $\theta=0,1$.

At $\mathbf r^\alpha(\mathfrak s)$ with
$\mathfrak s\in\mathcal A(\theta,u_*(\theta))$, \eqref{eq:canonical-prefix-entropy} and
\eqref{eq:root-canonical-continuation-gap} bound
$\mathcal S_{b,\theta}+G(\eta,\alpha u_*(\theta),\alpha\nu)$ below
by a positive constant on this fixed middle interval.
Moreover, \eqref{eq:canonical-prefix-count} gives
$$
 \alpha\nu-\eta\bigl(\mathbf r^\alpha(\mathfrak s)\bigr)
 =J_{\eta(\mathbf r^\alpha(\mathfrak s))}(\alpha z)
       -\eta\bigl(\mathbf r^\alpha(\mathfrak s)\bigr)>0.
$$
Compactness of $\mathcal R_*$ and of the interval of $\alpha$ bounds
this difference away from zero and places $\alpha z,\kappa$ in fixed
compact interior tilt intervals.  The inverse spatial and clock mean
maps are continuous there.  Lemma~\ref{lem:root-active-stability} thus
allows a fixed $\varepsilon_{\rm bulk}>0$ and a fixed
tolerance $\zeta_{\rm bulk}>0$ in the coordinate proportions such that the following holds.
For every $n\in\mathcal I_d^{J_{\rm rt},\varepsilon_{\rm bulk}}(m,t)$,
use its exact tilts $z_n,\kappa_n$ from \eqref{eqtag:GR0a}.
If
\begin{equation}
 \Bigl\|\mathbf r(v)-
        \mathbf r^\alpha\bigl((n/d,z_n,\kappa_n)\bigr)\Bigr\|_\infty
 \le\zeta_{\rm bulk},
 \label{eq:bulk-profile-neighbourhood}
\end{equation}
choose an admissible integer $j$ with $|j-\alpha n|\le2$ for the
prefix distance $d\eta(\mathbf r(v))$.  For $b=2$, choose its parity
to match that distance.  The uniform positive lower bound on the mutation
count density minus the trunk distance density ensures admissibility for
all sufficiently large $d$.
For some $c_3>0$, uniformly over these choices,
\begin{equation}
 \mathcal S_{b,\theta}\bigl(\mathbf r(v)\bigr)
 +G\left(\eta\bigl(\mathbf r(v)\bigr),\frac rd,\frac jd\right)\ge3c_3.
 \label{eq:bulk-prefix-positive}
\end{equation}
Indeed, choose $4c_3$ below the minimum of the left-hand side of
\eqref{eq:canonical-prefix-entropy}, taken over the critical saddles and
the fixed middle interval.  Then use continuity of $\mathcal S_{b,\theta}$ and $G$ and the
stability estimate in Lemma~\ref{lem:root-active-stability} to preserve the lower bound $3c_3$.
This also keeps the two inner tilts for the trunk count $j$ in
compact subsets of the interiors of their domains; replacing $j/d$ by $(j+1)/d$ in the clock mean
changes its tilt by $O(d^{-1})$.

Retain just this count in the prefix kernel.  Exponential tilting,
Stirling's bounds, and Lemmas~\ref{lem:spatialLCLT} and
\ref{lem:clockLCLT}, applied at these inner minimizers, give
\begin{equation}
 \widehat K_r(\mathbf x_m,v)
 \ge p_{d,d\eta(\mathbf r(v))}(j)h^{*(j+1)}(r)
 \ge c d^{-1/2}
       \exp\left(dG\left(\eta\bigl(\mathbf r(v)\bigr),\frac rd,\frac jd\right)\right).
 \label{eq:bulk-prefix-single-count}
\end{equation}
These local limit bounds concern the specified count; no maximizing
count for the segment is used. The extra terminal holding interval, which gives $j+1$ holding intervals
for $j$ mutations, changes the clock exponent by only $O(1)$ uniformly.
We now turn the lower bound on the total first-moment weight of trunks
ending at vertices with the specified coordinate proportions into an upper
bound on the continuation kernel. Retaining all vertices $w$ with $\mathbf r(w)=\mathbf r(v)$ in
\eqref{eq:scalar-convolution-inequality}, and applying
\eqref{eq:midpoint-class-lower}--\eqref{eq:bulk-prefix-single-count},
gives
\begin{equation}
 \widehat K_{t-r}(v,\mathbf0)
 \le d^C\widehat K_{d,m}(t)\exp(-3c_3d)
 \le\exp(-2c_3d).
 \label{eq:bulk-scalar-continuation}
\end{equation}
The second inequality uses the polynomial bound
\eqref{eqtag:FD5} for $|t-du_*(m/d)|\le C_0\log d$.

For every fixed $\varepsilon>0$, the comparison in the proof of
Lemma~\ref{lem:global-retained-arm} gives
$$
 \mathbf k_{t-r}\bigl((v,i),\mathbf0\bigr)
 \le
 \frac{\exp\bigl(\varepsilon(t-r)\bigr)}
      {P\bigl([\lambda_+-\varepsilon,\lambda_+]\bigr)}
      \widehat K_{t-r}(v,\mathbf0),\qquad i\le\lambda_+.
$$
Choose $\varepsilon$ so that
$\varepsilon(u_{\max}+1)<c_3/2$.
The denominator is a fixed positive number, so
\eqref{eq:bulk-scalar-continuation} implies, for all large $d$,
\begin{equation}
 \sup_i\mathbf k_{t-r}\bigl((v,i),\mathbf0\bigr)\le\exp(-c_3d)
 \quad\hbox{whenever \eqref{eq:bulk-profile-neighbourhood} holds}.
 \label{eq:bulk-typed-continuation}
\end{equation}
This comparison uses only the essential supremum of the rate law.

For each retained total mutation count, we show that
\eqref{eq:bulk-profile-neighbourhood} fails with exponentially small
conditional bridge probability.  We first control the mutation count
before the split, then the coordinate proportions of the split vertex.  Write $N$ for the total
mutation count and $J_r$ for the number of mutations before $r$ under
$\mathbb Q_{d,m,t}$.  By \eqref{eq:Qbridge},
conditional on $N=n$, the discrete mutation path is the $n$-step Hamming
bridge and is independent of the holding times.  Write these holding
times as $R_0,\ldots,R_n$.  They have the law of independent variables
of density $f_{\kappa_n}$,
conditioned on their sum being $t=(n+1)\mu_{\kappa_n}$.
For every measurable $A\subset[0,t]$, when $k,n+1-k\asymp d$,
Lemma~\ref{lem:clockLCLT} gives
\begin{equation}
\mathbb Q_{d,m,t}\left(\sum_{a=0}^{k-1}R_a\in A\,\middle|\,N=n\right)
 =\frac{\int_A\widetilde f_{k,\kappa_n}(s)
          \widetilde f_{n+1-k,\kappa_n}(t-s)\,\dd s}
         {\widetilde f_{n+1,\kappa_n}(t)}
 \le C\int_A\widetilde f_{k,\kappa_n}(s)\,\dd s.
\label{eq:bulk-conditioned-clock}
\end{equation}
For a small fixed $\epsilon>0$, take integers within a bounded
distance of $\alpha(n+1)\pm\epsilon d$.
For either chosen integer $k$, the sum $\sum_{a=0}^{k-1}R_a$ under the
unconditioned law has mean $k\mu_{\kappa_n}$, with
$|k\mu_{\kappa_n}-r|\ge c\epsilon d$.  Moreover, $k\asymp d$ and
$n+1-k\asymp d$.
The uniform exponential moment \eqref{eq:clockuniformmoments} yields
Chernoff bounds for these deviations.  Since the $k$th mutation occurs
at $\sum_{a=0}^{k-1}R_a$, \eqref{eq:bulk-conditioned-clock} therefore implies
\begin{equation}
 \mathbb Q_{d,m,t}\bigl(|J_r-\alpha n|>2\epsilon d\bigm| N=n\bigr)
 \le C\exp(-c_\epsilon d).
 \label{eq:bulk-clock-count-concentration}
\end{equation}
The difference between $\alpha n$ and $\alpha(n+1)$ is at most one
and is absorbed into the tolerance.

Fix an integer $k$ satisfying $|k-\alpha n|\le2\epsilon d$, and put $\beta=k/n$. We first obtain spatial concentration under independent coordinate
bridges, then condition on the total mutation count and the mutation
count before the cut. Consider the product law $\mathbb P_{\rm prod}^{z_n}$ of independent
one-coordinate bridges of duration $z_n$ with the full prescribed
endpoints, and write $Y_s$ for their vector of symbols.
Each entry of $\mathbf r(Y_{\beta z_n})$ is an average of $d$
independent indicators with mean equal to the corresponding entry
of $\mathbf r^\beta((n/d,z_n,\kappa_n))$.
Hoeffding's inequality \cite[Theorem~2]{Hoeffding1963}, applied to
the four entries, gives
\begin{equation}
 \mathbb P_{\rm prod}^{z_n}\biggl(
 \Bigl\|\mathbf r(Y_{\beta z_n})
       -\mathbf r^\beta\bigl((n/d,z_n,\kappa_n)\bigr)\Bigr\|_\infty>\epsilon
 \biggr)\le8\exp(-2\epsilon^2d).
 \label{eq:bulk-coordinate-Hoeffding}
\end{equation}
Under this product law, let $N$ again denote the total number of jumps
and $J_{\beta z_n}$ their number before $\beta z_n$.
Conditional on $N=n$ and the endpoints, the jump times are uniform
and independent of the discrete mutation path.  Thus,
\begin{equation}
\mathbb P_{\rm prod}^{z_n}(N=n,J_{\beta z_n}=k)
 =\mathbb P_{z_n}(\Sigma_{d,m,z_n}=n)
       \binom nk\beta^k(1-\beta)^{n-k}
 \ge\frac c{\sqrt d}\frac c{\sqrt n}\ge\frac c d.
\label{eq:bulk-spatial-conditioning}
\end{equation}
Here the spatial local limit theorem applies at its exact mean, and
the binomial estimate follows from Stirling's bounds, uniformly since
$\beta$ stays away from $0,1$.  Admissibility of $n$ supplies the binary
parity condition; there is no further parity condition on $k$.
After this conditioning, $Y_{\beta z_n}$ has exactly the law of the
$n$-step Hamming bridge at step $k$.

Divide \eqref{eq:bulk-coordinate-Hoeffding} by
\eqref{eq:bulk-spatial-conditioning}.  The vector $\mathbf r^\alpha((n/d,z_n,\kappa_n))$ is uniformly
continuous in $\alpha$ on the compact parameter set, and
$|\beta-\alpha|\le2\epsilon d/n\le C\epsilon$.
Choose $\epsilon$ small enough that the total error in coordinate proportions is at most
$\zeta_{\rm bulk}$.  Combining the conditioned spatial bound with
\eqref{eq:bulk-clock-count-concentration} proves
\begin{equation}
 \mathbb Q_{d,m,t}\biggl(
 \Bigl\|\mathbf r(X_r)
       -\mathbf r^\alpha\bigl((n/d,z_n,\kappa_n)\bigr)\Bigr\|_\infty
       >\zeta_{\rm bulk}\,\biggm|\,N=n\biggr)
 \le C\exp(-c d).
 \label{eq:bulk-profile-concentration}
\end{equation}
This is uniform over all retained $n$ and all middle-range $r$.

Let $E_r$ be the union, over
$n\in\mathcal I_d^{J_{\rm rt},\varepsilon_{\rm bulk}}(m,t)$, of the
events $N=n$ and \eqref{eq:bulk-profile-neighbourhood} with $v=X_r$.
The discarded-count estimate \eqref{eq:UO2}, together with
\eqref{eq:bulk-profile-concentration} and
\eqref{eq:bulk-typed-continuation}, gives a fixed $c_4>0$ such that
\begin{equation}
 \mathbb Q_{d,m,t}(E_r^c)\le\exp(-c_4d),\qquad
 \mathbf k_{t-r}\bigl((X_r,I_r),\mathbf0\bigr)\le\exp(-c_3d)
 \ \hbox{on }E_r.
 \label{eq:bulk-bridge-good-event}
\end{equation}
The retained-count probabilities sum to at most one, so this estimate
allows any number of maximizing full saddles.

The elementary first-moment bound holds everywhere:
$\mathbf k_{t-r}((v,i),\mathbf0)\le\exp(\lambda_+(t-r))$.
Choose the theorem's $q_0\le q_*$ so that
$q_0\lambda_+(u_{\max}+1)\le c_4/2$.
On $E_r$, the continuation kernel raised to $q$ is at most
$e^{-c_3qd}$.  On $E_r^c$, its contribution to the expectation in
\eqref{eq:fractional-bridge-expectation} is at most
$\lambda_+e^{-c_4d+q\lambda_+(t-r)}\le\lambda_+e^{-c_4d/2}$ for $0<q\le q_0$,
since $t-r\le(u_{\max}+1)d$.
For every fixed $0<q\le q_0$, equations
\eqref{eq:fractional-bridge-expectation} and
\eqref{eq:bulk-bridge-good-event} now give
\begin{equation}
 \widehat{\mathcal B}_{t,r}^{(1+q)}(\mathbf x_m)
 \le\lambda_+\widehat K_{d,m}(t)
   \Bigl(\exp(-c_3qd)+\exp\bigl(-c_4d+q\lambda_+(u_{\max}+1)d\bigr)\Bigr)
 \le C\widehat K_{d,m}(t)\exp(-c_5qd)
\label{eqtag:FD11}
\end{equation}
for a fixed $c_5>0$.  The parameter $\varepsilon_{\rm bulk}$, the other tolerances, and the constants
$c_3,c_4$ were fixed after the two time widths and before $q$.
The given $C_0$ affects only the threshold on $d$.
Since $t-r\le(u_{\max}+1)d$, \eqref{eqtag:FD11} is bounded by the second
term of \eqref{eqtag:FD2} after reducing $c_q$.  This proves the middle
range.

\emph{Step 3: the short terminal arm.}
Let $h=t-r\le\delta_Td$.  Every signed rate is at most
$\lambda_+<1$. For sufficiently small $\delta_T$, the short-time mutation cost
outweighs the maximal first-moment growth. We bound {$\mathbf k_h((v,i),\mathbf0)$} using the exact product
heat kernel, which, with
$\ell=d_{\mathrm H}(v,\mathbf 0)$, is
$$
 \mathbb P_v(X_h=\mathbf 0)=e^{-h}A_b\!\left(\frac hd\right)^\ell
 C_b\!\left(\frac hd\right)^{d-\ell}
 \le e^{-h+C\frac{h^2}{d}},
$$
where $A_b(u)\le C_b(u)$ and $\log C_b(u)\le Cu^2$ for small $u$.
Consequently,
\begin{equation*}
 \sup_{v,i}\mathbf k_h\bigl((v,i),\mathbf 0\bigr)
 \le e^{\lambda_+h}\sup_v\mathbb P_v(X_h=\mathbf 0)
 \le\exp\left(-(1-\lambda_+)h+C\frac{h^2}{d}\right)
 \le e^{-(1-\lambda_+)h/2},
\end{equation*}
by the choice of $\delta_T$.  Keep one terminal kernel in the semigroup
convolution and bound the remaining $q$ power by this supremum.  Since
$i_+\le\lambda_+$,
\begin{equation}
\begin{split}
 \widehat{\mathcal B}_{t,t-h}^{(1+q)}(\mathbf{x}_m)
 &\le\lambda_+
 \Bigl(\sup_{v,i}\mathbf k_h\bigl((v,i),\mathbf 0\bigr)\Bigr)^q
 \sum_v\int\mathbf K_{t-h}^P\bigl(\mathbf{x}_m;(v,i)\bigr)
                 \mathbf k_h\bigl((v,i),\mathbf 0\bigr)\,P(di)\\
 &\le C\widehat K_{d,m}(t)e^{-q(1-\lambda_+)h/2}.
\end{split}
\label{eqtag:FD13}
\end{equation}
This is the second term of \eqref{eqtag:FD2}, with a constant prefactor, including at
$h=0$.

For all sufficiently large $d$, the three ranges cover $[0,t]$.
The interpolated bound from Step~1 and
\eqref{eqtag:FD11}--\eqref{eqtag:FD13}, with a common smaller positive
$c_q$, prove \eqref{eqtag:FD2}.  Integrating the two exponentials gives
the stated integrated bound.
\end{proof}

\subsection{Identifying the main contributions to the nonterminal branching sum}
\label{sec:initial-branch}

\begingroup
We normalize the nonterminal branching weights to define a probability
law on trunks and continuations. Under this law, we bound the duration
and mutation count of the trunk. These bounds keep the continuation
parameters in the regular region with exponentially high probability,
as needed for environmental averaging.
\endgroup

{For the following integral and
Proposition~\ref{prop:fractional-short-trunk}, fix $C_0<\infty$ and
$q\in(0,\min\{q_0,1\}]$ from
Theorem~\ref{thm:globalfractionaldiagram}.}
Assume
$$
1\le m\le d,\qquad
\left|t-du_*\!\left(\frac md\right)\right|\le C_0\log d.
$$
By compactness, $\min_{\theta\in[0,1]}u_*(\theta)>0$.
Fix $\delta>0$, smaller than
$\frac14\min_{\theta\in[0,1]}u_*(\theta)$ and than the fixed initial and
terminal widths chosen in Theorem~\ref{thm:globalfractionaldiagram}.  \begingroup
For all
sufficiently large $d$, $t-\delta d>0$.  For a start $y$, define the
nonterminal independent-visit branching sum
\begin{equation}
 \widehat{\mathcal C}^{I}_{q,t}(y)
 =2\bar\lambda\int_0^{t-\delta d}
 \sum_v\widehat K_r(y,v)
       \widehat K_{t-r}(v,\mathbf 0)^{1+q}\dd r,
 \qquad \bar\lambda=\int i_+\,P(di).
 \label{eq:fresh-qnormalizer}
\end{equation}
The superscript $I$ records the nonterminal range $0\le r\le t-\delta d$.
The main starting vertex is $y=\mathbf{x}_m$.  We also need nearby starts
$d_{\rm H}(y,\mathbf{x}_m)\le M$ for fixed $M<\infty$: the preliminary family in the
main proof may make a bounded number of mutations before its descendants
begin their attempts to reach the target.
In \eqref{eq:fresh-qnormalizer}, $\widehat K_r(y,v)$ is the trunk weight
and $\widehat K_{t-r}(v,\mathbf0)^{1+q}$ is the fractional continuation
weight.  The cutoff removes the short terminal arm $t-r\le\delta d$, for which
we will later use a pathwise bound.

\endgroup

\begin{definition}
\label{def:trunk-arm-diagram}
A \emph{trunk--arm diagram} $\Gamma$ with starting vertex $y$ and terminal
time $t$ records a split time $r\in[0,t]$, a split vertex $v$, a trunk path
$\pi$ from $y$ to $v$, and $j\in\{1,2\}$ continuation paths
$\gamma_1,\ldots,\gamma_j$ from $v$ to $\mathbf0$.  Each path records its
successive holding times.  The trunk lasts $r$ and each arm lasts $t-r$.
The diagram also records the factor $(\Lambda_v)_+/\bar\lambda$ for the
branching event, to be included in its environmental weight.
Its prescribed endpoints are $y$ and $\mathbf0$.  The split vertex $v$
is sampled with the diagram; conditioning on it to analyse the paths
does not make it another prescribed endpoint of the complete diagram.
\end{definition}

The two-arm picture describes descendants sharing an ancestor before the
split, while the one-arm version is used in the comparison estimates.

\begingroup
Define the auxiliary split law by normalizing the deterministic weights in
\eqref{eq:fresh-qnormalizer}:
\begin{equation}
 \mathbb Q_{q;y,t}\bigl(\dd r,\{v\}\bigr)
 =\frac{2\bar\lambda}{\widehat{\mathcal C}_{q,t}^{I}(y)}
 \widehat K_r(y,v)
 \widehat K_{t-r}(v,\mathbf0)^{1+q}
 \one_{\{0\le r\le t-\delta d\}}\,\dd r.
\label{eq:weighted-split-law}
\end{equation}
Its mass is one by \eqref{eq:fresh-qnormalizer}.  Expectations under this
auxiliary law abbreviate the corresponding integrals divided by
$\widehat{\mathcal C}_{q,t}^{I}(y)$.
Conditional on $(r,v)$, take a trunk with law $\mathbb Q_{y,v,r}$.
For $j=1,2$, let $\mathbb Q_{q,j;y,t}$ be the resulting law after
attaching $j$ independent continuations with law
$\mathbb Q_{v,\mathbf0,t-r}$, independent also of the trunk conditional
on $(r,v)$.  This gives a law on the diagrams of
Definition~\ref{def:trunk-arm-diagram}.  The split factor
$(\Lambda_v)_+/\bar\lambda$ enters only the environmental weight; it does
not enter the bridge law or count as another visit.
\begingroup
We use these laws to average the environmental correction over a single diagram.
\endgroup

\endgroup

{The next lemma bounds the product of the separately averaged
kernel factors in \eqref{eq:fresh-qnormalizer} by the branching sum.}

\begin{lemma}
\label{lem:scalar-type-domination}
For every $x,v,y$, $r,h\ge0$, and $q\in(0,1)$,
\begin{equation}
 \bar\lambda\widehat K_r(x,v)\widehat K_h(v,y)^{1+q}
 \le\int\mathbf K_r^P\bigl(x;(v,i)\bigr)i_+
              \mathbf k_h\bigl((v,i),y\bigr)^{1+q}\,P(di).
 \label{eq:scalar-type-domination}
\end{equation}
\end{lemma}

\begin{proof}
The detailed-balance identity \eqref{eq:typed-detailed-balance} gives
$\mathbf K_r^P(x;(v,i))=\mathbf k_r((v,i),x)$.
For fixed $r,h,x,v,y$, both
$i\mapsto\mathbf k_r((v,i),x)$ and
$i\mapsto\mathbf k_h((v,i),y)$ are increasing functions of
$i$, because each term in the path expansion contains the initial
holding factor $e^{-(1-i)s}$.
Apply positive association (Lemma~\ref{lem:onevariableassociation} below)
first to $i_+$ and the product of the two kernel factors, and then to those
two factors.  Jensen's inequality gives the last step:
\begin{align*}
 &\int i_+\,\mathbf k_r\bigl((v,i),x\bigr)
                 \mathbf k_h\bigl((v,i),y\bigr)^{1+q}\,P(di)\\
 &\quad\ge\bar\lambda
 \left(\int\mathbf k_r\bigl((v,i),x\bigr)\,P(di)\right)
 \left(\int\mathbf k_h\bigl((v,i),y\bigr)^{1+q}\,P(di)\right)\\
 &\quad\ge\bar\lambda\widehat K_r(x,v)\widehat K_h(v,y)^{1+q}.
\end{align*}
This proves \eqref{eq:scalar-type-domination}.
\end{proof}

\begingroup
We return to \eqref{eq:fresh-qnormalizer}.  Under
$\mathbb Q_{q,1;y,t}$, let $J_{\rm tr}$ be the trunk mutation count.
The next proposition bounds the normalizing constant, the trunk duration,
and $J_{\rm tr}$.  These estimates place
$(d_{\mathrm H}(v,\mathbf0)/d,(t-r)/d)$ in the regular region
$\mathcal R$, except on an event of exponentially small probability.

\begin{proposition}
\label{prop:fractional-short-trunk}
Under Assumption~\ref{ass:environment}, for every fixed $M<\infty$,
uniformly for $|t-du_*(m/d)|\le C_0\log d$ and
$d_{\mathrm H}(y,\mathbf{x}_m)\le M$,
\begin{equation}
 \widehat{\mathcal C}^{I}_{q,t}(y)
 \asymp\widehat K_t(y,\mathbf0)^{1+q}.
 \label{eq:qnormalizer-bounds}
\end{equation}
There are fixed $C,c>0$ such that
\begin{equation}
 \mathbb Q_{q,1;y,t}(r>s)\le Ce^{-cs}+e^{-cd},\qquad s\ge0.
 \label{eq:qtrunk-time-tail}
\end{equation}
For every fixed $a>0$, there are $C_a,c_a>0$ such that
\begin{equation}
 \mathbb E_{\mathbb Q_{q,1;y,t}}\!\left[
 e^{aJ_{\rm tr}}\one_{\{J_{\rm tr}>C_ad\}}\right]\le e^{-c_ad}.
 \label{eq:qtrunk-count-tail}
\end{equation}
Moreover, for every fixed $\zeta>0$, a sufficiently small
$\delta_*>0$ gives
\begin{equation}
 \mathbb Q_{q,1;y,t}(r\le\delta_*d,\ J_{\rm tr}\ge\zeta d)
 \le e^{-cd}.
 \label{eq:short-trunk-count-cutoff}
\end{equation}
The constants are independent of the subsequently fixed radius $M$;
only the lower threshold on $d$ may depend on $M$.
\end{proposition}

\begin{proof}
Since $|d_{\mathrm H}(y,\mathbf0)-m|\le M$ and $u_*$ is Lipschitz,
replacing $m$ by $d_{\mathrm H}(y,\mathbf0)$ changes the time band by
at most $CM\le\log d$ for all sufficiently large $d$.  Thus, we may
index the starting vertex by its actual distance and use $C_0+1$ in
place of $C_0$, independently of $M$.

\begin{itemize}
\item By \eqref{eq:scalar-type-domination} and \eqref{eqtag:FD2},
\begin{align*}
 \widehat{\mathcal C}^{I}_{q,t}(y)
 &\ll\int_0^{t-\delta d}
       \widehat{\mathcal B}_{t,r}^{(1+q)}(y)\dd r\\
 &\ll\widehat K_t(y,\mathbf0)^{1+q}
   +\widehat K_t(y,\mathbf0)\int_{\delta d}^\infty
          e^{-c_qh}\dd h
 \ll\widehat K_t(y,\mathbf0)^{1+q}.
\end{align*}
The last inequality uses the polynomial bounds \eqref{eqtag:FD5}.
\item For the reverse inequality, retain only $0\le r\le r_0$ and
zero-mutation trunks, for fixed $r_0>0$.  Their weight is $h(r)$.
{Proposition~\ref{prop:global-oneobject-shifts}} gives
$$
 \widehat{\mathcal C}^{I}_{q,t}(y)
 \ge2\bar\lambda\int_0^{r_0}h(r)\widehat K_{t-r}(y,\mathbf0)^{1+q}\dd r
 \gg\widehat K_t(y,\mathbf0)^{1+q}.
$$
\end{itemize}
Restricting the upper estimate to $r>s$ and dividing by the lower
estimate proves \eqref{eq:qtrunk-time-tail}, again by
\eqref{eqtag:FD5}.

The remaining bounds use only the exponential series.  The sum over
all endpoints of the trunk contribution with $j$ mutations is at most
$e^{-(1-\lambda_+)r}r^j/j!$, and
$\widehat K_{t-r}(v,\mathbf0)\le e^{\lambda_+(t-r)}$.
Since $t\le Cd$ and the normalizer in
\eqref{eq:qnormalizer-bounds} is bounded below by a power of $d$,
there is a fixed $C_1$ such that, for $a\ge0$, $\beta>0$, $N\ge0$,
and $0\le R_0\le t$,
\begin{equation}
 \mathbb E_{\mathbb Q_{q,1;y,t}}\!\left[
 e^{aJ_{\rm tr}}\one_{\{r\le R_0,\ J_{\rm tr}\ge N\}}\right]
 \le e^{C_1d}\int_0^{R_0}\sum_{j\ge N}\frac{(e^ar)^j}{j!}\dd r
 \le Cd\exp\left(C_1d-\beta N+e^{a+\beta}R_0\right).
 \label{eq:trunk-factorial-bound}
\end{equation}
Here $C_1$ is independent of $a,\beta,N,R_0$.  Take $\beta=1$,
$R_0=t\le Cd$, and $N=C_ad$ with $C_a$ sufficiently large to
obtain \eqref{eq:qtrunk-count-tail}.  For
\eqref{eq:short-trunk-count-cutoff}, take $a=0$, choose
$\beta$ so that $\beta\zeta>C_1+2$, and then choose
$\delta_*$ so that $e^\beta\delta_*<1$. 

To show that the continuation parameters lie in the regular region, put $\rho(v)=d_{\mathrm H}(v,\mathbf0)$
and let $L_u$ be a Lipschitz constant of $u_*$.  Since
$|\rho(v)-\rho(y)|\le J_{\rm tr}$, for all sufficiently large $d$,
\[
 \left|\frac{t-r}{d}-u_*\!\left(\frac{\rho(v)}d\right)\right|
 \le\frac{(C_0+1)\log d}{d}+\frac rd+L_u\frac{J_{\rm tr}}d.
\]
Choose $\zeta,\delta_*>0$ with
$\delta_*+L_u\zeta<\eta_{\rm rt}/2$, decreasing $\delta_*$ if needed for
\eqref{eq:short-trunk-count-cutoff}.  Equations~\eqref{eq:qtrunk-time-tail}
and~\eqref{eq:short-trunk-count-cutoff} then place
$(\rho(v)/d,(t-r)/d)$ in $\mathcal K_{\rm rt}\subset\mathcal R$ except on
an event of probability at most $e^{-cd}$ under $\mathbb Q_{q,1;y,t}$.

This completes the proof.
\end{proof}
\endgroup

\section{Path intersections and transferring to the static environment}
\label{sec:overlap-transfer}

\subsection{Intersection estimates for bridges and diagrams}
\label{sec:bridges}

\subsubsection{Setting up stages}

\begingroup
Section~\ref{sec:bridge-fractional-second-moment} gives branching estimates
for the independent-visit model.  To bound the environmental mean of the
static branching integral, we control repeated uses of a rate (i.e., revisits of a vertex) within one
trunk--arm diagram. We also control intersections of independent bridges. In both cases, the objects we study consist of segments that are ordinary Hamming bridges, after properly conditioning on endpoints, durations, and mutation counts.
\endgroup

Recall the space--time bridges in \eqref{eq:Qbridge} and the trunk--arm
diagrams in Definition~\ref{def:trunk-arm-diagram}.  In this section, the
starting vertex of a diagram is denoted by $x$.  The next definition counts how many
factors use each static rate.

\begin{definition}
\label{def:diagram-counts}
For a path $\pi=(v_0,\ldots,v_{|\pi|})$, define its discrete local time by
\begin{equation}
 L_z(\pi)=\sum_{k=0}^{|\pi|}\one_{\{v_k=z\}}.
 \label{eq:discrete-local-time}
\end{equation}
Write $\pi$ for the trunk from $x$ to $v$, and
$\gamma_1,\ldots,\gamma_j$ for the continuation paths from
$v$ to $\mathbf0$.
The number of factors in a diagram that use
the static variable $\Lambda_z$ is
\begin{equation}
 n_\Gamma(z)=L_z(\pi)+\sum_{a=1}^jL_z(\gamma_a)+\one_{\{v=z\}}.
 \label{eq:diagram-factor-count}
\end{equation}
{
% The last term counts the branching factor separately from the holdings.
The vertices used by the diagram form its support.  Define
\begin{equation}
\begin{split}
 \operatorname{Supp}(\Gamma)
 =\{z:n_\Gamma(z)>0\},\qquad
 I_1(\Gamma)=\sum_z\bigl(n_\Gamma(z)-1\bigr)_+.
\end{split}
 \label{eq:diagram-repeat-counts}
\end{equation}
Thus, $I_1$ counts all holdings beyond the first at each vertex,
including endpoints, together with the branching factor.
}

\end{definition}

{
Figure~\ref{fig:complete-diagram-counts} illustrates the conventions.
}

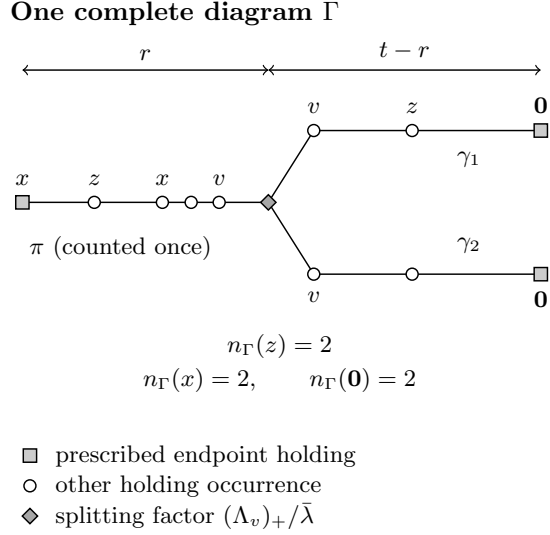
\begin{figure}[htbp]
\centering
\begin{tikzpicture}[
 x=1cm,y=1cm,font=\footnotesize,line width=.55pt,
 holding/.style={circle,draw,fill=white,inner sep=0pt,minimum size=4.5pt},
 prescribed/.style={rectangle,draw,fill=black!20,inner sep=0pt,
                    minimum size=5pt},
 splitting/.style={diamond,draw,fill=black!35,inner sep=0pt,
                   minimum size=6pt}]
 \node[anchor=west,font=\small\bfseries] at (.15,4.65)
   {{One complete diagram $\Gamma$}};

 % Distinct nodes record distinct holding factors, including the three
 % holdings at v.  The diamond records the additional splitting factor.
 \draw[<->] (.45,3.9)--node[above] {$r$} (3.7,3.9);
 \draw[<->] (3.7,3.9)--node[above] {$t-r$} (7.3,3.9);
 \draw (.45,2.15)--(3.7,2.15);
 \draw (3.7,2.15)--(4.3,3.1)--(7.3,3.1);
 \draw (3.7,2.15)--(4.3,1.2)--(7.3,1.2);
 \node[prescribed,label=above:{$x$}] at (.45,2.15) {};
 \node[holding,label=above:{$z$}] at (1.4,2.15) {};
 \node[holding,label=above:{$x$}] at (2.3,2.15) {};
 \node[holding] at (2.68,2.15) {};
 \node[holding,label=above:{$v$}] at (3.05,2.15) {};
 \node[splitting] at (3.7,2.15) {};
 \node[holding,label=above:{$v$}] at (4.3,3.1) {};
 \node[holding,label=above:{$z$}] at (5.6,3.1) {};
 \node[prescribed,label=above:{$\mathbf0$}] at (7.3,3.1) {};
 \node[holding,label=below:{$v$}] at (4.3,1.2) {};
 \node[holding] at (5.6,1.2) {};
 \node[prescribed,label=below:{$\mathbf0$}] at (7.3,1.2) {};
 \node at (1.75,1.55) {$\pi$ (counted once)};
 \node at (6.35,2.72) {$\gamma_1$};
 \node at (6.35,1.58) {$\gamma_2$};
 \node at (3.85,.25) {$n_\Gamma(z)=2$};
 \node at (3.85,-.2) {{$n_\Gamma(x)=2,\qquad n_\Gamma(\mathbf0)=2$}};

 \node[prescribed] at (.55,-1.2) {};
 \node[anchor=west] at (.75,-1.2) {prescribed endpoint holding};
 \node[holding] at (.55,-1.6) {};
 \node[anchor=west] at (.75,-1.6) {other holding occurrence};
 \node[splitting] at (.55,-2.0) {};
 \node[anchor=west] at (.75,-2.0)
   {splitting factor $(\Lambda_v)_+/\bar\lambda$};
\end{tikzpicture}
\caption{{Repeated uses of a rate in one diagram.  Each of
$z$, $x$, and $\mathbf0$ has two holdings and contributes $2-1=1$ to
$I_1$.  At the split $v$, three holdings and the branching factor
contribute $4-1=3$.}}
\label{fig:complete-diagram-counts}
\end{figure}

\begingroup
Proposition~\ref{prop:range} bounds the intersection probability of two
independent bridges by $O(d^{-1})$ after removing their endpoint vertices; this gives environmental concentration
for paths.  Proposition~\ref{prop:collisionmgf} bounds exponential moments
of repeated occurrences in a fixed family of bridges.  Together with
the mutation-count bounds in Proposition~\ref{lem:qdiagram-geometry},
it also bounds the environmental mean of one complete diagram.
\endgroup

\subsubsection{Discrete Hamming bridges}

Write
\[
 p_n(x,y)=\Pp_x(S_n=y),\qquad
 \Pp_{d,n;x,y}(\,\cdot\,)=\Pp_x(\,\cdot\mid S_n=y),
 \quad p_n(x,y)>0.
\]
We use only endpoint and count choices for which the bridge exists,
including the binary parity restriction.
For a bridge path $\pi$, write
$\operatorname{Supp}(\pi)=\{z:L_z(\pi)>0\}$ for its visited vertices.
For a space--time path $\Gamma$ with discrete path $\pi$, set
$\operatorname{Supp}(\Gamma)=\operatorname{Supp}(\pi)$.
The estimates for mutation counts proportional to the dimension, with
fixed positive comparison constants, remain valid at the geodesic
boundary $n=d_{\mathrm H}(x,y)$.

Lemma~\ref{lem:edgespread} bounds the probability of moving to a prescribed
next neighbour by the reciprocal of the remaining mutation count.

\begin{lemma}\label{lem:edgespread}
Fix $C_0<\infty$.  If $1\le n\le C_0d$, $p_n(x,y)>0$, and $w\sim x$, then
\begin{equation}
 \Pp_{d,n;x,y}(S_1=w)
 =\frac{p_{n-1}(w,y)}{(b-1)d\,p_n(x,y)}\ll_{C_0}\frac1n.
 \label{eq:edgespread}
\end{equation}
More generally, for $0\le k<n$ and $w\sim S_k$, almost surely,
\begin{equation}
 \Pp_{d,n;x,y}(S_{k+1}=w\mid S_0,\ldots,S_k)
 =\frac{p_{n-k-1}(w,y)}{(b-1)d\,p_{n-k}(S_k,y)}
 \ll_{C_0}\frac1{n-k}.
 \label{eq:edgespread-remaining}
\end{equation}
\end{lemma}

\begin{proof}
We first bound the expected number of uses of a fixed coordinate under
the endpoint conditioning.  Put $A=A_b$, $C=C_b$, and
$m=d_{\mathrm H}(x,y)$.  {By \eqref{eq:AC} and the one-coordinate case of
\eqref{eq:mutation-transition}, the power-series coefficients of $A$
and $C$ are the one-coordinate path weights.}  Interleaving the coordinate paths\footnote{A coordinate path records that coordinate's successive mutations,
omitting steps that change other coordinates.  Interleaving means
ordering the mutations of all coordinates while preserving their order
within each coordinate.  If coordinate $i$ makes $k_i$ mutations,
with $\sum_i k_i=n$, there are $n!/\prod_i k_i!$ such orders.
Conditional on the coordinate paths, they are equally likely: each
resulting Hamming-walk path has probability $((b-1)d)^{-n}$, and
they all have the same endpoint.  Thus, the first mutation uses coordinate
$i$ with probability $k_i/n$.} shows
that, if $N_i$ counts mutations of coordinate $i$, then
\[
 \Pp\bigl((N_i)=(k_i)\bigm| S_n=y\bigr)
 =\frac{\prod_i[z^{k_i}]f_i(z)}{Z_n},\qquad
 Z_n=[z^n]A(z)^mC(z)^{d-m}>0,\qquad \sum_i k_i=n,
\]
where $f_i=A$ when $x_i\ne y_i$ and $f_i=C$ otherwise.
The means are equal within each of these two coordinate classes.
If both classes occur, take $i$ in the first and $j$ in the second,
and put $B=A^{m-1}C^{d-m-1}$.  Differentiating the coordinate factors
$f_i=A$ and $f_j=C$, respectively, gives
\[
 \E[N_i-N_j\mid S_n=y]
 =\frac{[z^n]z(A'C-AC')B}{Z_n},\qquad
 A'C-AC'=\frac{e^{c_bz}}{b-1},\quad c_b=\frac{b-2}{b-1}.
\]
We compare coefficients of nonnegative power series.
\begin{itemize}
\item For $b>2$,
$[z^s]AC\ge[z^s]A$ and, for every $s\ge1$,
\[
 [z^s]A\ge\frac{b-2}{b(b-1)s!},\qquad
 [z^s]\frac{ze^{c_bz}}{b-1}
 =\frac{s c_b^{s-1}}{(b-1)s!}.
\]
Since $0<c_b<1$, these inequalities bound every coefficient of
$z(A'C-AC')$ by $K_b$ times the corresponding coefficient of $AC$,
for a finite $K_b$.
\item For $b=2$, the same conclusion follows from
$z(A'C-AC')=z$ and $AC=\sinh(2z)/2$, with $K_2=1$.
\end{itemize}
Multiplication by $B$ preserves the comparison.  The two class means
therefore differ by a number in $[0,K_b]$.  Their weighted average is
$n/d$, so every coordinate mean is at most $n/d+K_b$.
If only one class occurs, symmetry gives each mean exactly $n/d$;
this also covers $d=1$.  These coefficient arguments apply to every
$Z_n>0$, including binary parity and geodesic counts.

Conditional on the coordinate paths, their interleaving is uniform.
If $w$ differs from $x$ in coordinate $i$, then
\[
 \Pp_{d,n;x,y}(S_1=w)
 \le\Pp(\text{the first mutation uses coordinate }i\mid S_n=y)
 =\frac{\E[N_i\mid S_n=y]}n\ll_{C_0}\frac1n.
\]
The ordinary Markov property, conditioned on $S_n=y$, gives the
transition identity in \eqref{eq:edgespread-remaining}.
The first bound applies to the remaining bridge and gives the claimed
conditional estimate.
\end{proof}

Unless a probability weight is written explicitly, every
$\ell^p$ norm below is taken with respect to counting measure on the underlying state space. Our next result shows that two independent bridges intersect away from their endpoints with probability \(O(d^{-1})\), uniformly in the endpoints.

\begin{proposition}\label{prop:range}
Fix $0<\nu_1<\nu_2<\infty$.  Every admissible Hamming bridge from $x$
to $y$ with
$$
 \nu_1d\le n\le\nu_2d,\qquad d_{\mathrm H}(x,y)\le n,
$$
including the admissible geodesic boundary $n=d_{\mathrm H}(x,y)$, has marginal distributions
$\mu_k(z)=\Pp_{d,n;x,y}(S_k=z)$ satisfying
\begin{equation}
 \sum_{k=1}^{n-1}\|\mu_k\|_2\ll_{\nu_1,\nu_2}d^{-\frac12}.    \label{eq:sliceL2}
\end{equation}
Consequently, if $\pi_i$ is a bridge from $x_i$ to $y_i$, $i=1,2$,
and the two bridges are independent with mutation counts $\asymp d$,
then after their
endpoint vertex labels are removed,
\begin{equation}
 \Pp\Bigl(\bigl(\operatorname{Supp}(\pi_1)\cap\operatorname{Supp}(\pi_2)\bigr)
              \setminus\{x_1,y_1,x_2,y_2\}\ne\varnothing\Bigr)
 \ll \frac1d.                                              \label{eq:crosscap}
\end{equation}
All constants are uniform in the endpoints subject to the admissibility
and mutation-count conditions stated above.
\end{proposition}

\begin{proof}
For $k\le n/2$ and $1\le j\le k$, condition on $S_{k-j}=w$.
During the next $j$ mutations, at least $n-k\ge\nu_1d/2$ mutations remain
before the prescribed endpoint.  Lemma~\ref{lem:edgespread} therefore
bounds every prescribed path of length $j$ by $(C/d)^j$.
The ordinary walk gives each such path probability $\bigl((b-1)d\bigr)^{-j}$.
Summing over paths from $w$ to $z$ yields
\[
 \Pp_{d,n;x,y}(S_k=z\mid S_{k-j}=w)\le C^j p_j(w,z).
\]
Average over $w$ and apply the triangle inequality.  Symmetry and vertex
transitivity then give
\[
 \|\mu_k\|_2
 \le\sum_w\Pp_{d,n;x,y}(S_{k-j}=w)\,
       C^j\left(\sum_z p_j(w,z)^2\right)^{1/2}
 =C^j p_{2j}(x,x)^{1/2}.
\]
{
A return after $2j$ mutations uses at most $j$ coordinates: a coordinate
used exactly once cannot return to its initial value.  For $1\le j\le3$ and $d\ge3$,
cover the used coordinates by a $j$-element set and take a union bound:
\begin{equation}
 p_{2j}(x,x)\le\binom dj\left(\frac jd\right)^{2j}
 \ll_j d^{-j}.
 \label{eq:returnbound}
\end{equation}
Taking $j=\min\{k,3\}$ gives
$$
 \|\mu_1\|_2\ll d^{-1/2},\qquad
 \|\mu_2\|_2\ll d^{-1},\qquad
 \sup_{3\le k\le n/2}\|\mu_k\|_2\ll d^{-3/2}.
$$
The first half therefore sums to $O(d^{-1/2})$, since $n\le\nu_2d$.
Time reversal treats the second half and proves \eqref{eq:sliceL2}.
}

For two independent bridges with marginal distributions $(\mu_k)$ and
$(\nu_\ell)$ at their interior steps, Cauchy--Schwarz gives
\[
 \sum_{k,\ell}\sum_z\mu_k(z)\nu_\ell(z)
 \le\left(\sum_k\|\mu_k\|_2\right)
     \left(\sum_\ell\|\nu_\ell\|_2\right)\ll d^{-1}.
\]
Deleting endpoint vertex labels can only reduce this intersection local
time.  Markov's inequality proves \eqref{eq:crosscap}.
\end{proof}

To control repeated visits and intersections, we first bound how often a
bridge visits the neighbours of a fixed vertex.
\begin{lemma}\label{lem:bridgestar}
For every fixed $a,C_0>0$,
\begin{equation}
 \sup_{x,y,z}\ \sup_{\substack{n\le C_0d\\p_n(x,y)>0}}
 \E_{\Pp_{d,n;x,y}}\!\left[
 \exp\left(a\sum_{k=0}^n\one_{\{S_k\sim z\}}\right)\right]
 \ll_{a,C_0}1.                                           \label{eq:bridgestar}
\end{equation}
\end{lemma}

\begin{proof}
If $n\le1$, the sum in \eqref{eq:bridgestar} has at most two terms, so its
exponential is at most $e^{2a}$.  Hence, assume $n\ge2$.
Cut the path after $\lfloor n/2\rfloor$ mutations and first reveal the
initial half from $x$.  At each reveal at least $n/2$ mutations remain,
so \eqref{eq:edgespread-remaining} bounds every prescribed next neighbour
by $C/n$.

For $1\le k\le\lfloor n/2\rfloor$, put
$$
 B_k=\one_{\{S_k=z\}}
     +\one_{\{S_{k-1}\ne z,\ S_k\sim z\}}.
$$
For $u\ne z$, the common-neighbour count satisfies
$\bigl|\{w:w\sim u,\ w\sim z\}\bigr|\le\max(2,b-2)$.  It follows that
$\Pp(B_k=1\mid\mathcal F_{k-1})\le C/n$.
Charge a visit to a neighbour of $z$ that does not follow $z$ to $B_k$.
A visit immediately following $z$ is charged to $B_{k-1}$ instead.
Only boundary visits can remain uncharged.  Therefore,
$$
 \sum_{k\,\mathrm{in\ the\ half}}\one_{\{S_k\sim z\}}
 \le2+\sum_kB_k.
$$
Iterated conditional expectation gives
\begin{equation}
 \E\bigl[e^{a\sum_kB_k}\bigr]
 \le\left(1+\frac Cn(e^a-1)\right)^{\frac n2}
 \le e^{C(e^a-1)}.
                                                               \label{eq:starBernoulli}
\end{equation}
This bounds the exponential moment of the initial half's neighbour count.
Revealing the terminal half backwards from $y$ gives the same bound.
Writing $H_-,H_+$ for these two counts and applying the bound with
parameter $2a$, Cauchy--Schwarz gives
\begin{equation}
 \E\bigl[e^{a(H_-+H_+)}\bigr]
 \le\E\bigl[e^{2aH_-}\bigr]^{\frac12}\E\bigl[e^{2aH_+}\bigr]^{\frac12}\le C.       \label{eq:starhalves}
\end{equation}
{This proves \eqref{eq:bridgestar}.}
\end{proof}

\paragraph{A conditional counting bound.}
We will apply the following estimate to hits on another bridge and
to returns within one bridge.  It bounds the exponential moment of the
hit count using bounds on the conditional probability of each hit.
Let $(\mathcal F_k)$ be a filtration.  For integers $1\le\ell\le n$
and $0\le k<\ell$, let $B_{k+1}\in\{0,1\}$ be
$\mathcal F_{k+1}$-measurable.  Suppose
$p_k:=\E[B_{k+1}\mid\mathcal F_k]\le CH_k/n$ for nonnegative
$\mathcal F_k$-measurable $H_k$.  For every $a>0$ and $\ell\ge1$,
\begin{equation}
 \E\!\left[\exp\left(a\sum_{k<\ell}B_{k+1}\right)\right]
 \le\left(\sup_{k<\ell}\E\bigl[e^{C(e^{2a}-1)H_k}\bigr]\right)^{1/2}.
 \label{eq:hitHolder}
\end{equation}
Indeed, the conditional Bernoulli identity and $1+x\le e^x$ make
$\exp\bigl(2a\sum_{k<j}B_{k+1}-(e^{2a}-1)\sum_{k<j}p_k\bigr)$ a
nonnegative supermartingale with initial value one.  Cauchy--Schwarz gives
$$
 \E\!\left[e^{a\sum_{k<\ell}B_{k+1}}\right]
 \le\left(\E\!\left[
 e^{(e^{2a}-1)\sum_{k<\ell}p_k}\right]\right)^{1/2}.
$$
Using $p_k\le CH_k/n$, apply convexity with weight $1/n$ at each $H_k$
and the remaining weight $1-\ell/n$ at zero.  This gives
\begin{equation}
 \exp\left(\frac{C(e^{2a}-1)}n\sum_{k<\ell}H_k\right)
 \le1-\frac\ell n+\frac1n\sum_{k<\ell}e^{C(e^{2a}-1)H_k}.
 \label{eq:hitcompensator}
\end{equation}
The right side of \eqref{eq:hitcompensator} has expectation at most
$\sup_{k<\ell}\E\bigl[e^{C(e^{2a}-1)H_k}\bigr]$, since this supremum is at least
one.  This proves \eqref{eq:hitHolder}; with no steps the count is zero.

\begin{lemma}\label{lem:bridgehit}
Let $X,Y$ be independent Hamming bridges with fixed endpoints and lengths
at most $C_0d$.  Delete any specified initial or terminal occurrences
that are prescribed common endpoints, and put
\[
 \mathcal R(Y)=
 \bigl\{Y_j:0\le j\le |Y|,\ Y_j\text{ is not a deleted endpoint occurrence}\bigr\}.
\]
Deletion here concerns the indexed occurrence, not its vertex label:
a later revisit remains.  In \eqref{eq:crosscap}, instead, the endpoint
vertices are removed from the supports entirely.
Then, for every fixed $a>0$,
\begin{equation}
 \E\!\Biggl[\exp\biggl(a\sum_{\substack{0\le k\le |X|\\
                 X_k\text{ is not a deleted endpoint occurrence}}}
                   \one_{\{X_k\in\mathcal R(Y)\}}\biggr)\Biggr]
 \ll_{a,C_0}1.                                           \label{eq:bridgehit}
\end{equation}
\end{lemma}

\begin{proof}
Deleting the specified occurrences only decreases the hit count, so we
bound hits on the full range of $Y$.  If $n:=|X|\le1$, there are at most
two hits.  Otherwise split $X$ in halves and expose either half from its
external endpoint, revealing the whole independent bridge $Y$ at time zero.
At least $n/2$ mutations remain at each reveal.  Write
$N(u)=\{z:z\sim u\}$ and $H_Y(u)=\sum_j\one_{\{Y_j\in N(u)\}}$.
For $B_{k+1}=\one_{\{X_{k+1}\in\{Y_0,\ldots,Y_{|Y|}\}\}}$,
Lemma~\ref{lem:edgespread} gives
\[
 \E[B_{k+1}\mid Y,\mathcal F_k^X]\le\frac CnH_Y(X_k).
\]
Conditional on $X_k=u$, independence leaves $Y$ an ordinary bridge.
Lemma~\ref{lem:bridgestar} therefore gives
$\sup_k\E\bigl[e^{\beta H_Y(X_k)}\bigr]\ll_{\beta,C_0}1$ for every fixed $\beta>0$.
Applying \eqref{eq:hitHolder} under the joint law bounds the exponential
moment of the hits along either half's transitions.  Each external
endpoint adds at most one hit.  Cauchy--Schwarz, using the half bounds
at $2a$, proves \eqref{eq:bridgehit}.
\end{proof}

\subsubsection{Space--time bridges and environmental weights}

Recall the independent-visit bridge law in \eqref{eq:Qbridge}.
For the rate at $v$ and a holding time $r$, put
$$
 Y_v(r)=\frac{e^{-(1-\Lambda_v)r}}{h(r)}.
$$
Then $\Eenv[Y_v(r)]=1$.  For every $\varepsilon>0$, Assumption~\ref{ass:environment} gives
\begin{equation}
 p_\varepsilon:=\Penv(\Lambda_{\mathbf0}\ge\lambda_+-\varepsilon)>0,
 \qquad
 h(r)\ge p_\varepsilon e^{-(1-\lambda_++\varepsilon)r},
 \qquad
 Y_v(r)\le p_\varepsilon^{-1}e^{\varepsilon r}.
 \label{eq:quantile-holding-bound}
\end{equation}
The factor $e^{\varepsilon r}$ in \eqref{eq:quantile-holding-bound}
cannot in general be discarded.  For example,
if the rate law is uniform on $[-L,\lambda_+]$, with
$L>0$, then, as $r\to\infty$,
\[
 h(r)\sim\frac{e^{-(1-\lambda_+)r}}{(\lambda_++L)r},
 \qquad
 \Eenv\bigl[Y_v(r)^2\bigr]=\frac{h(2r)}{h(r)^2}
 \sim\frac{(\lambda_++L)r}{2}.
\]
\begingroup
Thus, the expected product of two holding ratios need not be bounded
uniformly in their durations.  To deal with this issue, in the proof of
Lemma~\ref{lem:abstract-environment-concentration} below, we integrate the
holding times before estimating these products.  This leaves only a
bounded cost per revisit, together with a small exponential
in the total trunk duration for diagrams.
\endgroup

Recall from Lemma~\ref{lem:clockLCLT} that
$f_\kappa(r)=e^{-\kappa r}h(r)/H(\kappa)$ has mean
$\mu_\kappa=-H'(\kappa)/H(\kappa)$, and that
$\widetilde f_{N,\kappa}$ is the density of the sum of $N$ independent
variables with density $f_\kappa$.
For $N\ge2$ and $T>0$, conditioning their sum to equal $T$ means the
following density relative to $\dd r_1\cdots\dd r_{N-1}$, on
$r_i\ge0$ and $\sum_{i<N}r_i\le T$, with
$r_N=T-\sum_{i<N}r_i$:
$$
 \frac{\prod_{i=1}^N f_\kappa(r_i)}
      {\widetilde f_{N,\kappa}(T)}
 =\frac{\prod_{i=1}^N h(r_i)}{h^{*N}(T)}.
$$
The denominator is positive for every $T>0$.  For $N=1$, the
conditional variable is $R_1=T$.  The tilt cancels at fixed total
duration, so this is the independent-visit holding-time law.

The repetition estimate below requires conditional independence of the
discrete bridge segments.  For $Q\in\mathbb N$ and $C_{\rm len}<\infty$,
consider a family $\gamma_1,\ldots,\gamma_Q$, where $Q$ counts the segments
after any cuts.  Let $x_i,y_i,n_i$ be the prescribed start, endpoint, and
mutation count of $\gamma_i$.  Call the family $(Q,C_{\rm len})$-admissible
if, almost surely for feasible prescribed data,
$$
 \mathcal L\!\left((\gamma_i)_{i=1}^Q
 \mid(x_i,y_i,n_i)_{i=1}^Q\right)
 =\bigotimes_{i=1}^Q\Pp_{d,n_i;x_i,y_i},
 \qquad 0\le n_i\le C_{\rm len}d.
$$
Given these data, the segments are independent ordinary Hamming bridges.
In particular, a fixed finite collection of independent ordinary Hamming
bridges of deterministic lengths at most $C_{\rm len}d$ is admissible;
the conditioning is then vacuous.  A segment shared by several diagrams
may be retained once before repeated occurrences are counted, provided
the resulting family satisfies this same conditional product law.
Using the discrete local time in \eqref{eq:discrete-local-time}, put
$$
 \widetilde I_Q=\sum_v\left(\sum_{j=1}^QL_v(\gamma_j)-1\right)_+.
$$
This counts all occurrences beyond the first at each vertex across the
family, including returns within one path and meetings of different paths.
\begin{proposition}
\label{prop:collisionmgf}
For every fixed $Q\in\mathbb N$, $C_{\rm len}<\infty$, and $a>0$, every
$(Q,C_{\rm len})$-admissible family satisfies
\begin{equation*}
 \sup_d\E\bigl[e^{a\widetilde I_Q}\bigr]\ll_{a,Q,C_{\rm len}}1.
\end{equation*}
The conclusion remains true after adding a fixed number of branching
factors $(\Lambda_v)_+/\bar\lambda$ to the factor count at their split vertices.
\end{proposition}

\begin{proof}
We first count repetitions within one bridge.  The case $n\le1$ is
immediate, so take a bridge of length $n\ge2$ and expose a half from its
external endpoint.  At every exposed step at least $\frac n2$ bridge
steps remain.  Let
$$
 R_{k+1}=\one_{\{S_{k+1}\in\{S_0,\ldots,S_k\}\}},
 \qquad H_k(u)=\sum_{j=0}^k\one_{\{S_j\in N(u)\}}.
$$
If $S_k=u$, every possible repeated next vertex belongs to
$N(u)\cap\{S_0,\ldots,S_k\}$.  Lemma~\ref{lem:edgespread}, applied with the
remaining length, therefore gives the predictable bound
\begin{equation}
 p_k:=\Pp(R_{k+1}=1\mid\mathcal F_k)
 \le \frac{C}{n} H_k(S_k).                              \label{eq:repeathazard}
\end{equation}
Conditional on $S_k=u$, the exposed prefix is an ordinary bridge ending at
$u$, and $H_k(u)$ is its occupation of the neighbour set $N(u)$.
Lemma~\ref{lem:bridgestar} bounds $\E\bigl[e^{\beta H_k(S_k)}\bigr]$ uniformly for
every fixed $\beta>0$.  Apply \eqref{eq:hitHolder} with
\eqref{eq:repeathazard} to obtain
\[
 \sup_{k\le\lceil n/2\rceil}
 \E\!\left[\exp\left(a\sum_{j<k}R_{j+1}\right)\right]
 \ll_{a,C_{\rm len}}1.
\]
Time reversal gives the same estimate for the other half.  Conditional
on the midpoint, the halves are independent, so
Lemma~\ref{lem:bridgehit} bounds their visits to each other's ranges.
Let $R_-,R_+$ count repetitions within the two halves, let
$H_{-+},H_{+-}$ be the two cross-hit counts after prescribed endpoint
occurrences are deleted, and let
$C_{\rm bd}$ account for the fixed midpoint boundary occurrences.  Pathwise,
\begin{equation}
 \sum_v\bigl(L_v(\gamma)-1\bigr)_+
 \le R_-+R_++H_{-+}+H_{+-}+C_{\rm bd}.                  \label{eq:rangeDeficitDecomp}
\end{equation}
Each of the first four variables has a bounded exponential moment at every
fixed parameter.  H\"older's inequality applied to
\eqref{eq:rangeDeficitDecomp} therefore gives, for every $a>0$,
$$
 \E\!\left[e^{a\sum_v(L_v(\gamma)-1)_+}\right]
 \le e^{aC_{\rm bd}}
 \prod_{X\in\{R_-,R_+,H_{-+},H_{+-}\}}
       \E\bigl[e^{4aX}\bigr]^{\frac14}\ll_{a,C_{\rm len}}1.
$$

Now order the conditionally independent paths
$\gamma_1,\ldots,\gamma_Q$.  On adding $\gamma_j$, the increase in
$\widetilde I_Q$ is at most the number of its own repeated occurrences,
$\sum_v(L_v(\gamma_j)-1)_+$, plus
$$
 \sum_{i<j}\sum_k
 \one_{\{\gamma_j(k)\in\mathcal R(\gamma_i)\}}.
$$
The self-repetition count has the one-bridge bound just proved, and each
cross-hit count in the sum has the bound of Lemma~\ref{lem:bridgehit}.
There are at most
$R=Q+\binom Q2$ variables in the resulting upper bound for
$\widetilde I_Q$; enumerate them as $X_1,\ldots,X_R$.  H\"older's
inequality gives
\begin{equation*}
 \E\bigl[e^{a\widetilde I_Q}\bigr]
 \le\prod_{j=1}^{R}\E\bigl[e^{aR X_j}\bigr]^{\frac1R}
 \ll_{a,Q,C_{\rm len}}1,
\end{equation*}
uniformly in the prescribed bridge data.  A shared segment is counted once
before applying the bound.  Each additional branching factor at a split
vertex increases the count by at most one, hence multiplies the bound
by at most $e^a$.
\end{proof}

\begingroup
\subsubsection{Mutation counts and repeated factors in a diagram}
\label{sec:intersection-estimates}

For a diagram with law $\mathbb Q_{q,j;x,t}$, write $N_a$ for
the mutation count of arm $a$, $1\le a\le j\in\{1,2\}$.

\begin{proposition}
\label{lem:qdiagram-geometry}
Fix $C_0,M<\infty$.  Uniformly for
\begin{equation}
 1\le m\le d,\qquad d_{\mathrm H}(x,\mathbf{x}_m)\le M,\qquad
 \left|t-du_*\!\left(\frac md\right)\right|\le C_0\log d,
 \label{eq:qdiagram-parameters}
\end{equation}
and $j\in\{1,2\}$, the following hold for all sufficiently large $d$.
\begin{enumerate}
\renewcommand{\labelenumi}{(\roman{enumi})}
\item For every fixed $a>0$, there are $C_a,c_a>0$ such that
\begin{equation}
 \E_{\mathbb Q_{q,j;x,t}}\!\left[
 e^{a(J_{\rm tr}+\sum_{e=1}^jN_e)}
 \one_{\{J_{\rm tr}+\sum_{e=1}^jN_e>C_ad\}}\right]\le e^{-c_ad}.
 \label{eq:qdiagram-counttail}
\end{equation}
\item There are fixed $\zeta,\delta_*,\varepsilon>0$ and a compact
$\mathcal K\subset\mathcal R$ such that the event
\begin{equation}
 \mathcal G_\Gamma
 =\{r\le\delta_*d,\ J_{\rm tr}<\zeta d\}
 \cap\bigcap_{a=1}^j
 \left\{N_a\in\mathcal I_d^{J_{\mathcal K},\varepsilon}
       \bigl(d_{\mathrm H}(v,\mathbf0),t-r\bigr)\right\}
 \label{eq:qdiagram-retained-event}
\end{equation}
satisfies $\mathbb Q_{q,j;x,t}(\mathcal G_\Gamma^c)\ll e^{-cd}$.
On this event $(d_{\mathrm H}(v,\mathbf0)/d,(t-r)/d)\in\mathcal K$,
and every arm has $N_a\asymp d$ and a clock tilt in a fixed compact
subset of $(\lambda_+-1,\infty)$.
\end{enumerate}
The constants are independent of $M$, except for the lower threshold
on $d$.
\end{proposition}

\begin{proof}
\noindent\emph{(i)} Equation~\eqref{eq:qtrunk-count-tail} bounds the
trunk count.  For an arm, put $h=t-r$ and
$\ell=d_{\mathrm H}(v,\mathbf0)$, so that
$\delta d\le h\le Cd$.  Choose $L$ with $P([-L,\lambda_+])>0$.
The geodesic contribution gives
$$
 \widehat K_{d,\ell}(h)
 \ge P([-L,\lambda_+])^{\ell+1}e^{-(1+L)h}
       \left(\frac h{(b-1)d}\right)^\ell\ge e^{-Cd},
$$
Using \eqref{eq:UO17} for the numerator summands, for every fixed $a'>0$,
$$
 \E[e^{a'N_a}\one_{\{N_a>Rd\}}\mid r,v]
 \le e^{Cd}\sum_{n>Rd}\frac{(e^{a'}h)^n}{n!}
 \le\exp\left(Cd-Rd+e^{a'+1}h\right).
$$
This is exponentially small for sufficiently large $R$.
For the $j+1$ component counts $N_i$, use
$$
 e^{a\sum_iN_i}\one_{\{\sum_iN_i>C_ad\}}
 \le\sum_i e^{a(j+1)N_i}
           \one_{\{N_i>C_ad/(j+1)\}}.
$$
The component bounds at $a'=(j+1)a$ prove
\eqref{eq:qdiagram-counttail}.

\noindent\emph{(ii)} Choose $\zeta,\delta_*>0$ small enough that
$$
 \bigl|d_{\mathrm H}(v,\mathbf0)-m\bigr|\le\zeta d+M,
 \qquad
 \left|\frac{t-r}d-u_*(m/d)\right|
 \le\delta_*+\frac{C_0\log d}d
$$
place the pair $(d_{\mathrm H}(v,\mathbf0)/d,(t-r)/d)$ in one compact
$\mathcal K\subset\mathcal R$.  Decrease $\delta_*$ as required by
\eqref{eq:short-trunk-count-cutoff}, and choose
$0<\varepsilon\le\varepsilon_{\mathcal K}$ in
\eqref{eq:UO2}--\eqref{eq:UO4}.  Conditional on the split time $r$ and
vertex $v$, each arm has law $\mathbb Q_{v,\mathbf0,t-r}$ and is
independent of the trunk, by the construction following
\eqref{eq:weighted-split-law}.  For
$(d_{\mathrm H}(v,\mathbf0)/d,(t-r)/d)\in\mathcal K$,
\eqref{eq:UO2} gives
$$
 \mathbb Q_{q,j;x,t}\!\left(
 N_a\notin\mathcal I_d^{J_{\mathcal K},\varepsilon}
       \bigl(d_{\mathrm H}(v,\mathbf0),t-r\bigr)\mid r,v
 \right)\le Ce^{-cd}.
$$
On $\{r\le\delta_*d,\ J_{\rm tr}<\zeta d\}$, the pair lies in
$\mathcal K$.  Summing over the arms and averaging over the split and
trunk therefore gives
$$
 \mathbb Q_{q,j;x,t}(\mathcal G_\Gamma^c)
 \le\mathbb Q(r>\delta_*d)
   +\mathbb Q(r\le\delta_*d,\ J_{\rm tr}\ge\zeta d)
   +jCe^{-cd}\ll e^{-cd}.
$$
The first two terms use \eqref{eq:qtrunk-time-tail} and
\eqref{eq:short-trunk-count-cutoff}; the arm-count and tilt conclusions
follow from \eqref{eq:UO4}.  All constants are independent of $M$ once
$d$ is sufficiently large.
\end{proof}
\endgroup

\subsection{Transferring to the static environment}
\label{sec:static-transfer}

\subsubsection{Recap: differentiating static and independent-visit models}

Let us start by recalling the quantitative difference between the static and independent-visit models.  
In the static environment, the rate $\Lambda_z$ is sampled once at each
vertex and is reused on every visit to that vertex.  In the
independent-visit expansion, each holding occurrence is averaged
separately.

By the many-to-one formula (see
\cite[Section~2.3]{KoenigSurvey2021} for its relation to the
parabolic Anderson equation),
\[
 \Ew_x\bigl[N_t(y)\bigr]=K_t^\omega(x,y).
\]
If $\pi=(v_0,\ldots,v_n)$ is a discrete mutation path from $x$ to $y$ of
length $n$, conditioning on the path and integrating over its $n+1$
holding times gives the exact expansion of the continuous-time
Feynman--Kac kernel:
\begin{equation}
 K_t^\omega(x,y)
 =\sum_{n\ge0}\bigl((b-1)d\bigr)^{-n}\sum_{\pi:x\to y,\,|\pi|=n}
   \int_{\Delta_n(t)}
        \exp\left(-\sum_{j=0}^n(1-\Lambda_{v_j})r_j\right)
        \dd r_0\cdots\dd r_{n-1}.
                                                        \label{eq:quenched-path-expansion}
\end{equation}
To pass between this expansion and \eqref{eq:freshK}, we multiply
each independent-visit path weight by the ratios of its actual holding
factors to their separate averages, as follows.

For an independent-visit space--time path $\Gamma$, define
$$
 W_\Gamma^\omega=\prod_{a\in\Gamma}
 \frac{e^{-(1-\Lambda_{v_a})r_a}}{h(r_a)}.
$$
Each holding ratio is increasing in $\Lambda_{v_a}$, has
environment mean one, and satisfies \eqref{eq:quantile-holding-bound}.
Here $a$ indexes a holding occurrence, $v_a$ is its vertex,
and $r_a$ is its duration.
For a single path in a fixed environment, averaging $W_\Gamma^\omega$ under the deterministic
independent-visit bridge law gives exactly
$K_t^\omega(y,\mathbf0)/\widehat K_t(y,\mathbf0)$. 

\subsubsection{Concentration of path weights}

The goal of this section is to develop quenched and annealed estimates for the correction term that links the independent-visit model to the static model. We start from defining the key quantities that deal with the ratios in two different cases: space--time bridge and trunk--arm diagram. Recall also the trunk--arm diagram consists of two subcases: one or two arms; see Figure \ref{fig:sketch-shared-rate}.

 For a single space--time bridge $\Gamma$
from $x$ to $\mathbf0$, with law $\mathbb Q_{x,\mathbf0,t}$, put
\begin{equation}
 F_{\Gamma,z}^{\mathrm{path}}(\lambda)
 =\prod_{a:v_a=z}\frac{e^{-(1-\lambda)r_a}}{h(r_a)}.
 \label{eq:path-site-factor}
\end{equation}
This product collects the holding ratios on the vertex $z$. Recall
$\bar\lambda=\int(\lambda)_+P(d\lambda)$. 
For a trunk--arm diagram with law $\mathbb Q_{q,j;x,t}$,
$j\in\{1,2\}$, and split vertex $v$, let $a$ index the holding
occurrences of the trunk and all arms, and define
\begin{equation}
 F_{\Gamma,z}^{\mathrm{diagram}}(\lambda)
 =\begin{cases}
 \displaystyle\prod_{a:v_a=z}\frac{e^{-(1-\lambda)r_a}}{h(r_a)},
       &z\ne v,\\[6pt]
 \displaystyle\frac{(\lambda)_+}{\bar\lambda}
       \prod_{a:v_a=v}\frac{e^{-(1-\lambda)r_a}}{h(r_a)},
       &z=v,
 \end{cases}
 \label{eq:diagram-site-factor}
\end{equation}
For each of the three choices of $(\mathbb Q,F_{\Gamma,z})\in \{(\mathbb Q_{x,\mathbf0,t},F_{\Gamma,z}^{\mathrm{path}}),(\mathbb Q_{q,1;x,t},F_{\Gamma,z}^{\mathrm{diagram}}),(\mathbb Q_{q,2;x,t},F_{\Gamma,z}^{\mathrm{diagram}})\}$, write
\begin{equation}
 X^\omega=\E_{\mathbb Q}\!\left[\prod_zF_{\Gamma,z}(\Lambda_z)\right],
 \qquad
 B^{\rm end}=\Eenv\bigl[X^\omega\bigm|\Lambda_x,\Lambda_{\mathbf0}\bigr].
 \label{eq:environment-path-average}
\end{equation}
In the path case, $X^\omega=K_t^\omega(x,\mathbf0)/\widehat K_t(x,\mathbf0)$;
in the diagram case, with $\Gamma_0$ denoting the trunk and $\Gamma_1,\ldots,\Gamma_j$
denoting its arms,
\begin{equation}
X^\omega
 =\E_{\mathbb Q_{q,j;x,t}}\!\left[
          \prod_zF_{\Gamma,z}^{\mathrm{diagram}}(\Lambda_z)\right]
 =\E_{\mathbb Q_{q,j;x,t}}\!\left[
       \frac{(\Lambda_v)_+}{\bar\lambda}
       W_{\Gamma_0}^\omega\prod_{e=1}^jW_{\Gamma_e}^\omega\right].
 \label{eq:diagram-environment-average}
\end{equation}

\begin{lemma}
\label{lem:abstract-environment-concentration}
\begingroup
For $X^\omega,B^{\rm end}$ in \eqref{eq:environment-path-average}, the
following estimates hold.
\begin{enumerate}
\renewcommand{\labelenumi}{(\roman{enumi})}
\item If $F_{\Gamma,z}=F_{\Gamma,z}^{\mathrm{path}}$ and
$\mathbb Q=\mathbb Q_{x,\mathbf0,t}$, then for every integer $R\ge1$ and $s>1$,
\begin{equation}
 \Eenv\Bigl[\bigl|X^\omega-B^{\rm end}\bigr|^{2R}\Bigr]
 \ll_{R,s}d^{-R(1-\frac1s)},\qquad B^{\rm end}\ll1.
 \label{eq:abstract-high-moment}
\end{equation}
These bounds are uniform for $(d_{\mathrm H}(x,\mathbf0)/d,t/d)$ in
any fixed compact subset of $\mathcal R$.  For fixed $L\ge0$, also
$B^{\rm end}\asymp_L1$ when $\Lambda_x,\Lambda_{\mathbf0}\ge-L$.
The moment estimate holds conditional on the endpoint rates, uniformly
in those rates, and hence also unconditionally.
\item If $F_{\Gamma,z}=F_{\Gamma,z}^{\mathrm{diagram}}$ and
$\mathbb Q=\mathbb Q_{q,j;x,t}$, then $B^{\rm end}\ll1$,
uniformly for $j\in\{1,2\}$, the parameters in
\eqref{eq:qdiagram-parameters}, and the conditioned endpoint rates.
\end{enumerate}
Case (i) also covers \eqref{eq:qdiagram-parameters}.  Under those
conditions, all constants are independent of the subsequently fixed
radius $M$; only the lower threshold on $d$ may depend on $M$.
They may depend on the compact set, $C_0$, and the fixed $q,\delta$
in case (ii).
\endgroup

\end{lemma}

\begin{proof}
\begingroup
Throughout the proof, we condition on $\Lambda_x,\Lambda_{\mathbf0}$.
For a single path, put $n_\Gamma(z)=L_z(\pi)$; for a diagram, use the
count in Definition~\ref{def:diagram-counts}, including the branching factor.

The proof consists of three main parts.  Steps~1--3 bound the environmental mean $B^{\rm end}$ 
for both paths and diagrams by charging a constant for each repeated
use of a rate, after integrating the holding times.  In Step~4, we
apply a resampling argument to show that $X^\omega$ is concentrated near $B^{\rm end}$ in the path case.  Step~5 gives the lower bound on the path mean by
controlling the time spent at the two endpoints.

\medskip\noindent\textbf{Step 1: Count repeated uses of the same rate.}
For a fixed integer $p\ge1$, take independent paths or diagrams
$\Gamma_1,\ldots,\Gamma_p$ with law $\mathbb Q$, evaluated in the same
environment to estimate $\Eenv[(X^\omega)^p]$. Put
$$
 I_p=\sum_z\left(\sum_{i=1}^p n_{\Gamma_i}(z)-1\right)_+.
$$
Every split vertex has a holding occurrence, so $I_p$ is the full
excess holding count plus the number of branching factors.
In particular, it counts repetitions within each copy and meetings between copies. 
Let $N_i$ be the total mutation count of $\Gamma_i$.
A path has $N_i+1$ holding factors; a diagram has $N_i+j+1$ holding
factors and one branching factor.  Removing the copies with $N_i>Ld$
therefore gives
$$
 I_p\le
 \sum_z\left(\sum_{i:N_i\le Ld}n_{\Gamma_i}(z)-1\right)_+
 +\sum_{i=1}^p(N_i+4)\one_{\{N_i>Ld\}}.
$$
After conditioning, the remaining copies of paths/diagrams contain at most
$3p$ independent ordinary Hamming bridges of length at most $Ld$.
Proposition~\ref{prop:collisionmgf} then gives that for fixed $a>0$,
\begin{equation}
\begin{aligned}
 \E_{\mathbb Q^p}[e^{aI_p}]
 &=\E_{\mathbb Q^p}\!\left[
   \E_{\mathbb Q^p}\!\left[e^{aI_p}
       \mid\text{segment counts and split data}\right]\right]\\
 &\ll_{a,p,L}\E_{\mathbb Q^p}\!\left[
       \prod_{i=1}^p e^{a(N_i+4)\one_{\{N_i>Ld\}}}\right]\\
 &=\prod_{i=1}^p\left(
       \mathbb Q(N_i\le Ld)
       +e^{4a}\E_{\mathbb Q}[e^{aN_i}\one_{\{N_i>Ld\}}]\right)\\
 &\le\prod_{i=1}^p\left(
       1+e^{4a}\E_{\mathbb Q}[e^{aN_i}\one_{\{N_i>Ld\}}]\right)
 \ll_{a,p}1.,
\end{aligned}
 \label{eq:combined-factor-moment}
\end{equation}
where we have used independence between copies, and that for the final step, for sufficiently
large fixed $L=L(a,p)$, \eqref{eq:qdiagram-counttail} gives
$\E_{\mathbb Q}[e^{aN_i}\one_{\{N_i>Ld\}}]\le e^{-cd}$,
which proves the last bound.

\medskip\noindent\textbf{Step 2: Integrate the holding times to bound the repeated factors.}
Choose the fixed retained-count tolerance from
Propositions~\ref{prop:global-oneobject} and~\ref{prop:global-oneobject-local}
for paths, or Proposition~\ref{lem:qdiagram-geometry} for diagrams,
and denote it by $\varepsilon_{\mathcal K}$ in either case.
For paths, put
$$
 \mathcal G=\bigcap_{i=1}^p
 \left\{N_i\in\mathcal I_d^{J_{\mathcal K},\varepsilon_{\mathcal K}}
       \bigl(d_{\mathrm H}(x,\mathbf0),t\bigr)\right\}.
$$
For diagram $\Gamma_i$, write $r_i,v_i$ for its split time and vertex,
$J_{\rm tr}(\Gamma_i)$ for its trunk mutation count, and $\gamma_{i,a}$
for its $a$th arm. Expanding \eqref{eq:qdiagram-retained-event}, put
$$
\begin{aligned}
 \mathcal G=\bigcap_{i=1}^p\Bigl(&
   \{r_i\le\delta_*d,\ J_{\rm tr}(\Gamma_i)<\zeta d\}\\
 &{}\cap\bigcap_{a=1}^j
   \left\{|\gamma_{i,a}|
      \in\mathcal I_d^{J_{\mathcal K},\varepsilon_{\mathcal K}}
       \bigl(d_{\mathrm H}(v_i,\mathbf0),t-r_i\bigr)\right\}\Bigr).
\end{aligned}
$$
Here $|\gamma_{i,a}|$ is the arm's mutation count; in the path case,
set $r_i=0$. By \eqref{eq:UO2} and
Proposition~\ref{lem:qdiagram-geometry}, a union bound gives
$\mathbb Q^p(\mathcal G^c)\ll e^{-cd}$.
% These tolerances and $c>0$ stay fixed while the quantile parameter
% $\varepsilon$ in \eqref{eq:quantile-holding-bound} is chosen below.

\emph{On $\mathcal G$.}
Condition on the spatial paths, counts, split vertices and split times.
Each path or arm has $N\asymp d$ holdings (one more than its mutation
count), total duration $T=O(d)$,
and an exact tilt $\kappa$ satisfying $N\mu_\kappa=T$.
By \eqref{eq:UO4} and Proposition~\ref{lem:qdiagram-geometry},
these tilts lie in a fixed compact subset of $(\lambda_+-1,\infty)$,
so the uniform clock estimates of Lemma~\ref{lem:clockLCLT} apply.

Take $N$ independent variables with density
$f_\kappa(u)=e^{-\kappa u}h(u)/H(\kappa)$, then condition their sum
to equal $T$. This gives the holding-time law, where the density of the
first $N-1$ holdings, with respect to $\dd u_1\cdots\dd u_{N-1}$, is
$$
 \frac{\prod_{i=1}^N f_\kappa(u_i)}
      {\widetilde f_{N,\kappa}(T)}
 =\frac{\prod_{i=1}^N h(u_i)}{h^{*N}(T)},
 \qquad u_N=T-\sum_{i=1}^{N-1}u_i,\quad u_1,\ldots,u_N\ge0,
$$
where $\widetilde f_{N,\kappa}=f_\kappa^{*N}$.
% ; the exponential tilt
% cancels because the total is fixed. Holdings within a segment are now
% dependent, but the holding-time vectors of distinct segments remain
% independent given the counts and split data. Trunks keep their original
% conditional holding-time laws.

{
We apply the integration argument in \eqref{eq:1.5} under this conditioning.
Let $g_\kappa(u)=(\kappa+1-\lambda_+)
e^{-(\kappa+1-\lambda_+)u}$, $u\ge0$, be the exponential density.
For every $\lambda\le\lambda_+$,
\begin{equation}
 f_\kappa(u)\frac{e^{-(1-\lambda)u}}{h(u)}
 =\frac{e^{-(\kappa+1-\lambda)u}}{H(\kappa)}
 \le\frac{g_\kappa(u)}{H(\kappa)(\kappa+1-\lambda_+)}.
 \label{eq:integrated-holding-factor}
\end{equation}
The factor $1/(H(\kappa)(\kappa+1-\lambda_+))$ is uniformly bounded.
At each nonendpoint vertex visited by the
copies, choose one holding in a fixed order, without using its duration,
choosing a trunk holding whenever one is present. For the other factors, use the following bounds:
\begin{itemize}
\item For a holding on a path or arm, use
\eqref{eq:integrated-holding-factor}.
\item For a trunk holding, use
$Y_z(u)\le p_\varepsilon^{-1}e^{\varepsilon u}$ from
\eqref{eq:quantile-holding-bound}.  The trunk can be too short for
the retained-count clock estimates.
\item At $x,\mathbf0$, apply these respective bounds to every
path or arm holding and every trunk holding.  Bound every branching
factor by $\lambda_+/\bar\lambda$.
\end{itemize}
Only the chosen holding at each nonendpoint vertex still depends on
its rate, and $\int e^{-(1-\lambda)u}P(d\lambda)/h(u)=1$.
Independence of these rates therefore removes all the chosen ratios.
The number of factors to which a bound was applied is exactly
$$
 I_p+\#\left\{z\in\{x,\mathbf0\}:
                 \sum_{i=1}^p n_{\Gamma_i}(z)>0\right\}\le I_p+2.
$$
If $k$ holdings on a path or arm were bounded, integrating the resulting
density costs the constant in \eqref{eq:integrated-holding-factor}
to power $k$, times
$$
 \frac{(f_\kappa^{*(N-k)}*g_\kappa^{*k})(T)}
      {\widetilde f_{N,\kappa}(T)}\ll1,
 \qquad 0\le k\le N,
$$
where a zeroth convolution power is the point mass at zero.\footnote{Indeed, one of $N-k$ and $k$ is at least $N/2$.
Equation~\eqref{eq:clockconcentration}, or the gamma density bound
$\|g_\kappa^{*k}\|_\infty\ll k^{-1/2}$, bounds the numerator by
$CN^{-1/2}$ after convolution with the other probability measure.
The denominator is at least $cN^{-1/2}$ by
\eqref{eq:uniformclockLCLT}.} 
Choose $K_\varepsilon\ge1$ to dominate the holding and branching
constants.  There are at most $2p$ paths or arms.  The bounded trunk
holdings have total duration at most $\sum_i r_i$, so their exponential
factors have product at most $e^{\varepsilon\sum_i r_i}$.
Integrating the trunk holding times under their original conditional
laws therefore bounds the conditional environmental and holding-time mean by
$C_pK_\varepsilon^{I_p+2}e^{\varepsilon\sum_i r_i}$.

\emph{On $\mathcal G^c$.}
Average one holding per nonendpoint site and use
\eqref{eq:quantile-holding-bound} for every other holding, again
bounding each branching factor by $\lambda_+/\bar\lambda$.
The total duration of all holdings is $pt$ for paths and
$\le2pt$ for diagrams.  It is therefore at most
$C_pd$, giving the bound
$K_\varepsilon^{I_p+2}e^{\varepsilon C_pd}$.

To summarize, for an event $A$ determined by the spatial paths,
counts, split times and split vertices, we have
\begin{equation}
 \Eenv\!\left[\E_{\mathbb Q^p}
 \left[\one_A\prod_{i=1}^p\prod_zF_{\Gamma_i,z}(\Lambda_z)\right]\right]
 \ll_p\E_{\mathbb Q^p}\left[
 \one_AK_\varepsilon^{I_p+2}
 \left(e^{\varepsilon\sum_i r_i}
       +e^{\varepsilon C_pd}\one_{\mathcal G^c}\right)\right].
 \label{eq:integrated-product-bound}
\end{equation}
}

\medskip\noindent\textbf{Step 3: Bound weighted probabilities and the conditional means.}
For the trunk durations, \eqref{eq:qtrunk-time-tail} and $r_i\le Cd$
give $\E_{\mathbb Q^p}[e^{\eta\sum_i r_i}]\ll_p1$ for some fixed $\eta>0$.
Indeed, writing the tail bound as
$\mathbb Q(r_i>u)\ll e^{-c_1u}+e^{-c_2d}$ gives
\begin{align}
    \E_{\mathbb Q}[e^{\eta r_i}]
 =1+\eta\int_0^{Cd}e^{\eta u}\mathbb Q(r_i>u)\dd u
 \ll1+\frac{\eta}{c_1-\eta}+e^{(C\eta-c_2)d}\ll1\label{eq:155}
\end{align}
when $\eta<c_1$ and $C\eta<c_2$.  Independence of the copies gives
the bound for their sum; in the path case every $r_i=0$.
Fix $r>1$, then choose $\varepsilon$ so that
$2r\varepsilon<\eta$ and $\varepsilon C_p<c/(2r)$.
H\"older's inequality with exponents $r/(r-1),2r,2r$ applied to \eqref{eq:integrated-product-bound}, followed by applications of \eqref{eq:combined-factor-moment} and \eqref{eq:155} gives
\begin{align}\label{eq:localized-product-moment}
   \begin{split}
   &     \Eenv\!\left[\E_{\mathbb Q^p}
 \left[\one_A\prod_{i=1}^p\prod_zF_{\Gamma_i,z}(\Lambda_z)\right]\right]\\
 &\ll_p\E_{\mathbb Q^p}\left[
 \one_AK_\varepsilon^{I_p+2}
 \left(e^{\varepsilon\sum_i r_i}
       +e^{\varepsilon C_pd}\one_{\mathcal G^c}\right)\right]\\
       &\ll \mathbb Q^p(A)^{1-\frac1r}
 \E_{\mathbb Q^p}\!\left[{K_\varepsilon^{2r(I_p+2)}}\right]^{\frac1{2r}}
 \left(
 \E_{\mathbb Q^p}\!\left[e^{2r\varepsilon\sum_i r_i}\right]^{\frac1{2r}}
 +e^{\varepsilon C_pd}\mathbb Q^p(\mathcal G^c)^{\frac1{2r}}
 \right)\\
 &\ll_{p,r}\mathbb Q^p(A)^{1-1/r}
 \left(1+e^{(\varepsilon C_p-c/(2r))d}\right)
 \ll_{p,r}\mathbb Q^p(A)^{1-1/r}.
   \end{split}
\end{align}
Taking $A$ to be the whole space and using independence of the copies
under $\mathbb Q^p$ gives
$\Eenv[X^p]=\Eenv[\E_{\mathbb Q^p}[\prod_i\prod_z
F_{\Gamma_i,z}(\Lambda_z)]]\ll_p1$.
For $p=1$, this proves $B^{\rm end}\ll1$ in both cases.
This completes part~(ii); the remaining steps below concern only paths without diagrams.

\medskip\noindent\textbf{Step 4: Concentration via martingale resampling.}
Changing $\Lambda_z$ affects only paths visiting $z$.  We first bound
the visit probabilities, then use them to estimate the effect of
resampling each rate.
For $z\notin\{x,\mathbf0\}$, put
$a_z=\mathbb Q(z\in\operatorname{Supp}(\Gamma))$.
The count tails give
$\sum_{z\notin\{x,\mathbf0\}}a_z\le\E_{\mathbb Q}[N_1+1]\ll d$.
For each retained count $n$, the union bound and
\eqref{eq:sliceL2} give
$$
 \left(\sum_{z\notin\{x,\mathbf0\}}
 \Pp_{d,n;x,\mathbf0}(z\in\operatorname{Supp}(\pi))^2\right)^{\frac12}
 \le\sum_{k=1}^{n-1}
 \left(\sum_{z\notin\{x,\mathbf0\}}\Pp_{d,n;x,\mathbf0}(S_k=z)^2\right)^{\frac12}
 \ll d^{-\frac12}.
$$
Taking $p=1$ in the definition of $\mathcal G$,
\eqref{eq:UO2} and \eqref{eq:UO3} give
$\E_{\mathbb Q}[(N_1+1)\one_{\mathcal G^c}]\ll e^{-cd}$, by splitting
the count at $Ld$.  Minkowski's inequality, averaging first over the
retained counts and then adding this discarded contribution, gives
$$
 \sum_{z\notin\{x,\mathbf0\}}a_z^2
 =\E_{\mathbb Q^2}\!\left[
   \left|\bigl(\operatorname{Supp}(\Gamma_1)\cap\operatorname{Supp}(\Gamma_2)\bigr)
   \setminus\{x,\mathbf0\}\right|\right]
 \ll d^{-1}.
$$
Thus, two independent paths have few common vertices away from the
endpoints.

Set $p=2R$.  For each nonendpoint $z$, replace $\Lambda_z$ by an
independent rate with law $P$, independent of the entire original
environment, and denote the resulting value by $X^{(z)}$.\footnote{
All environmental expectations and $L^p$ norms below average the
nonendpoint rates and any resampled rates, conditional on
$\Lambda_x,\Lambda_{\mathbf0}$.}
Put $H_z=\E_{\mathbb Q}
[\one_{\{z\in\operatorname{Supp}(\Gamma)\}}
 \prod_wF_{\Gamma,w}(\Lambda_w)]$.
Let $H_z^{(z)}$ denote the same quantity after resampling $\Lambda_z$.
Paths avoiding $z$ have unchanged weights, so
$X-X^{(z)}=H_z-H_z^{(z)}$ and
$\|X-X^{(z)}\|_p\le2\|H_z\|_p$.
Taking $p$ independent paths, using Tonelli's theorem and then
\eqref{eq:localized-product-moment}, gives
$$
 \Eenv[H_z^p]
 =\Eenv\!\left[\E_{\mathbb Q^p}\!\left[
 \prod_{i=1}^p\one_{\{z\in\operatorname{Supp}(\Gamma_i)\}}
 \prod_{i=1}^p\prod_wF_{\Gamma_i,w}(\Lambda_w)\right]\right]
 \ll_{p,r}a_z^{p(1-\frac1r)}.
$$
Here the event that all $p$ copies visit $z$ has probability $a_z^p$
under $\mathbb Q^p$.

Revealing the nonendpoint rates one at a time gives a Doob martingale
from $B^{\rm end}$ to $X$.
{
Enumerate these vertices as $z_1,z_2,\ldots$.  Independence of the
resampled rates gives its increments as
$$
 D_i=\Eenv\!\left[X-X^{(z_i)}\mid
                 \Lambda_{z_1},\ldots,\Lambda_{z_i}\right],
 \qquad \sum_iD_i=X-B^{\rm end}.
$$
The classical case of the martingale convexity inequality
\cite[Corollary~3]{RicardXu2016} for $p>2$ and orthogonality for $p=2$, 
followed by conditional Jensen, gives
\begin{equation}
 \|X-B^{\rm end}\|_p^2
 \le(p-1)\sum_i\|D_i\|_p^2
 \le(p-1)\sum_{z\notin\{x,\mathbf0\}}\|X-X^{(z)}\|_p^2.
 \label{eq:path-resampling-moment}
\end{equation}
}
Taking $r=4s$ therefore gives
$$
 \|X-B^{\rm end}\|_p^2
 \ll_{p,s}\sum_{z\notin\{x,\mathbf0\}}a_z^{2-2/r}
 \le\left(\sum_{z\notin\{x,\mathbf0\}}a_z\right)^{2/r}
       \left(\sum_{z\notin\{x,\mathbf0\}}a_z^2\right)^{1-2/r}
 \ll_{p,s}d^{-1+4/r}=d^{-(1-1/s)}.
$$
Raising to the power $p/2=R$ proves \eqref{eq:abstract-high-moment}.

\medskip\noindent\textbf{Step 5: Bound the conditional path mean from below.}
Take $p=1$ and suppose $\Lambda_x,\Lambda_{\mathbf0}\ge-L$.
At every nonendpoint site, positive association gives
$$
 \Eenv\!\left[\prod_{a:v_a=z}Y_z(r_a)\right]
 \ge\prod_{a:v_a=z}\Eenv[Y_z(r_a)]=1.
$$
At the endpoints, the factors are at least
$e^{-(L+\lambda_+)u}$ for a holding of length $u$.
Conditional on a path of $n$ mutations, exchangeability gives mean
$t/(n+1)$ to each holding.  On retained counts this is bounded, and
there are at most $2+I_1$ endpoint occurrences.  Their total duration
is at most $t$ on the discarded counts, so
\eqref{eq:combined-factor-moment} and \eqref{eq:UO2} together give
$$
 \E_{\mathbb Q}\!\left[\sum_{a:v_a\in\{x,\mathbf0\}}r_a\right]
 \le C\E_{\mathbb Q}[2+I_1]+t\mathbb Q(\mathcal G^c)
 \ll1+de^{-cd}\ll1.
$$
Jensen's inequality therefore gives
$$
 B^{\rm end}\ge\E_{\mathbb Q}\!\left[
 e^{-(L+\lambda_+)\sum_{a:v_a\in\{x,\mathbf0\}}r_a}\right]
 \ge e^{-C_L}\gg_L1,
$$ as desired.
\endgroup
\end{proof}

\subsubsection{Independent-visit time-shift estimates}

{The occupation and bootstrapping arguments require
independent-visit estimates at nearby times and starting vertices.
The next lemma gives these comparisons.}

\begin{lemma}
\label{lem:global-fresh-harnack}
There is $\kappa_0>0$ with the following property.  For every fixed
$C_0,C_1,M<\infty$ and every fixed $a\ge0$,
for all sufficiently large $d$,\footnote{The threshold for $d$ may depend on $M,a$.} uniformly in
$$
1\le m\le d,\qquad d_{\mathrm H}(y,\mathbf{x}_m)\le M,\qquad
 \left|t-du_*\!\left(\frac md\right)\right|\le C_0\log d,
$$
we have
\begin{equation}
 \widehat M_{t+a}(y)
 \gg_M e^{\kappa_0a}\widehat M_t(\mathbf{x}_m).              \label{eq:global-fresh-harnack}
\end{equation}
Uniformly also for $0\le h\le C_1\log d$,
\begin{equation}
 \widehat M_{t-h}(\mathbf{x}_m)
 \ll e^{-\kappa_0h}\widehat M_t(\mathbf{x}_m),               \label{eq:global-fresh-backward-ratio}
\end{equation}
whenever the time $t-h$ is nonnegative.
The implicit constants may depend on $C_0,C_1$.
\end{lemma}

At the deterministic center,
$\widehat M_{\widehat\theta_{d,m}}(\mathbf x_m)=1$ by
\eqref{eq:independent-center-equality}.  Thus,
{\eqref{eq:global-fresh-backward-ratio} gives exponential
decay of the independent-visit occupation mean in the preceding
logarithmic window.}  In the forward direction,
\eqref{eq:global-fresh-harnack} lets a sufficiently large fixed time shift
compensate for any fixed starting radius $M$; this is the mean estimate
used in the bootstrapping argument.

\begin{proof}
{Choose the retained level $\varepsilon_{\rm ret}>0$ in
\eqref{eqtag:GR0} sufficiently small, within the range of
Proposition~\ref{prop:global-oneobject-local}.  Every maximizing saddle
on the zero set has $\kappa\ge\kappa_{\rm rt}>0$ by
Lemma~\ref{lem:global-root-model-only}.  The finite-$d$ conclusion of
Lemma~\ref{lem:root-active-stability} therefore lets us choose this
level so that the exact tilts of $n\in\mathcal I_d(m,s)$ are at least
$\kappa_{\rm rt}/2$ throughout the enlarged time interval below.}
We work uniformly in $m$ and
base times $s$ with
$|s-du_*(m/d)|\le(C_0+C_1+3)\log d$, which includes all the times and
unit windows used below, for sufficiently large $d$.
By \eqref{eq:UO4}, their exact clock tilts satisfy
$$
 -\frac{H'(\kappa_n(s))}{H(\kappa_n(s))}=\frac{s}{n+1}.
$$
Thus, the sum of the $n+1$ tilted holding times has mean $s$.
{The choice of retained level gives}
\begin{equation}
 \kappa_n(s)\ge\kappa_0:=\frac{\kappa_{\rm rt}}2>0,
 \qquad n\in\mathcal I_d(m,s).
 \label{eq:time-shift-clock-positive}
\end{equation}
Moreover, \eqref{eq:UO2} gives
\begin{equation}
 \sum_{n\in\mathcal I_d(m,s)}p_{d,m}(n)h^{*(n+1)}(s)
 \ge(1-e^{-cd})\widehat K_{d,m}(s).
 \label{eq:retained-kernel-mass}
\end{equation}
Let $\ell=d_{\mathrm H}(y,\mathbf0)$, so $|\ell-m|\le M$, and put
$a_n(s)=A_b\bigl(z_n(s)\bigr)/C_b\bigl(z_n(s)\bigr)$.  To compare a retained summand with
one at distance $\ell$, use count $n+\varepsilon$, where $\varepsilon=1$
if $b=2$ and $\ell-m$ is odd, and $\varepsilon=0$ otherwise.
This assignment is injective and respects the binary lattice.
By \eqref{eq:tilted-pdm},
\[
 \frac{p_{d,\ell}(n+\varepsilon)}{p_{d,m}(n)}
 =a_n(s)^{\ell-m}
   \frac{(n+\varepsilon)!}{n!\,(dz_n(s))^\varepsilon}
   \frac{\Pp_{z_n(s)}(\Sigma_{d,\ell,z_n(s)}=n+\varepsilon)}
        {\Pp_{z_n(s)}(\Sigma_{d,m,z_n(s)}=n)}
 \gg e^{-CM}.
\]
Here $a_n(s)^{\ell-m}\ge e^{-CM}$, and the factorial factor
is bounded above and away from zero on retained counts.
The numerator count is within $O(M+1)$ of its tilted mean, so
Lemma~\ref{lem:spatialLCLT} bounds the local-probability ratio below.

For $0\le h\le(C_1+1)\log d$, the clock ratio with the changed count is
$$
 \frac{h^{*(n+1+\varepsilon)}(s+h)}{h^{*(n+1)}(s)}
 =e^{\kappa_n(s)h}H\bigl(\kappa_n(s)\bigr)^\varepsilon
  \frac{\widetilde f_{n+1+\varepsilon,\kappa_n(s)}(s+h)}
       {\widetilde f_{n+1,\kappa_n(s)}(s)}.
$$
The numerator's displacement from its mean is
$h-\varepsilon\mu_{\kappa_n(s)}=O(\log d)$.
Lemma~\ref{lem:clockLCLT} bounds that ratio below uniformly, and
$H\bigl(\kappa_n(s)\bigr)^\varepsilon$ is bounded away from zero.
Using \eqref{eq:time-shift-clock-positive}, summing over these distinct
numerator terms, and applying \eqref{eq:retained-kernel-mass} gives
\[
 \widehat K_{d,\ell}(s+h)
 \ge ce^{-CM}e^{\kappa_0h}\widehat K_{d,m}(s).
\]
For a fixed $a\ge0$, choose $d$ large enough that $a\le\log d$ and
integrate this bound over $s\in[t,t+1]$ with $h=a$.
This proves \eqref{eq:global-fresh-harnack}, with a prefactor independent
of $a$.

For the backward estimate, apply the same lower comparison at base time
$s-h$, with $\ell=m$, and take its reciprocal:
\[
 \widehat K_{d,m}(s-h)\ll
 e^{-\kappa_0h}\widehat K_{d,m}(s).
\]
Both base times satisfy the enlarged time condition above.
Integration over $s\in[t,t+1]$
proves \eqref{eq:global-fresh-backward-ratio}.
\end{proof}

{\subsubsection{
An occupation bound before the deterministic centering}}

% Before bounding occupation near the center, we estimate the annealed
% kernel where $\widehat F<0$.

{
\begin{lemma}
\label{lem:global-onebridge}
For every fixed compact
$\mathcal K\subset\{(\theta,u):u>0,\ \widehat F(\theta,u)<0\}$,
there is $c_{\mathcal K}>0$ such that, uniformly for
$(m/d,t/d)\in\mathcal K$,
\begin{equation}
 \Eenv\bigl[K_t^\omega(\mathbf x_m,\mathbf0)\bigr]
 \le e^{-c_{\mathcal K}d}
 \label{eq:precritical-annealed-upper}
\end{equation}
for all sufficiently large $d$.
\end{lemma}

\begin{proof}
On $\mathcal K$, write $\widehat F\le-c_0$ and $t\le C_0d$, and fix
$0<\varepsilon<c_0/(2C_0)$.
\begin{itemize}
\item For $n\le Rd$, fix a path and its holding times.  At a vertex $z$
visited for durations $r_1,\ldots,r_k$, apply
\eqref{eq:quantile-holding-bound} to all but the first ratio and average
the first ratio exactly:
$$
\Eenv\!\left[\prod_{\ell=1}^kY_z(r_\ell)\right]
 \le p_\varepsilon^{-(k-1)}e^{\varepsilon\sum_{\ell=2}^kr_\ell}
       \Eenv[Y_z(r_1)]
 =p_\varepsilon^{-(k-1)}e^{\varepsilon\sum_{\ell=2}^kr_\ell}.
$$
Independence between vertices and the total holding time $t$ give
$$
 \Eenv[W_\Gamma^\omega]
 \le p_\varepsilon^{-\widetilde I_1(\pi)}e^{\varepsilon t},
 \qquad\widetilde I_1(\pi)=\sum_z(L_z(\pi)-1)_+.
$$
Here, $\widetilde I_1$ is the one-bridge repetition count from
Proposition~\ref{prop:collisionmgf}; it includes endpoint revisits.
Conditional on $N=n$, the spatial path under
$\mathbb Q_{\mathbf x_m,\mathbf0,t}$ is an ordinary Hamming bridge,
independent of its holding times.  For every admissible $n\le Rd$,
Proposition~\ref{prop:collisionmgf} gives
$$
 \E_{\Pp_{d,n;\mathbf x_m,\mathbf0}}
     [p_\varepsilon^{-\widetilde I_1(\pi)}]
 =\E_{\Pp_{d,n;\mathbf x_m,\mathbf0}}
     [e^{\log(1/p_\varepsilon)\widetilde I_1(\pi)}]
 \ll_{\varepsilon,R}1.
$$
If $p_\varepsilon=1$, this expectation equals one directly.  Consequently,
$$
 \E_{\mathbb Q_{\mathbf x_m,\mathbf0,t}}
     [\Eenv[W_\Gamma^\omega]\mid N=n]
 \ll_{\varepsilon,R}e^{\varepsilon t}.
$$
By \eqref{eq:Qbridge}, each admissible count has probability
$$
 \mathbb Q_{\mathbf x_m,\mathbf0,t}(N=n)
 =\frac{p_{d,m}(n)h^{*(n+1)}(t)}{\widehat K_{d,m}(t)}.
$$
Multiplying the bridge average of $W_\Gamma^\omega$ by
$\widehat K_{d,m}(t)$ recovers the static kernel, by
\eqref{eq:quenched-path-expansion}.  Thus, the environmental mean of
its contribution from $n\le Rd$ is
\begin{align*}
 &\widehat K_{d,m}(t)\E_{\mathbb Q_{\mathbf x_m,\mathbf0,t}}
       [\one_{\{N\le Rd\}}\Eenv[W_\Gamma^\omega]]\\
 &\quad=\sum_{\substack{0\le n\le Rd\\p_{d,m}(n)>0}}
 p_{d,m}(n)h^{*(n+1)}(t)
 \E_{\mathbb Q_{\mathbf x_m,\mathbf0,t}}
       [\Eenv[W_\Gamma^\omega]\mid N=n]\\
 &\quad\ll_{\varepsilon,R}e^{\varepsilon t}
 \sum_{n\le Rd}p_{d,m}(n)h^{*(n+1)}(t)
 \le e^{\varepsilon t}\widehat K_{d,m}(t).
\end{align*}
Infeasible counts contribute zero; all admissible counts below the cutoff,
including zero when feasible, satisfy this estimate.
\item For the remaining counts,
the pointwise bound $\Lambda_z\le\lambda_+$ bounds their contribution by
$$
 e^{-(1-\lambda_+)t}\sum_{n>Rd}\frac{t^n}{n!}.
$$
Since $\sum_{n>Rd}t^n/n!\le\exp(-Rd+et)$, increasing the fixed
$R$ if necessary makes this tail at most $e^{-c_0d}$ for $t\le C_0d$.
\end{itemize}
By the global upper bound \eqref{eq:global-oneobject-upper},
the contribution from $n\le Rd$ is
$\ll_{\varepsilon,R}d^{3/2}e^{-c_0d/2}$.
Absorbing the polynomial factor proves
\eqref{eq:precritical-annealed-upper}.
\end{proof}}

The next proposition bounds the expected total time spent at the target
by all particles before the deterministic center $\widehat\theta_{d,m}$. 
Write $M_t^\omega(x)=\int_t^{t+1}K_s^\omega(x,\mathbf0)\dd s$
for the expected target occupation over a unit time window.

\begin{proposition}
\label{prop:global-center}
Under Assumption~\ref{ass:environment}, there is $c>0$ such that, for
every fixed $L\ge0$, on events of environment probability tending to
one, whenever $\Lambda_{\mathbf0}\ge-L$, simultaneously for every $A\ge0$,
\begin{equation}
 \sup_{1\le m\le d}
 \int_0^{(\widehat\theta_{d,m}-A)_+}
 K_s^\omega(\mathbf{x}_m,\mathbf 0)\,\dd s
 \ll_L e^{-cA}+d^{-1}.                                \label{eq:global-backward}
\end{equation}
For each fixed $M$, the same bound holds with $\mathbf x_m$ in the
integrand replaced by any $y\ne\mathbf0$ satisfying
$d_{\rm H}(y,\mathbf x_m)\le M$, with an implicit constant depending also
on $M$. 
\end{proposition}

\begin{proof}
{
We first show that, near the center, the independent-visit kernel
decreases exponentially as we move backwards in time.  We then use
environmental concentration to bound the corresponding weighted
occupation integral in the static model.  Finally, we control earlier
times separately and remove the exponential weight to obtain the bound
for every backward shift $A\ge0$.

\noindent\emph{Step 1: exponential decay before the center.}
Choose a small fixed $\eta>0$ so that $2\eta<u_{\min}$,
$2\eta\le\eta_{\rm rt}$, and the closed
strip $|u-u_*(\theta)|\le2\eta$ lies in $\mathcal R$.
At $u=u_*(\theta)$, every maximizing saddle satisfies
$\kappa\ge\kappa_{\rm rt}>0$ by
Lemma~\ref{lem:global-root-model-only}.  By
Lemma~\ref{lem:root-active-stability}, we can decrease $\eta$ and choose
the tolerance $\varepsilon_{\rm ret}>0$ in \eqref{eqtag:GR0} sufficiently
small that
$$
 \left|\frac{s}{d}-u_*\!\left(\frac{\ell}{d}\right)\right|\le2\eta,
 \qquad n\in\mathcal I_d(\ell,s)
 \quad\Longrightarrow\quad
 \kappa_n(s)\ge\frac{\kappa_{\rm rt}}2,
$$
uniformly for all sufficiently large $d$.  Here the retained counts in
\eqref{eqtag:GR0} satisfy
$$
 G\!\left(\frac{\ell}{d},\frac{s}{d},\frac nd\right)
 \ge\widehat F\!\left(\frac{\ell}{d},\frac{s}{d}\right)
       -\varepsilon_{\rm ret},
$$
and $\kappa_n(s)$ is defined by \eqref{eqtag:GR0a}, namely
$-H'(\kappa_n(s))/H(\kappa_n(s))=s/(n+1)$.
Equations~\eqref{eq:UO2} and \eqref{eq:UO19}, with displacement one,
and the uniform clock local limit estimate give
$$
 \frac{\widehat K_{d,\ell}(s+1)}{\widehat K_{d,\ell}(s)}
 \ge(1-o(1))e^{\kappa_{\rm rt}/2}
 \ge e^\gamma,\qquad \gamma:=\frac{\kappa_{\rm rt}}4>0,
$$
uniformly for $1\le\ell\le d$ and
$|s/d-u_*(\ell/d)|\le\eta$, once $d$ is sufficiently large.
Fix this lower threshold on $d$ before iterating the last inequality.
Lemma~\ref{lem:independent-center-location} puts
$\widehat\theta_{d,\ell}+D$ inside this strip for every fixed $D\ge0$.
Iteration up to that time, followed by the bounded-shift comparison in
Proposition~\ref{prop:global-oneobject-shifts} for the remaining
displacement in $[0,1)$, yields
\begin{equation}
\begin{split}
 \widehat K_{d,\ell}(s)
 \ll e^{-\gamma(\widehat\theta_{d,\ell}+D-s)}
       \widehat K_{d,\ell}(\widehat\theta_{d,\ell}+D),\qquad 
 du_*(\ell/d)-\eta d\le s\le\widehat\theta_{d,\ell}+D.
\end{split}
 \label{eq:occupation-fresh-decay}
\end{equation}
The bounded-shift comparison and
\eqref{eq:independent-center-equality} give
$\widehat K_{d,\ell}(\widehat\theta_{d,\ell}+D)\asymp_D1$.
Fix $0<c<\gamma$.  Integrating \eqref{eq:occupation-fresh-decay}
therefore gives, uniformly in $\ell$,
\begin{equation}
 \int_{du_*(\ell/d)-\eta d}^{\widehat\theta_{d,\ell}+D}
 e^{c(\widehat\theta_{d,\ell}+D-s)}
 \widehat K_{d,\ell}(s)\dd s\ll_D1.
 \label{eq:occupation-fresh-weighted}
\end{equation}
The choices of $\eta,\gamma,c$ are independent of the starting radius.

\noindent\emph{Step 2: concentrate the weighted occupation.}
Fix $M<\infty$.  We first compare the centers at distances
$|\ell-m|\le M$, with $\ell,m\ge1$.
The Lipschitz bound on $u_*$ and
Lemma~\ref{lem:independent-center-location} put each center in the
logarithmic root band for both distances.  Since
$d_{\rm H}(\mathbf x_\ell,\mathbf x_m)=|\ell-m|$, Lemma~\ref{lem:global-fresh-harnack}
and \eqref{eq:independent-center-equality} give
$$
 \widehat M_{\widehat\theta_{d,m}+D}(\mathbf x_\ell)
 \gg_M e^{\kappa_0D}
$$
for every fixed $D\ge0$, with a prefactor independent of $D$.
Choose $D_M$ so that this lower bound exceeds one.  The first-crossing
definition \eqref{eq:thetahat definition} gives
$\widehat\theta_{d,\ell}\le\widehat\theta_{d,m}+D_M$.
Exchanging $\ell,m$ proves
\begin{equation}
 |\widehat\theta_{d,\ell}-\widehat\theta_{d,m}|\le D_M
 \qquad (|\ell-m|\le M,\ \ell,m\ge1).
 \label{eq:nearby-center-distance}
\end{equation}

For each required start $y\ne\mathbf0$, put
$\ell=d_{\rm H}(y,\mathbf0)$ and define its weighted occupation by
$$
 Z_y^\omega=
 \int_{du_*(\ell/d)-\eta d}^{\widehat\theta_{d,\ell}+D_M}
 e^{c(\widehat\theta_{d,\ell}+D_M-s)}
 K_s^\omega(y,\mathbf0)\dd s.
$$
Condition on $\Lambda_y,\Lambda_{\mathbf0}$, and use conditional
environmental $L^{2R}$ norms below.  The conditional mean bound in
Lemma~\ref{lem:abstract-environment-concentration}(i) and
\eqref{eq:occupation-fresh-weighted} give
$\Eenv[Z_y^\omega\mid\Lambda_y,\Lambda_{\mathbf0}]\ll_M1$.
For every fixed integer $R\ge1$ and $a>1$, Minkowski's integral
inequality and \eqref{eq:abstract-high-moment} give
\begin{align*}
 &\left\|Z_y^\omega-
       \Eenv[Z_y^\omega\mid\Lambda_y,\Lambda_{\mathbf0}]\right\|_{2R}\ll_{R,a}d^{-(1-1/a)/2}
 \int_{du_*(\ell/d)-\eta d}^{\widehat\theta_{d,\ell}+D_M}
 e^{c(\widehat\theta_{d,\ell}+D_M-s)}
 \widehat K_{d,\ell}(s)\dd s
 \ll_{M,R,a}d^{-(1-1/a)/2}.
\end{align*}
All these estimates are uniform in the conditioned endpoint rates.
There are at most $C_Md^{M+1}$ required starts.  After averaging the
conditional moment bound, Markov's inequality and a union bound,
with $R(1-1/a)>M+2$, give a single environment event of probability
tending to one on which $Z_y^\omega\ll_M1$ for every such $y$.

\noindent\emph{Step 3: earlier times and all backward shifts.}
Choose $\varepsilon>0$ small enough that
$2\varepsilon<u_{\min}-\eta$.
Homogeneous domination and the exact Hamming heat kernel give
$$
 K_s^\omega(y,\mathbf0)
 \le e^{-(1-\lambda_+)s}A_b(s/d)^\ell C_b(s/d)^{d-\ell}.
$$
For small $u$, $A_b(u)\le Cu$, $A_b(u)\le C_b(u)$, and
$\log C_b(u)\le Cu^2$.  Decreasing $\varepsilon$ if necessary and
using $\ell\ge1$ proves the pathwise bound
\begin{equation}
 \sup_{y\ne\mathbf0}\int_0^{\varepsilon d}
 K_s^\omega(y,\mathbf0)\dd s
 \ll\frac1d\int_0^\infty se^{-(1-\lambda_+)s/2}\dd s
 \ll d^{-1}.
 \label{eq:global-initial}
\end{equation}
On the compact set
$\{(\theta,u):\theta\in[0,1],
\varepsilon\le u\le u_*(\theta)-\eta\}$,
$\widehat F$ has a strictly negative maximum by \eqref{eqtag:R1}.
Lemma~\ref{lem:global-onebridge} and invariance of the environment law
under Hamming-cube automorphisms therefore give
$$
 \Eenv\!\left[
 \sum_{m=1}^d
 \sum_{\substack{y\ne\mathbf0\\d_{\rm H}(y,\mathbf x_m)\le M}}
 \int_{\varepsilon d}^{du_*(d_{\rm H}(y,\mathbf0)/d)-\eta d}
 K_s^\omega(y,\mathbf0)\dd s\right]
 \ll_M d^{M+2}e^{-c_1d}.
$$
Markov's inequality makes this sum at most $e^{-c_1d/2}$ on another
environment event of probability tending to one.

Finally, we intersect the two events above.  Fix any eligible $m,y$ and $A\ge0$.
If $(\widehat\theta_{d,m}-A)_+>du_*(\ell/d)-\eta d$, then every $s$
between these two times satisfies
$\widehat\theta_{d,\ell}+D_M-s\ge A$ by
\eqref{eq:nearby-center-distance}.  Hence, this part of the occupation
integral is at most $e^{-cA}Z_y^\omega\ll_M e^{-cA}$.
If the interval is empty it contributes zero.  The earlier part is
bounded by \eqref{eq:global-initial} and the preceding exponential
bound.  This proves \eqref{eq:global-backward} and its nearby-start
version, simultaneously for every $A\ge0$, on the same event.}
\end{proof}

Proposition \ref{prop:global-center} does not directly 
give a lower bound to the FPT, because a small population at
one time would not rule out an earlier visit followed by departure.
The idea is to show that an early hit produces a positive expected amount of occupation
during the following unit time interval; see Section \ref{sec:all-distance-proof}. 

\begingroup
\subsubsection{A positive hitting probability}

We first bound the annealed $(1+q)$-moment of $N_t(\mathbf0)$.
We then obtain a lower bound on $\Pw_y(N_t(\mathbf0)>0)$, which also
bounds the hitting probability from below since
$\{N_t(\mathbf0)>0\}\subset\{\tau_y^\omega\le t\}$.

The next two propositions use the parameter range
{
\begin{equation}
 1\le m\le d,\qquad d_{\mathrm H}(y,\mathbf{x}_m)\le M,\qquad
 \left|t-du_*\!\left(\frac md\right)\right|\le C_0\log d,
 \label{eq:global-onebridge-range}
\end{equation}}
where $C_0,M<\infty$ are fixed.

Define the branching integral
\begin{equation}
 \mathcal C_{q,t}^{\omega}(x)
 =2\int_0^t 
 \sum_vK_r^\omega(x,v)(\Lambda_v)_+
       K_{t-r}^\omega(v,\mathbf0)^{1+q}\dd r,
 \label{eq:quenched-branching-integral}
\end{equation}
and write $\mathcal C_{q,t}^{\omega,I}$ for its restriction to
$r\le t-\delta d$.

\begin{proposition}
\label{prop:quenched-fractional-diagram}
Fix $q\in(0,q_0]$ as in Section~\ref{sec:initial-branch}.
Uniformly in the starting vertices and times in
\eqref{eq:global-onebridge-range}, and in the conditioned endpoint rates,
\begin{equation}
 \Eenv\!\left[\Ew_y[N_t(\mathbf0)^{1+q}]
          \bigm|\Lambda_y,\Lambda_{\mathbf0}\right]
 \ll \widehat K_t(y,\mathbf0)
          +\widehat K_t(y,\mathbf0)^{1+q}.
 \label{eq:annealed-population-moment}
\end{equation}
The constant is independent of every subsequently fixed starting radius
$M$; only the lower threshold on $d$ may depend on $M$.
\end{proposition}

\begin{proof}
For independent copies $Z,Z'$ of a nonnegative random variable,
subadditivity of the $q$-th power and Jensen's inequality give
$$
 \E[(Z+Z')^{1+q}-Z^{1+q}-(Z')^{1+q}]
 \le2\E[Z]\E[Z^q]\le2\E[Z]^{1+q}.
$$
First-event conditioning in the fixed environment, followed by
variation of constants, therefore yields
\begin{equation}
 \Ew_y[N_t(\mathbf0)^{1+q}]
 \le K_t^\omega(y,\mathbf0)
 +2\int_0^t\sum_vK_r^\omega(y,v)(\Lambda_v)_+
                 K_{t-r}^\omega(v,\mathbf0)^{1+q}\dd r
 =K_t^\omega(y,\mathbf0)+\mathcal C_{q,t}^{\omega}(y).
 \label{eq:population-branching-inequality}
\end{equation}
Note that all moments are finite: suppressing deaths and increasing splitting
rates to $\lambda_+$ dominates the population by a Yule process.

For $h=t-r\le\delta d$, the exact mutation kernel and
$\log C_b(u)\ll u^2$ give, after decreasing $\delta$,
$$
 \sup_vK_h^\omega(v,\mathbf0)
 \le e^{\lambda_+h}\Pp_{\mathbf0}(X_h=\mathbf0)
 \le e^{-(1-\lambda_+)h/2}.
$$
The semigroup identity and $(\Lambda_v)_+\le\lambda_+$ then imply
$$
 \mathcal C_{q,t}^{\omega}(y)
      -\mathcal C_{q,t}^{\omega,I}(y)
 \le2\lambda_+K_t^\omega(y,\mathbf0)
       \int_0^{\delta d}e^{-q(1-\lambda_+)h/2}\dd h
 \ll K_t^\omega(y,\mathbf0),
$$
since $q(1-\lambda_+)>0$.

For the other split times, the scalar split law gives
\begin{equation}
 \frac{\mathcal C_{q,t}^{\omega,I}(y)}
      {\widehat{\mathcal C}_{q,t}^{I}(y)}
 =\E_{\mathbb Q_{q;y,t}}\!\left[
 \frac{(\Lambda_v)_+}{\bar\lambda}
 \frac{K_r^\omega(y,v)}{\widehat K_r(y,v)}
 \left(\frac{K_{t-r}^\omega(v,\mathbf0)}
             {\widehat K_{t-r}(v,\mathbf0)}\right)^{1+q}\right]
 \le X_{q,1}^\omega+X_{q,2}^\omega.
 \label{eq:exact-qratio}
\end{equation}
The inequality uses $a^{1+q}\le a+a^2$, and the two terms are exactly
the single-diagram averages in \eqref{eq:diagram-environment-average}.
Their conditional environmental means are bounded by
Lemma~\ref{lem:abstract-environment-concentration}(ii).
Together with \eqref{eq:qnormalizer-bounds}, this gives
$$
 \Eenv[\mathcal C_{q,t}^{\omega,I}(y)
                  \mid\Lambda_y,\Lambda_{\mathbf0}]
 \ll \widehat K_t(y,\mathbf0)^{1+q}.
$$
Finally, the conditional mean bound in
Lemma~\ref{lem:abstract-environment-concentration}(i) gives
$\Eenv[K_t^\omega(y,\mathbf0)\mid\Lambda_y,\Lambda_{\mathbf0}]
\ll\widehat K_t(y,\mathbf0)$.
Averaging \eqref{eq:population-branching-inequality} proves
\eqref{eq:annealed-population-moment}, with the stated uniformity.
\end{proof}
\endgroup

\begingroup
\begin{proposition}
\label{prop:quenched-barrier}
For every fixed $L\ge0$, there is $p_L>0$ with the following property.
For each fixed $M<\infty$, one can choose a fixed $a_M\ge0$ such that,
on events of environment probability tending to one, simultaneously
for $1\le m\le d$ and $d_{\mathrm H}(y,\mathbf x_m)\le M$,
\begin{equation}
 \Pw_y\!\left(N_{\widehat\theta_{d,m}+a_M}(\mathbf0)>0\right)
 \ge p_L
 \qquad\text{whenever }\Lambda_y,\Lambda_{\mathbf0}\ge-L.
 \label{eq:quenched-hit}
\end{equation}
The success probability $p_L$ is independent of $M$.
\end{proposition}

\begin{proof}

The proof proceeds by first applying
Proposition~\ref{prop:quenched-fractional-diagram} and fractional
Paley--Zygmund to bound below the environmental mean of
$\Pw_y(N_t(\mathbf0)>0)$, conditional on the endpoint rates.  Then, we
show that this probability is concentrated at its conditional
environmental mean by a resampling argument.

\noindent\emph{Step 1: the conditional environmental mean.}
Fix $y,t$ in \eqref{eq:global-onebridge-range}, write
$k=\widehat K_t(y,\mathbf0)$, and put
$p_t^\omega(y)=\Pw_y(N_t(\mathbf0)>0)$.
All environmental expectations and $L^p$ norms in this proof are
conditional on $\Lambda_y,\Lambda_{\mathbf0}$.
If $y=\mathbf0$, this means conditioning on the single rate
$\Lambda_{\mathbf0}$.
When these two rates are at least $-L$,
Lemma~\ref{lem:abstract-environment-concentration}(i) and
Proposition~\ref{prop:quenched-fractional-diagram} give
$$
 \Eenv[\Ew_y[N_t(\mathbf0)]]\gg_L k,
 \qquad
 \Eenv[\Ew_y[N_t(\mathbf0)^{1+q}]]\ll k+k^{1+q}.
$$
The fractional Paley--Zygmund inequality
\cite[Lemma~4.4]{AngelHutchcroftJarai2021} therefore gives
\begin{equation}
 \Eenv[p_t^\omega(y)\mid\Lambda_y,\Lambda_{\mathbf0}]
 \ge\left(
 \frac{\Eenv[\Ew_y[N_t(\mathbf0)]]^{1+q}}
      {\Eenv[\Ew_y[N_t(\mathbf0)^{1+q}]]}
       \right)^{1/q}
 \gg_L\frac{k}{1+k}.
 \label{eq:annealed-success-lower}
\end{equation}

\noindent\emph{Step 2: concentration of the success probability.}
Let $\omega^{(z)}$ resample the rate at a vertex
$z\notin\{y,\mathbf0\}$.
{Kill each particle when it first enters $z$, and write
$p_t^{\neg z}(y)$ for {\color{black}the probability that the resulting process
has a particle at the target at time $t$}.
This process uses only rates outside $z$, so its law is the same in
$\omega$ and $\omega^{(z)}$.  It can be obtained from either original
genealogy by deleting every descendant of a particle entering $z$.
If this deletion removes a success at time $t$, an original target
particle has an ancestral path through $z$.
Let $kH_z^\omega$ be the quenched expected number of such target particles.
Then
$$
 0\le p_t^\omega(y)-p_t^{\neg z}(y)\le kH_z^\omega.
$$
Apply the same bound in $\omega^{(z)}$.  The triangle inequality and
the many-to-one path expansion give}
\begin{equation}
 H_z^\omega
 =\E_{\mathbb Q_{y,\mathbf0,t}}
   [\one_{\{z\in\operatorname{Supp}(\Gamma)\}}W_\Gamma^\omega],
 \qquad
 |p_t^\omega(y)-p_t^{\omega^{(z)}}(y)|
 \le k(H_z^\omega+H_z^{\omega^{(z)}}).
 \label{eq:success-resampling}
\end{equation}
Put $a_z=\mathbb Q_{y,\mathbf0,t}(z\in\operatorname{Supp}(\Gamma))$
for $z\notin\{y,\mathbf0\}$.
As in the path case of
Lemma~\ref{lem:abstract-environment-concentration},
$\sum_{z\notin\{y,\mathbf0\}}a_z\ll d$ and
$\sum_{z\notin\{y,\mathbf0\}}a_z^2\ll d^{-1}$.
For every integer $R\ge1$ and $r>1$,
\eqref{eq:localized-product-moment}, applied to $2R$ independent
bridges all using $z$, gives
$\Eenv[(H_z^\omega)^{2R}]\ll_{R,r}a_z^{2R(1-1/r)}$.
Apply the martingale resampling inequality \eqref{eq:path-resampling-moment} to
\eqref{eq:success-resampling}.  Taking $r=4s$, for $s>1$, yields
\begin{align*}
 \|p_t^\omega(y)-\Eenv[p_t^\omega(y)
                    \mid\Lambda_y,\Lambda_{\mathbf0}]\|_{2R}^2
 &\ll_{R,s} k^2\sum_{z\notin\{y,\mathbf0\}}a_z^{2-2/r}\\
 &\ll_{R,s}k^2
       \left(\sum_{z\notin\{y,\mathbf0\}}a_z\right)^{2/r}
       \left(\sum_{z\notin\{y,\mathbf0\}}a_z^2\right)^{1-2/r}
 \ll_{R,s} k^2d^{-(1-1/s)}.
\end{align*}
Consequently,
\begin{equation}
 \Eenv\!\left[
  \left|p_t^\omega(y)-\Eenv[p_t^\omega(y)
            \mid\Lambda_y,\Lambda_{\mathbf0}]\right|^{2R}
   \bigm|\Lambda_y,\Lambda_{\mathbf0}\right]
 \ll_{R,s} k^{2R}d^{-R(1-1/s)}.
 \label{eq:success-high-moment}
\end{equation}
The estimate is uniform in the conditioned rates and holds on any
fixed compact subset of $\mathcal R$; averaging also gives its
unconditional version.

\noindent\emph{Step 3: concluding the proof.}
At the deterministic center,
$\widehat M_{\widehat\theta_{d,m}}(\mathbf x_m)=1$ by
\eqref{eq:independent-center-equality}.
The bounded-time comparison in
{Proposition~\ref{prop:global-oneobject-shifts}} and the Harnack bound
\eqref{eq:global-fresh-harnack} give, for every fixed $a\ge0$ and
all sufficiently large $d$,
$$
 \widehat K_{\widehat\theta_{d,m}+a}(y,\mathbf0)
 \asymp\widehat M_{\widehat\theta_{d,m}+a}(y)
 \gg_M e^{\kappa_0a},\qquad d_{\mathrm H}(y,\mathbf x_m)\le M.
$$
Choose a sufficiently large fixed $a_M$ for the lower bound in
\begin{equation}
 1\le\widehat K_{\widehat\theta_{d,m}+a_M}(y,\mathbf0)
 \le C_M<\infty,
 \qquad d_{\mathrm H}(y,\mathbf x_m)\le M,
 \label{eq:success-common-time-kernel}
\end{equation}
uniformly in $m$ for all sufficiently large $d$.
For the upper bound, we use the fixed-distance comparison in the proof of
Lemma~\ref{lem:global-fresh-harnack}, with the two distances exchanged,\footnote{The same coefficient and clock estimates apply when
$y=\mathbf0$, since the retained-count and local limit estimates
include $\theta=0$.}
and then the bounded-time comparison at distance $m$.
These times satisfy \eqref{eq:global-onebridge-range} with a fixed
coefficient of $\log d$ independent of $M$, by
\eqref{eq:all-distance-center-log}; only the threshold on $d$ changes.

By \eqref{eq:annealed-success-lower}, the conditional mean at each
such time is at least $2p_L>0$, with $p_L$ independent of $M$.
There are $O_M(d^{M+1})$ pairs $(m,y)$.
Choose $R$ so that $R(1-1/s)>M+2$ in
\eqref{eq:success-high-moment}.  Markov's inequality, the upper bound
in \eqref{eq:success-common-time-kernel}, and a union bound show that
all success probabilities differ from their conditional means by at
most $p_L$ on one event of environment probability tending to one.
This proves \eqref{eq:quenched-hit}.
\end{proof}
\endgroup

\section{Proof of main results}
\label{sec:main-results-proofs}

\subsection{Proof of Theorem~\ref{thm:sparse-expansion}}
\label{sec:sparse-expansion-proof}

\paragraph{Right derivative of $u_*$ at zero.}

To expand the deterministic center at small starting distance, we compute
the right derivative of $u_*$ at zero, using
$\widehat F(\theta,u_*(\theta))=0$.  Write $u_0=u_*(0)$.
Recall from the paragraph following
Definition~\ref{def:maximizing-saddle} that $\mathcal S_*$ is the set
of maximizing saddles at $(0,u_0)$ and
$\mathcal K_*=\{\kappa:(\nu,z,\kappa)\in\mathcal S_*\}$ is the set
of their clock tilts.

For the maximizing saddles considered here,
Lemma~\ref{lem:variational-regularity} gives $\nu>\theta$,
$0<z<\infty$, and $\lambda_+-1<\kappa<\infty$.
Recall that $J_\theta$ satisfies
$$
 J_\theta(z)
 =z\frac{\dd}{\dd z}
   \bigl(\theta\log A_b(z)+(1-\theta)\log C_b(z)\bigr).
$$
Setting the derivatives of the spatial objective in \eqref{eq:Psi} with respect
to $z$, $G$ with respect to $\nu$, and the clock objective in
\eqref{eq:chi} with respect to $\kappa$ equal to zero gives, respectively,
\begin{equation}
 \nu=J_\theta(z),\qquad
 z=\nu H(\kappa),\qquad
 u=-\nu\frac{H'(\kappa)}{H(\kappa)}.
 \label{eq:saddlesystem}
\end{equation}
The middle identity uses
$\partial_\nu G=\log\nu-\log z+\log H(\kappa)=0$.

Although $G$ is smooth near its maximizers, its maximum over $\nu$ can
have a corner when the maximizing count changes.
Danskin's formula \eqref{eq:compact-max-Hadamard} gives the directional
derivative of this maximum by taking the largest directional derivative
over the maximizing counts.  The right Hadamard derivative $D_H^+$
from Section~\ref{sec:root-curve} permits converging directions
$(\alpha,\beta)$ as the positive step size tends to zero.  We need this
when evaluating at $(\theta,u_*(\theta))$ as $\theta\downarrow0$.
Feasible directions satisfy $\alpha\ge0$, since $\theta\ge0$.

\begin{lemma}
\label{lem:sparsederivative}
The sets $\mathcal S_*$ and $\mathcal K_*$ are nonempty and compact, and
every $\kappa\in\mathcal K_*$ satisfies $\kappa>0$ and $H(\kappa)>1$.
For every $\alpha\ge0$ and $\beta\in\mathbb R$, the following right
directional derivative exists:
\begin{equation}
 D_H^+\widehat F(0,u_0)(\alpha,\beta)
 =\max_{\kappa\in\mathcal K_*}
     \bigl(\beta\kappa-\alpha\log H(\kappa)\bigr).              \label{eq:HadamardF}
\end{equation}
Consequently, \eqref{eq:sparseright} holds.
\end{lemma}

\begin{proof}
Lemmas~\ref{lem:variational-regularity} and
\ref{lem:global-root-model-only} give nonempty compact maximizing
saddle sets at $(0,u_0)$, with $\kappa\ge\kappa_{\rm rt}>0$.
At $\theta=0$, $C_b'(z)=A_b(z)$, so the first two equations in
\eqref{eq:saddlesystem} give, at every such saddle,
\begin{equation}
 \frac{A_b(z)}{C_b(z)}=\frac1{H(\kappa)}.                \label{eq:activeAH}
\end{equation}
Since $0<A_b(z)<C_b(z)$ for $z>0$, this also proves $H(\kappa)>1$.

The proof of Lemma~\ref{lem:global-root-model-only} already localizes all
nearby maximizing counts to a fixed compact set on which the inner
minimizers are smooth, including a smooth extension across $\theta=0$.
Apply \eqref{eq:compact-max-Hadamard} with
$p=\theta$ and $t=u$.  Equations~\eqref{eq:Genvelope} and
\eqref{eq:activeAH} give \eqref{eq:HadamardF} for all feasible directions
$\alpha\ge0$.  Applying \eqref{eq:compact-max-root-Hadamard} to the zero
$u_*(\theta)$, with these partial derivatives, gives
$$
 \lim_{\theta\downarrow0}\frac{u_*(\theta)-u_0}{\theta}
 =\min_{(\nu,z,\kappa)\in\mathcal S_*}
      \frac{-\partial_\theta G(0,u_0,\nu)}{\partial_uG(0,u_0,\nu)}
 =\min_{\kappa\in\mathcal K_*}\frac{\log H(\kappa)}{\kappa}=r_+.
$$
The minimum is attained and positive by compactness, proving
\eqref{eq:sparseright}.
\end{proof}

\begin{proof}[Proof of Theorem~\ref{thm:sparse-expansion}]
Lemma~\ref{lem:independent-center-location} proves
\eqref{eq:all-distance-center-log}, and
Lemma~\ref{lem:sparsederivative} proves \eqref{eq:sparseright}.
Substituting \eqref{eq:sparseright}, with $\theta=m_d/d$, into
\eqref{eq:all-distance-center-log} gives \eqref{eq:sparse-center-expansion}.
\end{proof}

\subsection{Proof of Theorem~\ref{thm:quenched-all-distance}}
\label{sec:all-distance-proof}

To prove Theorem~\ref{thm:quenched-all-distance}, we use {Proposition~\ref{prop:quenched-barrier}} to obtain a
positive hitting probability, then bootstrap this bound over descendant
families to control late hits, which leads to an upper bound for the FPT.  Expected target occupation controls early hits, giving the lower bound for the FPT.
To mitigate large death rates, we
use starting vertices whose rates are bounded below.  The annealed
conclusion then follows by averaging with the survival probabilities.

Recall from \eqref{eq:survival-probability} that
$s^\omega(x)=\Pw_x(\mathsf S)$ is the probability of survival from
one particle at $x$.  Its average over a uniformly chosen neighbour is
$$
 (Ps^\omega)(x)=\sum_vP(x,v)s^\omega(v)
 =\frac1{(b-1)d}\sum_{v:\,v\sim x}s^\omega(v).
$$
Here $P(x,v)$ is the one-mutation transition probability.
If $\Lambda_x<0$, the first event is mutation or death, so
$s^\omega(x)=(Ps^\omega)(x)/(1-\Lambda_x)$: the factor
$1/(1-\Lambda_x)$ is the probability of mutation before death.
The next lemma bounds the average survival probability after that
mutation away from zero and estimates $s^\omega(x)$ uniformly in $x$.

\begin{lemma}\label{lem:uniform-survival}
There is $c>0$ such that, on events of environment probability tending
to one, for every vertex $x$,
\begin{equation}
 \frac{c}{1-\Lambda_x}\le s^\omega(x)
 \le\min\left(1,\frac1{1-\Lambda_x}\right),
 \qquad (Ps^\omega)(x)\ge c.
 \label{eq:survival-comparison}
\end{equation}
For all sufficiently large $d$, also
$\inf_x\Eenv[s^\omega(x)]\gg1$.
\end{lemma}

\begin{proof}
The variables $(1-\Lambda_x)^{-1}$ are independent, take values in
$(0,(1-\lambda_+)^{-1}]$, and have mean $H(0)>1$.
Choose $\varepsilon>0$ with $H(0)\ge1+4\varepsilon$, and let
$$
 G=\left\{x:\frac1{(b-1)d}\sum_{v:\,v\sim x}
                    \frac1{1-\Lambda_v}\ge1+2\varepsilon\right\}.
$$
We first show that, for each fixed $\eta>0$, with probability tending
to one every vertex has at most $\eta(b-1)d$ neighbours outside $G$.

Fix a vertex $x$ and a set $T$ of $k$ of its neighbours.  In the sum
$\sum_{v\in T}\sum_{z:\,z\sim v}(1-\Lambda_z)^{-1}$,
keep only vertices $z$ at distance two from $x$.
Their contribution is $\sum_z w_z/(1-\Lambda_z)$, where $w_z$
counts the neighbours $v\in T$ through which $z$ is reached.  We have
$$
 0\le w_z\le2,\qquad
 \sum_z w_z=k(b-1)(d-1),\qquad
 \sum_z w_z^2\le2k(b-1)d.
$$
Indeed, a vertex at distance two can be reached through exactly two
neighbours of $x$; each $v\in T$ has $(b-1)(d-1)$ neighbours at distance
two from $x$.  If $T\subset G^c$, positivity of the omitted terms gives
$\sum_z w_z/(1-\Lambda_z)\le k(b-1)d(1+2\varepsilon)$.
For all sufficiently large $d$, this is at least $\varepsilon k(b-1)d$
below its expectation.  Hoeffding's inequality therefore gives
$$
 \Penv(T\subset G^c)
 \le\exp\left(-\varepsilon^2 k(b-1)d(1-\lambda_+)^2\right).
$$
Take $k=\lfloor\eta(b-1)d\rfloor+1$ and sum over the $b^d$ choices of
$x$ and at most $2^{(b-1)d}$ choices of $T$.  The resulting failure probability
is $\exp(-c_\eta d^2)$ for all sufficiently large $d$.

Choose $\eta\le\min(1/4,\varepsilon(1-\lambda_+))$.
Since $(1-\Lambda_x)^{-1}>0$ almost surely, we can choose $a_0>0$
so small that, with probability tending to one, every vertex has at
least $3(b-1)d/4$ neighbours with $(1-\Lambda_v)^{-1}\ge a_0$.
Indeed, writing $p=\Penv((1-\Lambda_x)^{-1}<a_0)$, the failure
probability is at most $b^d2^{(b-1)d}p^{(b-1)d/4}$, which tends to
zero exponentially when $p$ is sufficiently small.
Work on the intersection of the two events just established:
\begin{equation}
\begin{split}
 &\left\{\text{for every }x,\quad
       \#\{v:v\sim x,\ v\notin G\}\le\eta(b-1)d\right\},\\
 &\left\{\text{for every }x,\quad
       \#\left\{v:v\sim x,\ \frac1{1-\Lambda_v}\ge a_0\right\}
           \ge\frac34(b-1)d\right\}.
\end{split}
\label{eq:survival-neighbour-events}
\end{equation}
For $x\in G$, removing the neighbours outside $G$ loses at most
$\eta/(1-\lambda_+)$ from the average.  For every $x$, at least
the fraction $3/4-\eta\ge1/2$ of its neighbours are in $G$ and
satisfy $(1-\Lambda_v)^{-1}\ge a_0$.  Thus,
\begin{align*}
 \frac1{(b-1)d}\sum_{v:\,v\sim x}
    \frac{\one_{\{v\in G\}}}{1-\Lambda_v}
 &\ge1+2\varepsilon-\frac{\eta}{1-\lambda_+}
 \ge1+\varepsilon\qquad(x\in G),\\
 \frac1{(b-1)d}\sum_{v:\,v\sim x}
    \frac{\one_{\{v\in G\}}}{1-\Lambda_v}
 &\ge a_0\left(\frac34-\eta\right)\ge\frac{a_0}{2}
 \qquad\text{for every }x.
\end{align*}

For a fixed finite environment, the probability $u(t,x)$ that the
population is nonempty at time $t$ solves
$$
 \partial_tu(t,x)=(Pu(t,\cdot))(x)-(1-\Lambda_x)u(t,x)
                    -(\Lambda_x)_+u(t,x)^2,
 \qquad u(0,x)=1.
$$
Choose $c_1>0$ with $c_1\le1-\lambda_+$ and
$c_1\lambda_+/(1-\lambda_+)^2\le\varepsilon$, and put
$f(x)=c_1\one_G(x)/(1-\Lambda_x)$.
\begin{itemize}
\item On $G$,
$$
 Pf(x)-(1-\Lambda_x)f(x)-(\Lambda_x)_+f(x)^2
 \ge c_1\varepsilon-\frac{\lambda_+c_1^2}{(1-\lambda_+)^2}\ge0;
$$
\item Off $G$, the left side is $Pf(x)\ge0$.
\end{itemize}
Thus $f$ is a stationary subsolution below the initial value one.
The comparison principle for this finite cooperative system gives
$u(t,x)\ge f(x)$.  The functions $u(t,x)$ decrease to $s^\omega(x)$,
which satisfies the stationary equation, so
\begin{equation}
 s^\omega(x)
 =\frac{(Ps^\omega)(x)}{1-\Lambda_x+(\Lambda_x)_+s^\omega(x)}.
 \label{eq:survival-stationary}
\end{equation}
The denominator lies between $1-\Lambda_x$ and
$1+(-\Lambda_x)_+$, while $(Ps^\omega)(x)\ge c_1a_0/2$ and
$(Ps^\omega)(x)\le1$.  Consequently,
$$
 \frac{c_1a_0}{2(1+(-\Lambda_x)_+)}
 \le s^\omega(x)\le\frac1{1-\Lambda_x},
 \qquad
 \frac1{1+(-\Lambda_x)_+}\ge\frac{1-\lambda_+}{1-\Lambda_x}.
$$
Together with $s^\omega(x)\le1$, this proves
\eqref{eq:survival-comparison}, after decreasing $c$.
Finally, if $E_d$ is the environment event just constructed, then
$$
 \Eenv[s^\omega(x)]\ge c\Eenv\left[\frac{\one_{E_d}}{1-\Lambda_x}\right]
 \ge c\left(H(0)-\frac{\Penv(E_d^c)}{1-\lambda_+}\right)\gg1.
$$
The bound is uniform in the deterministic starting vertex, as claimed.
\end{proof}

We mark a particle if its descendants survive forever.  Conditional
on $\mathsf S$, at least one particle is marked, and a marked particle
has at least one marked child at every split.  The marked population
therefore never decreases.
The construction is the finite-state version of the survival
decomposition in~\cite[Theorem~3 and Corollaries~14--15]{HarrisHesseKyprianou2016}.
We use the fixed-time conclusion below in Section~\ref{sec:kpp-proof}
to obtain occupation at a common later time under the no-death assumption,
and the stopping-time conclusion in the proof of
Theorem~\ref{thm:quenched-all-distance}, where the trial probabilities
require a lower bound on the starting rates.

\begin{lemma}\label{lem:survival-preliminary-family}
There is $L_0<\infty$ such that, for every fixed integer $Q\ge1$ and
$\eta>0$, there are fixed $T,M<\infty$ and events of environment
probability tending to one with the following properties, uniformly
over the starting vertices $\mathbf x_m$, conditional on survival.
\begin{enumerate}
\item[(i)] With probability at least $1-\eta$, one can retain $Q$
marked particles at time $T$, all within distance $M$ of the start.
\item[(ii)] With probability at least $1-\eta$, one can stop $Q$
marked particles whose descendant trees are disjoint, by time $T$, all
within distance $M$ of the start, each at a vertex of rate at least
$-L_0$.  Given the marked-particle history up to these stops, their
descendant systems, including unmarked descendants, are independent
copies of the original process, each conditioned on its own survival.
\end{enumerate}
\end{lemma}

\begin{proof}
\begingroup
We may without loss of generality assume that $Q\ge2$, since $Q=1$ follows by retaining one of two.
Fix an environment in the event of Lemma~\ref{lem:uniform-survival}.

\smallskip
\noindent\emph{Step 1: rates of the marked particles.}
Under the unconditioned quenched law, given the ordinary population
history up to a fixed time, the particles  alive are marked
independently, with probabilities $s^\omega$ at their locations.
By the Bayes' rule, first-event conditioning gives the mutation rate
of a marked particle and its rate of producing two marked children:
\begin{equation}
 x\longrightarrow v:\quad P(x,v)\frac{s^\omega(v)}{s^\omega(x)},
\qquad x\longrightarrow(x,x):\quad
 (\Lambda_x)_+\frac{s^\omega(x)^2}{s^\omega(x)}
 =(\Lambda_x)_+s^\omega(x).
 \label{eq:survival-marked-rates}
\end{equation}

Equation~\eqref{eq:survival-stationary} gives
$(Ps^\omega)(x)=(1-\Lambda_x)s^\omega(x)
+(\Lambda_x)_+s^\omega(x)^2$.
Summing the mutation rates in \eqref{eq:survival-marked-rates} therefore gives
\begin{equation}
 \sum_vP(x,v)\frac{s^\omega(v)}{s^\omega(x)}
 =\frac{(Ps^\omega)(x)}{s^\omega(x)}
 =1-\Lambda_x+(\Lambda_x)_+s^\omega(x)\ge1-\lambda_+>0.
 \label{eq:marked-total-mutation-rate}
\end{equation}
If $\Lambda_x\ge\lambda_0>0$, then \eqref{eq:survival-comparison} also gives
$$
 \frac{(Ps^\omega)(x)}{s^\omega(x)}
 =1-\Lambda_x(1-s^\omega(x))\le1,
 \qquad (\Lambda_x)_+s^\omega(x)\ge c\lambda_0.
$$

\smallskip
\noindent\emph{Step 2: produce $Q$ marked lineages in bounded time.}
Choose $\lambda_0>0$ with $p=P([\lambda_0,\lambda_+])>0$;
such a choice is possible because $H(0)>1$.
For each fixed integer $J$, impose also the event
$$
 \#\{v:v\sim z,\ \Lambda_v\ge\lambda_0\}\ge p(b-1)d/2
 \quad\text{if }d_{\rm H}(z,\mathbf x_m)\le J
 \text{ for some }1\le m\le d.
$$
Each count is binomial, and there are at most $C_Jd^{J+1}$ such
vertices.  Therefore, the failure probability is at most
$C_Jd^{J+1}e^{-p(b-1)d/8}$.

Follow one marked lineage, taking the unique marked child when there is
one, and the first in a fixed ordering when there are two.  Denote its
mutation times by $\sigma_k$.  For $1\le k\le J$, let $B_k$ indicate
that the vertex visited at $\sigma_k$ has rate at least $\lambda_0$ and
that a split into two marked children occurs during
$[\sigma_k,\sigma_{k+1})$.  Reveal the mutation time before its destination, and
let $\mathcal F_k$ be this pre-destination history.  The previous
indicators are $\mathcal F_k$-measurable.  If $x$ is the location just
before $\sigma_k$, the probability of choosing a vertex of rate at
least $\lambda_0$ is
$$
 \frac{\sum_{v:\,v\sim x,\ \Lambda_v\ge\lambda_0}
              P(x,v)s^\omega(v)}{(Ps^\omega)(x)}\ge\frac{cp}{2},
$$
since $s^\omega(v)\ge c$ on these neighbours and $(Ps^\omega)(x)\le1$.
At that vertex, Step 1 bounds the probability of a split into two marked
children before the next mutation below by $c\lambda_0/(1+c\lambda_0)$.
Consequently,
$$
 \Pw_{\mathbf x_m}(B_k=1\mid\mathcal F_k,\mathsf S)
 \ge p_*:=\frac{cp}{2}\frac{c\lambda_0}{1+c\lambda_0}>0.
$$
Hence,
$$
 \Ew_{\mathbf x_m}[e^{-B_k}\mid\mathcal F_k,\mathsf S]
 \le1-p_*(1-e^{-1}).
$$
Iterating this bound and applying Markov's inequality gives
\begin{equation}
\begin{split}
 \Pw_{\mathbf x_m}\!\left(\sum_{k=1}^JB_k<Q-1\,\middle|\,\mathsf S\right)
 &\le e^{Q-1}\Ew_{\mathbf x_m}\left[e^{-\sum_{k=1}^JB_k}\mid\mathsf S\right]\\
 &\le e^{Q-1}\bigl(1-p_*(1-e^{-1})\bigr)^J.
\end{split}
\label{eq:marked-split-count-tail}
\end{equation}
Choose $J$ so that \eqref{eq:marked-split-count-tail} is at most $\eta/4$.
Freeze the other child at the first double-marked split in each of the
first $Q-2$ successful intervals, and freeze both children at that split
in the $(Q-1)$st successful interval. Reveal no future of a frozen child
beyond its survival mark. This gives exactly $Q$ genealogically disjoint
leaves, all at rates at least $\lambda_0$ and within distance $J$ of the
start, by time $\sigma_{J+1}$ whenever $\sum_{k=1}^JB_k\ge Q-1$.
Let $E_i$ be independent random variables distributed as $\mathrm{Exp}(1-\lambda_+)$.
By \eqref{eq:marked-total-mutation-rate}, choose a fixed $T$ with
\[
 \Pw_{\mathbf x_m}(\sigma_{J+1}>T\mid\mathsf S)
 \le\Pp\left(\sum_{i=1}^{J+1}E_i>T\right)\le\eta/4.
\]
Thus, all the stops occur by $T$ with probability at least $1-\eta/2$.
% The auxiliary continuation to $\sigma_{J+1}$ is used only for this bound;
% it is not revealed after the final freeze.

At a split, the two full descendant trees are independent; conditioning
both to survive gives the product of their survival-conditioned laws.
Iterating this observation along the explored marked spine shows that,
given the stopped marked history, all frozen full descendant systems are
independent and individually conditioned on survival. This proves (ii)
with $L_0=0$ and radius $J$.

For (i), continue one marked lineage from each frozen particle to time
$T$. At each birth site $y$, we have $s^\omega(y)\ge c$ by
\eqref{eq:survival-comparison}. Let $Z_n(t)$ count descendants at elapsed
time $t$ whose ancestral lines have made at least $n$ mutations since
that birth. Suppressing deaths and raising the splitting rate to one gives
\[
 \Ew_y[Z_n(t)]\le e^t\Pp(\operatorname{Poisson}(t)\ge n).
\]
If the chosen marked lineage has made $n$ mutations, its endpoint is
counted in $Z_n(t)$. Conditional on the stopped history, its mutation
count over the remaining time $t\le T$ therefore exceeds $n-1$ with
probability at most $c^{-1}e^T\Pp(\operatorname{Poisson}(T)\ge n)$.
Choose $n$ so that $Q$ times this bound is at most
$\eta/2$. A union bound gives $Q$ marked particles at time $T$, each
within distance $M=J+n$ of the start, with probability at least $1-\eta$.
This same $M$ also works in (ii), and all constants and environment
events have the required uniformity.
\endgroup
\end{proof}

\begin{proof}[Proof of Theorem~\ref{thm:quenched-all-distance}]
\begingroup
Fix $\eta>0$. Choose $L\ge L_0$ from
Lemma~\ref{lem:survival-preliminary-family} so large that
$\Penv(\Lambda_{\mathbf0}<-L)<\eta$.
Proposition~\ref{prop:quenched-barrier} supplies $p_L>0$ independent
of every subsequently fixed starting radius. On the event of
Lemma~\ref{lem:uniform-survival}, if $\Lambda_{\mathbf0}\ge-L$, the
branching property at a deterministic time $t$ gives
$$
 \Pw_y(N_t(\mathbf0)>0,\mathsf S)
 \ge s^\omega(\mathbf0)\Pw_y(N_t(\mathbf0)>0).
$$
Indeed, on target occupation, select one particle there and require its
descendants to survive. Dividing by $s^\omega(y)\le1$ and using
\eqref{eq:survival-comparison}, we obtain
$$
 \Pw_y(N_t(\mathbf0)>0\mid\mathsf S)
 \ge\frac{c}{1+L}\Pw_y(N_t(\mathbf0)>0).
$$

Choose $Q$ so that $(1-cp_L/(1+L))^Q<\eta$, and apply
Lemma~\ref{lem:survival-preliminary-family}(ii) to obtain fixed
$T,M<\infty$. Next, choose $a_M$ from
Proposition~\ref{prop:quenched-barrier}. Intersect the environment
events in these two results and Lemma~\ref{lem:uniform-survival}.
Their probabilities tend to one, and all choices made so far are
independent of $d$.

We restrict to this intersection and on $\Lambda_{\mathbf0}\ge-L$.
Conditional on survival, with probability at least $1-\eta$, the
preliminary construction stops $Q$ marked particles by time $T$,
at vertices $y$ within distance $M$ of $\mathbf x_m$ and with
$\Lambda_y\ge-L$. Given the stopped marked history, their full
descendant systems are independent, each conditioned on its own
survival. By \eqref{eq:quenched-hit}, every such system occupies
the target at elapsed time $\widehat\theta_{d,m}+a_M$ with conditional
probability at least $cp_L/(1+L)$. This statement includes a stopped
particle at the target. All these trial times occur before the
absolute deadline $\widehat\theta_{d,m}+T+a_M$. Consequently,
$$
 \sup_{1\le m\le d}
 \Pw_{\mathbf x_m}\!\left(
  \tau_{\mathbf x_m}^\omega>\widehat\theta_{d,m}+T+a_M
  \,\middle|\,\mathsf S\right)
 \le\eta+\left(1-\frac{cp_L}{1+L}\right)^Q<2\eta.
$$
This proves the required upper-tail bound by the amplification argument
of~\cite[Lemma~25 and the upper-bound proof of Theorem~2]{BZ}.

For the lower tail, intersect also the environment event of
Proposition~\ref{prop:global-center}, including its radius-one version.
Following one child after every split at the target, the probability
of neither death nor mutation during the next unit interval is at
least $e^{-(1+L)}$. Therefore, for $y\ne\mathbf0$ and $t\ge0$,
$$
 e^{-(1+L)}\Pw_y(\tau_y^\omega\le t)
 \le\int_0^{t+1}K_s^\omega(y,\mathbf0)\,\dd s.
$$
For every fixed $A\ge1$, Proposition~\ref{prop:global-center} bounds
the right-hand side by $O_L(e^{-c(A-1)}+d^{-1})$ at
$t=\widehat\theta_{d,m}-A$, uniformly over
$d_{\mathrm H}(y,\mathbf x_m)\le1$ with $y\ne\mathbf0$.

If $x=\mathbf x_m$ has $\Lambda_x\ge0$, divide this bound by
$s^\omega(x)\ge c$ from \eqref{eq:survival-comparison}.
If $\Lambda_x<0$, there is no split before the first mutation or
death, and the survival equation gives
$s^\omega(x)=(Ps^\omega)(x)/(1-\Lambda_x)$.
Conditional on survival, the first mutation time is exponential of
rate $1-\Lambda_x$, and its destination has distribution
$P(x,v)s^\omega(v)/(Ps^\omega)(x)$. Thus,
\begin{align*}
 \Pw_x(\tau_x^\omega\le t\mid\mathsf S)
 &=\frac1{(Ps^\omega)(x)}
   \sum_vP(x,v)s^\omega(v)\int_0^t
   (1-\Lambda_x)e^{-(1-\Lambda_x)r}
   \Pw_v(\tau_v^\omega\le t-r\mid\mathsf S)\,\dd r\\
 &\le\frac1{(Ps^\omega)(x)}
        \sum_vP(x,v)\Pw_v(\tau_v^\omega\le t).
\end{align*}
Here $s^\omega(v)$ times the conditional probability is the
probability of hitting and survival, hence at most the unconditional
hitting probability. The denominator is bounded below by
\eqref{eq:survival-comparison}. Every neighbour $v\ne\mathbf0$
satisfies the preceding radius-one bound. If $x$ is adjacent to the
target, its target neighbour contributes at most $1/((b-1)d)$ to
the sum. Hence, uniformly in $m$,
$$
 \Pw_{\mathbf x_m}\!\left(
  \tau_{\mathbf x_m}^\omega\le\widehat\theta_{d,m}-A
  \,\middle|\,\mathsf S\right)
 \ll_L e^{-c(A-1)}+d^{-1}.
$$
Choose a fixed large $A\ge T+a_M$ after these parameters.
The excluded target-rate event has probability less than $\eta$;
all other environment failure probabilities tend to zero. Since
$\eta$ is arbitrary, the two tail bounds prove
\eqref{eq:all-distance-window}.

Finally, Lemma~\ref{lem:uniform-survival} bounds
$\Eenv[s^\omega(\mathbf x_m)]$ below uniformly in $m$.
For every $\varepsilon>0$, the ratio in
\eqref{eq:all-distance-annealed} is at most a constant times
$$
 \varepsilon+
 \Penv\!\left(
  \sup_{1\le m\le d}\Pw_{\mathbf x_m}\!\left(
   |\tau_{\mathbf x_m}^\omega-\widehat\theta_{d,m}|>A
   \,\middle|\,\mathsf S\right)>\varepsilon\right).
$$
Indeed, its numerator averages the conditional tail with weight
$s^\omega(\mathbf x_m)\le1$. Taking first the dimension limit,
then $A\to\infty$, and finally $\varepsilon\downarrow0$ proves
\eqref{eq:all-distance-annealed}.
\endgroup
\end{proof}

\subsection{Proofs of comparative statics results}

\subsubsection{Proof of Corollary~\ref{thm:mutation-rate}}
\label{sec:mutation-rate-proof}

We compare first-passage times in physical units, keeping the signed
rates fixed.  Changing $\varrho$ therefore changes the normalized rate law
to which we apply the rate-one results.

Let the mutation walk have generator $\varrho(P-I)$ and write
$\mathbb P_x^{\omega,\varrho}$ and $\tau_x^{\omega,\varrho}$ for the
corresponding quenched law and first-passage time.  Deterministic time change gives,
for every Borel set $B\subseteq[0,\infty]$,
\begin{equation}
 \mathbb P_x^{\omega,\varrho}
       \bigl(\varrho\tau_x^{\omega,\varrho}\in B\bigr)
 =\mathbb P_x^{\omega/\varrho,1}
       \bigl(\tau_x^{\omega/\varrho,1}\in B\bigr).      \label{eq:mutation-time-change}
\end{equation}
After the clock change, mutation has total rate one, splitting at $v$
has rate $(\Lambda_v)_+/\varrho$, and death has rate
$(-\Lambda_v)_+/\varrho$.  The scaled law satisfies
Assumption~\ref{ass:environment} precisely when
\eqref{eq:mutation-supercritical} holds.  The time change preserves eternal
survival, so \eqref{eq:mutation-time-change} also holds after conditioning
both laws on survival.

Let $\widehat F^{[\varrho]}$ denote the exponent for the law of
$\Lambda_{\mathbf0}/\varrho$, whose root is $u_*^{[\varrho]}$.
The subscript $\varrho$ below instead denotes mutation rate $\varrho$
with time in physical units.  Following \eqref{eq:H}, \eqref{eq:chi}, and \eqref{eq:F}, we define
\begin{equation}
 H_\varrho(\beta)
 =\int\frac{P(d\lambda)}{\beta+\varrho-\lambda},
 \qquad \beta>\lambda_+-\varrho,                       \label{eq:Hphysical}
\end{equation}
and
\begin{equation}
 \widehat F_\varrho(\theta,u)
 =\sup_{\nu\ge0}\left\{
    \Psi_b(\theta,\nu)
    +\inf_{\beta>\lambda_+-\varrho}
       \left(\beta u+\nu\log\bigl(\varrho
                    H_\varrho(\beta)\bigr)\right)
  \right\}.                                             \label{eq:Fphysical}
\end{equation}
The substitution $\beta=\varrho\kappa$ gives
\begin{equation}
 \widehat F_\varrho(\theta,u)
 =\widehat F^{[\varrho]}(\theta,\varrho u).             \label{eq:Fphysical-scaling}
\end{equation}
Consequently, its unique positive zero is
\begin{equation}
 \frac{u_*^{[\varrho]}(\theta)}{\varrho}.              \label{eq:vstar}
\end{equation}
Write $v_0(\varrho)=u_*^{[\varrho]}(0)/\varrho$, which is the first-order coefficient in the FPT asymptotics when $m=o(d)$.

For $j\ge1$ and $\beta>\lambda_+-\varrho$, set
\begin{equation}
 M_{j,\varrho}(\beta)
  =\int\frac{P(d\lambda)}{(\beta+\varrho-\lambda)^j}.
                                                               \label{eq:mutation-Mj}
\end{equation}
Thus, $M_{1,\varrho}=H_\varrho$.  Moreover,
\begin{equation}
 \chi_\varrho(\nu,u)
  =\inf_{\beta>\lambda_+-\varrho}
       \left(\beta u+\nu\log\bigl(\varrho
                         M_{1,\varrho}(\beta)\bigr)\right).
                                                               \label{eq:mutation-clock-objective}
\end{equation}
The factor $\varrho$ inside the logarithm records the factor
$\varrho^n$ carried by a path with $n$ mutation jumps.
Write
$G_\varrho(\theta,u,\nu)
 =\Psi_b(\theta,\nu)+\chi_\varrho(\nu,u)$.

\begin{lemma}
\label{lem:mutation-profile-monotonicity}
The function $v_0(\varrho)=u_*^{[\varrho]}(0)/\varrho$ is strictly
increasing on each interval satisfying \eqref{eq:mutation-supercritical}.
\end{lemma}

\begin{proof}
Let $(\nu,z,\beta)$ be any maximizing saddle in \eqref{eq:Fphysical}
at $(0,v_0(\varrho))$.  Lemma~\ref{lem:variational-regularity} applied
to the scaled law gives $\nu>0$ and interior spatial and clock
minimizers.  Lemma~\ref{lem:global-root-model-only} gives
$\beta=\varrho\kappa>0$, where $\kappa$ is the corresponding scaled
clock tilt.  Stationarity in the spatial, count, and clock variables gives
\begin{equation}
 \nu=z\frac{A_b(z)}{C_b(z)},\qquad
 z=\varrho\nu M_{1,\varrho}(\beta),\qquad
 v_0(\varrho)=\nu\frac{M_{2,\varrho}(\beta)}
                  {M_{1,\varrho}(\beta)}.              \label{eq:mutation-saddle-system}
\end{equation}
The first two equations imply
\begin{equation}
 \varrho M_{1,\varrho}(\beta)
   =\frac{C_b(z)}{A_b(z)}>1.                            \label{eq:mutation-H-greater-one}
\end{equation}
Since $P$ is a probability measure, Cauchy--Schwarz gives
$M_{2,\varrho}\ge M_{1,\varrho}^2$.  Therefore,
\begin{equation}
 v_0(\varrho)-\frac{\nu}{\varrho}
 =\nu\left(\frac{M_{2,\varrho}}{M_{1,\varrho}}
             -\frac1\varrho\right)>0.                 \label{eq:mutation-strict-sign}
\end{equation}
All terms in this display are evaluated at $\beta$.  Strictness follows
from \eqref{eq:mutation-H-greater-one}, even for a constant branching rate.

To compare maximizing counts at nearby rates $\varrho_1<\varrho_2$, first put their objectives
on a common clock domain.  For this change of variables, use
$\kappa=\beta+\varrho-1$ in
\eqref{eq:mutation-clock-objective}.  This maps the constraint onto
$\kappa>\lambda_+-1$ and gives
$M_{1,\varrho}(\beta)=H(\kappa)$, with $H$ the transform of the
original law.  Hence
\begin{equation}
 G_\varrho(\theta,u,\nu)
 =G(\theta,u,\nu)+\nu\log\varrho-(\varrho-1)u,
 \qquad \varrho>0,\ u>0,\ \nu\ge0,                     \label{eq:mutation-affine-tilt}
\end{equation}
with $G$ the rate-one objective of the original law.  In particular,
$\partial_\varrho G_\varrho=\nu/\varrho-u$ at every feasible count.

Fix a base rate $\varrho_0$ satisfying \eqref{eq:mutation-supercritical}
and its root $v_0(\varrho_0)$.  The condition persists in a neighbourhood
of $\varrho_0$ by dominated convergence.
On a compact neighbourhood of this point where time and rate are positive,
\eqref{eq:mutation-affine-tilt} and \eqref{eq:UO12} give
$$
 G_\varrho(0,u,\nu)
 \le-(\varrho-\lambda_+)u
   +\nu\left(1+\log\frac{\varrho u}{\nu}\right),
$$
which tends uniformly to $-\infty$ as $\nu\to\infty$.
Keeping one maximizing count at the base point bounds the maximum below
uniformly nearby.  Since $G$ is continuous also at $\nu=0$, the maximum
can therefore be restricted to a fixed compact count interval.
Berge's maximum theorem~\cite{FeinbergKasyanovVoorneveld2014} gives
joint continuity of the value and upper semicontinuity of the maximizing
count sets.

At the base point, apply Lemmas~\ref{lem:variational-regularity} and
\ref{lem:global-root-model-only} to the single scaled law
$\Lambda_{\mathbf0}/\varrho_0$.  The set of maximizing saddles is compact and
interior, with $\beta>0$ and the strict sign
\eqref{eq:mutation-strict-sign}.  The inner minimizers extend smoothly
to a neighbourhood of these maximizing counts.  By continuity and upper
semicontinuity, we may restrict all nearby maxima to a fixed compact
set of counts on which
\begin{equation}
 \partial_uG_\varrho=\beta>0,\qquad
 \partial_\varrho G_\varrho=\frac\nu\varrho-u<0.
 \label{eq:mutation-envelope-signs}
\end{equation}
Compactness makes both signs uniform.  The local-zero argument following
\eqref{eq:compact-max-local-bounds}, with $p=\varrho$ and $t=u$,
therefore gives a continuous root for nearby rates.  This root equals
$v_0(\varrho)$ by \eqref{eq:Fphysical-scaling} and the uniqueness in
Lemma~\ref{lem:global-root-model-only} for that scaled law.

For nearby $\varrho_1<\varrho_2$, integrate the rate derivative in
\eqref{eq:mutation-envelope-signs} and take maxima over the same fixed
count set.  This gives
$$
 \widehat F_{\varrho_2}\bigl(0,v_0(\varrho_1)\bigr)
 <\widehat F_{\varrho_1}\bigl(0,v_0(\varrho_1)\bigr)=0.
$$
Since $u\mapsto\widehat F_{\varrho_2}(0,u)$ increases near its zero,
$v_0(\varrho_2)>v_0(\varrho_1)$.  A finite cover of any compact
admissible rate interval, followed by a subdivision into intervals
contained in these neighbourhoods, proves strict increase throughout.
\end{proof}

\begin{proof}[Proof of Corollary~\ref{thm:mutation-rate}]
Fix $\varrho$ satisfying \eqref{eq:mutation-supercritical} and deterministic
integers $1\le m_d\le d$.  We apply
Theorem~\ref{thm:quenched-all-distance} to the rate-one model with
rate law $\mathcal L(\Lambda_{\mathbf0}/\varrho)$ and use the scaling relation 
\eqref{eq:mutation-time-change}.  Theorems~\ref{thm:quenched-all-distance}
and~\ref{thm:sparse-expansion} give, uniformly in $m$,
\[
 \varrho\tau_{\mathbf{x}_m}^{\omega,\varrho}
 =d u_*^{[\varrho]}(m/d)
  +O_{\mathbb P_{\mathbf{x}_m}^{\omega,\varrho}(\cdot\mid\mathsf S)}(\log d)
 \quad\text{in $\Penv$-probability}.
\]
Dividing by $\varrho d$ gives
\begin{equation}
 \frac{\tau_{\mathbf{x}_{m_d}}^{\omega,\varrho}}d
 =\frac{u_*^{[\varrho]}(m_d/d)}{\varrho}
  +o_{\mathbb P_{\mathbf{x}_{m_d}}^{\omega,\varrho}(\cdot\mid\mathsf S)}(1)
 \quad\text{in $\Penv$-probability}.                  \label{eq:mutation-rate-Op}
\end{equation}
{Fix two rates $\varrho_1<\varrho_2$ in $I$.}
Lemma~\ref{lem:mutation-profile-monotonicity} gives
$$
 \frac{u_*^{[\varrho_2]}(0)}{\varrho_2}
 -\frac{u_*^{[\varrho_1]}(0)}{\varrho_1}
 =v_0(\varrho_2)-v_0(\varrho_1)>0.
$$
By Lemma~\ref{lem:global-root-model-only}, applied to each scaled law,
$\theta\mapsto\frac{u_*^{[\varrho_2]}(\theta)}{\varrho_2}
 -\frac{u_*^{[\varrho_1]}(\theta)}{\varrho_1}$ is continuous. It therefore remains
positive on $[0,\varepsilon]$ for some
$\varepsilon=\varepsilon(\varrho_1,\varrho_2)\in(0,1]$, proving the stated
inequality.  Its minimum on this interval is positive and its maximum
is finite.  {Applying \eqref{eq:mutation-rate-Op} at both
rates and taking a union bound gives an order-$d$ delay whenever
$1\le m_d\le\varepsilon d$, under any joint law of the two quenched
passage times, each conditioned on its own survival.  Environments
where either survival probability is zero have probability tending
to zero by Lemma~\ref{lem:uniform-survival}, applied to the two scaled laws.}
\end{proof}

\subsubsection{Proof of Corollary~\ref{cor:rate-convex-order}}
\label{sec:convex-order-proof}

\begin{proof}[Proof of Corollary~\ref{cor:rate-convex-order}]
The case $P_1=P_2$ is immediate, so suppose the laws differ.
For $i=1,2$, $h_i(r)$ is the expected number of descendants whose
ancestral path has made no mutation by time $r$, averaged over the
starting rate with law $P_i$:
$$
 h_i(r)=\int e^{-(1-\lambda)r}P_i(d\lambda),\qquad r\ge0.
$$
For every $s$, convex order gives
$\int(\lambda-s)_+(P_2-P_1)(d\lambda)\ge0$.
As a function of $s$, this integral is continuous and zero above the
larger essential upper endpoint.  It is positive on some interval, since the functions
$s\mapsto\int(\lambda-s)_+P_i(d\lambda)$ determine their laws.
The identity $e^{r\lambda}=r^2\int_{-\infty}^{\lambda}
e^{rs}(\lambda-s)\,\dd s$ and Tonelli's theorem give, for every $r>0$,
\[
 h_2(r)-h_1(r)
 =r^2e^{-r}\int_{\mathbb R}e^{rs}
       \left(\int(\lambda-s)_+(P_2-P_1)(d\lambda)\right)\dd s>0.
\]
Positivity of convolution in \eqref{eq:freshK} yields
$\widehat K_{2,d,m}(t)>\widehat K_{1,d,m}(t)$ for $t>0$, and hence
$$
 \widehat M_{2,t}(\mathbf x_m)
 =\int_t^{t+1}\widehat K_{2,d,m}(s)\,\dd s
 >\int_t^{t+1}\widehat K_{1,d,m}(s)\,\dd s
 =\widehat M_{1,t}(\mathbf x_m),\qquad t\ge0.
$$
The definition \eqref{eq:thetahat definition} gives
$\widehat\theta_{2,d,m}\le\widehat\theta_{1,d,m}$.

To obtain strictness on the macroscopic scale, write
$\lambda_{+,i}=\operatorname*{ess\,sup}_{\lambda\sim P_i}\lambda$,
$i=1,2$, for the essential upper endpoints of the two laws.  We have
$\lambda_{+,1}\le\lambda_{+,2}$ by applying convex order to
$(\lambda-\lambda_{+,2})_+$.  Fix $\nu,u>0$, and choose the second
law's minimizing clock tilt $\kappa_2$ from
Lemma~\ref{lem:variational-regularity}, allowing its finite boundary
minimum.  Then $\kappa_2\ge\lambda_{+,2}-1\ge\lambda_{+,1}-1$
and $H_2(\kappa_2)<\infty$.  Since $h_2(r)>h_1(r)$ for every $r>0$,
also $H_1(\kappa_2)<\infty$, with endpoint values taken by continuity, and
$$
 H_2(\kappa_2)-H_1(\kappa_2)
 =\int_0^\infty e^{-\kappa_2r}\bigl(h_2(r)-h_1(r)\bigr)\,\dd r>0.
$$
It follows that
\[
 \chi_2(\nu,u)
 =\kappa_2u+\nu\log H_2(\kappa_2)
 >\kappa_2u+\nu\log H_1(\kappa_2)
 \ge\chi_1(\nu,u).
\]
At $u=u_{*,1}(\theta)$, every maximizing count $\nu_1$ for the
first law is positive by Lemma~\ref{lem:variational-regularity}.
Hence
\[
 \widehat F_2(\theta,u)
 \ge\Psi_b(\theta,\nu_1)+\chi_2(\nu_1,u)
 >\Psi_b(\theta,\nu_1)+\chi_1(\nu_1,u)=0.
\]
The sign of $\widehat F_2(\theta,\cdot)$ on either side of its unique positive zero,
proved in Lemma~\ref{lem:global-root-model-only}, gives
$u_{*,2}(\theta)<u_{*,1}(\theta)$.

By Theorems~\ref{thm:sparse-expansion}
and~\ref{thm:quenched-all-distance}, for any deterministic integers
$1\le m_d\le d$, the passage time divided by $d$ differs from
$u_{*,i}(m_d/d)$ by a quantity tending to zero
in quenched probability conditional on survival, in environment
probability, under law $P_i$.
This gives the claimed comparison of the leading times.
\end{proof}

\subsection{Proof of Corollary~\ref{cor:kpp-arrival}}
\label{sec:kpp-proof}

\begingroup
Similarly to the proof of Theorem~\ref{thm:quenched-all-distance}, we again use independent descendant families, but now observed at a common
time. A related amplification argument appears
in~\cite[Section~4.3]{DrewitzSchmitz2022}.

\begin{proof}[Proof of Corollary~\ref{cor:kpp-arrival}]
We first exclude early crossings. Since there are no deaths, following
one child after each split and requiring no mutation at the target gives
$K_{t+h}^\omega(y,\mathbf0)\ge e^{-h}K_t^\omega(y,\mathbf0)$ for
$0\le h\le1$. In particular, $K_t^\omega(y,\mathbf0)\le eM_t^\omega(y)$.
For each fixed $A\ge1$ and all sufficiently large $d$,
Proposition~\ref{prop:global-center} therefore gives, uniformly 
in $m$,
$$
 \sup_{0\le s\le\widehat\theta_{d,m}-A}
 v^\omega(s,\mathbf x_m)
 \le e\int_0^{\widehat\theta_{d,m}-A+1}
            K_r^\omega(\mathbf x_m,\mathbf0)\,\dd r
 \ll e^{-c(A-1)}+d^{-1}
$$
on events of environment probability tending to one. Choose $A$
large enough that this bound is less than $\alpha$ for all large $d$.
Since $w^\omega\le v^\omega$, this excludes early crossings for both
equations.

For the late bound, let $p_0>0$ be the constant in
Proposition~\ref{prop:quenched-barrier} with $L=0$.
Fix $\eta>0$ with $2\eta<1-\alpha$, and choose $Q$ so that
$(1-p_0)^Q<\eta$. The no-death assumption gives $s^\omega\equiv1$.
Lemma~\ref{lem:survival-preliminary-family}(i) supplies fixed
$T,M<\infty$ such that, on events of environment probability tending
to one, there are $Q$ particles within distance $M$ of the start at
time $T$ with quenched probability at least $1-\eta$, uniformly in
$m$. Choose the fixed $a_M$ in Proposition~\ref{prop:quenched-barrier}
and intersect its environment event with the preceding one.

On the event that these particles are available, choose them using
the population history through time $T$. Their descendant systems
are conditionally independent under the ordinary quenched law.
Equation~\eqref{eq:quenched-hit} gives success probability at least
$p_0$ for each family at the same elapsed time
$\widehat\theta_{d,m}+a_M$, including families started at
$\mathbf0$. Therefore, uniformly in $m$,
$$
 w^\omega(\widehat\theta_{d,m}+T+a_M,\mathbf x_m)
 \ge1-\eta-(1-p_0)^Q>\alpha.
$$
This proves the late bound for $w^\omega$, and hence for $v^\omega$.
The comparison $w^\omega\le v^\omega$ implies
$\mathcal T_\alpha[v^\omega](\mathbf x_m)
 \le\mathcal T_\alpha[w^\omega](\mathbf x_m)$.
The preceding early- and late-crossing bounds prove the first assertion. On their common event, the nonlinear crossing time exceeds the linear one by at most \(2A\), which proves the second assertion.
\end{proof}
\endgroup

% Keep the notation heading with the table opening when space is short.
\par
\ifdim\dimexpr\pagegoal-\pagetotal\relax<12\baselineskip
  \vfil\break
\fi
\appendix

\section{Frequently used notation}
\label{sec:notations}

% The superscript $I$ in the branching integrals denotes the nonterminal
% range $0\le r\le t-\delta d$.
% The scalar kernels in $\widehat{\mathcal C}_{q,t}^{I}$ average
% rates before the continuation power, whereas
% $\widehat{\mathcal B}_{t,r}^{(1+q)}$ retains the rate shared at the split.

\begingroup
\small
\setlength{\LTpre}{0.4\baselineskip}
\setlength{\LTpost}{0pt}
\renewcommand{\arraystretch}{1.18}
\begin{longtable}{@{}>{\raggedright\arraybackslash}p{0.21\textwidth}
                       >{\raggedright\arraybackslash}p{0.74\textwidth}@{}}
\caption{Notation used throughout the paper.}\label{tab:notation}\\
\toprule
\textbf{Notation} & \textbf{Meaning} \\
\midrule
\endfirsthead
\multicolumn{2}{@{}l}{\small\itshape Table~\ref{tab:notation} continued}\\
\toprule
\textbf{Notation} & \textbf{Meaning} \\
\midrule
\endhead
\midrule
\multicolumn{2}{r@{}}{\small\itshape Continued on next page}\\
\endfoot
\bottomrule
\endlastfoot
$b,d,\Hc_d^{(b)}$
  & The alphabet size, dimension, and $b$-ary Hamming cube;
    Section~\ref{sec:model}. \\
$d_{\mathrm H},\mathbf 0,\mathbf{x}_m$,\newline
$P(x,y),P^n$
  & Hamming distance, the target vertex, the canonical vertex at distance $m$
    from the target, and the one-step mutation matrix and its powers;
    Section~\ref{sec:main-results}. \\
$\Lambda_v$
  & The random net growth rate at vertex $v$; splitting and death rates
    are $(\Lambda_v)_+$ and $(-\Lambda_v)_+$;
    Assumption~\ref{ass:environment}. \\
$P(d\lambda),\lambda_+,\bar\lambda$
  & The rate law, its essential upper bound, and the mean splitting rate
    $\Eenv[(\Lambda_{\mathbf0})_+]$;
    Assumption~\ref{ass:environment} and \eqref{eq:atom-summary}. \\
$\Penv,\Eenv$
  & Probability and expectation over the static environment. \\
$\Pw_x,\Ew_x$
  & Quenched probability and expectation for the BRW started at $x$, with the
    environment held fixed; Section~\ref{sec:model}. \\
$\mathsf S,s^\omega(x)$
  & Eternal survival and its quenched probability from $x$;
    \eqref{eq:survival-probability}. \\
$\Pp_x,\E_x$
  & Probability and expectation for the mutation walk started at $x$;
    Section~\ref{sec:heuristics}. \\
$N_t(y),\tau_x^\omega$
  & The number of particles at $y$ at time $t$ and the first time a descendant
    of the particle started at $x$ reaches $\mathbf 0$;
    Section~\ref{sec:model}. \\
$K_t^\omega,M_t^\omega$
  & The quenched first-moment kernel and its integral over the unit window
    $[t,t+1]$; Section~\ref{sec:heuristics} and the definition before
    Proposition~\ref{prop:global-center}, respectively. \\
$h,H$
  & The environment-averaged one-holding-time weight and its Laplace
    transform; Section~\ref{sec:heuristics} and \eqref{eq:H}. \\
$\widehat K_{d,m},\widehat M_t,\widehat\theta_{d,m}$
  & The deterministic independent-visit kernel, its unit-window integral, and
    the first time that integral reaches one;
    \eqref{eq:freshK}, \eqref{eq:fresh-unit-window}, and
    Section~\ref{sec:main-results}. \\
$\Psi_b,\chi,G,\widehat F$
  & The functions in the logarithmic estimates for the two factors in
    \eqref{eq:freshK}, their fixed-count sum, and
    the result of maximizing that sum over the mutation-count density;
    \eqref{eq:Psi}--\eqref{eq:F} and Section~\ref{sec:main-results}. \\
$u_*(\theta),r_+$
  & The unique positive zero of $u\mapsto\widehat F(\theta,u)$ and the right
    derivative of $u_*$ at $\theta=0$;
    Section~\ref{sec:main-results} and \eqref{eq:sparseright}. \\
$\mathbf K_t^P,\mathbf k_t$
  & The forward first-moment density in the terminal rate with respect to
    $P$, and {the expected target count from one particle
    with a fixed initial rate};
    \eqref{eq:typed-kernel-definitions}. \\
$\widehat{\mathcal B}_{t,r}^{(1+q)}$
  & The fractional branching sum at split time $r$, retaining the same
    inherited rate in the trunk, split factor, and continuation;
    \eqref{eq:fractional-branching-sum}. \\
$\widehat{\mathcal C}_{q,t}^{I}$
  & {The scalar independent-visit branching integral with
    nonterminal cutoff; \eqref{eq:fresh-qnormalizer}.}
    It normalizes the split law. \\
$\mathcal C_{q,t}^{\omega}$, 
$\mathcal C_{q,t}^{\omega,I}$
  & The quenched branching integral over all split times and its
    nonterminal restriction; \eqref{eq:quenched-branching-integral}
    and the definition following it. \\
$\mathbb Q_{d,m,t},\mathbb Q_{x,y,t}$
  & The independent-visit first-moment space--time bridge laws, normalized
    by their endpoint kernels; \eqref{eq:Qbridge}. \\
$\mathbb Q_{q;y,t}$, 
$\mathbb Q_{q,j;y,t}$
  & The split marginal, \eqref{eq:weighted-split-law}, and its
    extension by a trunk and $j$ independent arms conditional on the split,
    for $j=1,2$; Section~\ref{sec:initial-branch}. \\
\end{longtable}
\endgroup

\section{Finite-state and rate-coordinate semigroup tools}
\label{sec:linearprimer}

\begingroup
\subsection{Jump-count expansions and conjugation}

On a finite state space, write $A=D+J$, where $D$ is diagonal and
$J$ is the nonnegative off-diagonal part.  Conditioning on the ordered
transition times gives
$$
 e^{t(D+J)}
 =\sum_{n=0}^\infty
 \int_{0<s_1<\cdots<s_n<t}
 e^{s_1D}J e^{(s_2-s_1)D}J\cdots J e^{(t-s_n)D}
 \dd s_1\cdots\dd s_n,
$$
where the $n=0$ term is $e^{tD}$.  Each term has nonnegative entries
and contains exactly $n$ transitions.  Replacing $J$ by $e^aJ$ therefore
weights that term by $e^{an}$.  For any real-valued function $r$ and
$c>0$, conjugation by
$x\mapsto c^{r(x)}$ multiplies the $(x,y)$ entry of both the operator
and its semigroup by $c^{r(y)-r(x)}$; the factors at intermediate
vertices cancel in every term.

For the independent-visit model, work on $L^2$ of counting measure on
$\Hc_d^{(b)}$ times $P$.  Here $D$ is multiplication by $-(1-i)$ and
$J$ is the mutation operator that resamples the rate from $P$.
The multiplication semigroup $e^{tD}$ is strongly continuous, with norm
at most $e^{-(1-\lambda_+)t}$, and $J$ is bounded, with $\|J\|\le1$.
The same expansion converges in operator norm because its $n$th term
has norm at most $e^{-(1-\lambda_+)t}\|J\|^nt^n/n!$.
For a bounded nonnegative terminal function $f$, it is also defined
pointwise for every finite initial rate $i\le\lambda_+$: the $n$th term
is at most $\|f\|_\infty e^{-(1-\lambda_+)t}t^n/n!$, and their sum is
at most $\|f\|_\infty e^{\lambda_+t}$.  Thus, the jump-count and
conjugation identities remain valid when rates are unbounded below;
for conjugation, the positive multiplying function and its reciprocal
are bounded.  The expansion includes the zero-mutation term with its
unchanged rate.

Weighting the initial rate by $P$, then reversing the vertices, rate
labels, and holding durations in each term preserves its weight:
the mutation matrix is symmetric and every rate label carries $P$.
This proves the detailed-balance identities
\eqref{eq:typed-detailed-balance}, including the stated pointwise
version of the terminal-rate density.

\subsection{First moments and Fisher--KPP level times}

\begin{proposition}
\label{prop:finite-state-tools}
Under Assumption~\ref{ass:environment}, let $P$ be the one-step mutation
matrix and
$A^\omega=P-I+\operatorname{diag}((\Lambda_v)_{v\in\Hc_d^{(b)}})$.
Then
\begin{equation}
 \Ew_x[N_t(y)]=K_t^\omega(x,y)=\bigl(e^{tA^\omega}\bigr)(x,y).
 \label{eq:FKmatrix}
\end{equation}
If all rates are nonnegative and at least one is positive, then
$w^\omega(t,x)\to1$ at every vertex.  Consequently, the level times
$\mathcal T_\alpha[w^\omega](x)$ and
$\mathcal T_\alpha[v^\omega](x)$ are finite for every $\alpha\in(0,1)$.
\end{proposition}

\begin{proof}
First-event conditioning gives
$\partial_t\Ew_x[N_t(y)]=(A^\omega\Ew_{\cdot}[N_t(y)])(x)$,
with initial value $\one_{\{x=y\}}$.  The finitely many rates in a fixed
environment and domination by a Yule process with splitting rate
$\lambda_+$ justify this calculation.  Conditioning the Feynman--Kac
expectation on the first mutation gives the same finite linear system,
so uniqueness proves \eqref{eq:FKmatrix}.

Now suppose that all rates are nonnegative and at least one is positive.
Comparison gives $0\le w^\omega\le1$ and
$w^\omega(1,x)\ge(e^{P-I})(x,\mathbf0)>0$, by connectivity.
Choose $0<c<\min_x w^\omega(1,x)$ and let $w_c$ solve the same
Fisher--KPP system with constant initial value $c$.  Since
$\Lambda_xc(1-c)\ge0$, the constant $c$ is a subsolution; comparison
and the autonomous solution map give
$c\le w_c(t,x)\le w_c(t+s,x)\le1$ for $s,t\ge0$.
Hence, $w_c(t,\cdot)$ converges to $w_\infty\in[c,1]^{\Hc_d^{(b)}}$.
Continuity of the solution map makes this limit an equilibrium.
Since $P$ is doubly stochastic, summing its equilibrium equations gives
$$
 0=\sum_x\Lambda_xw_\infty(x)(1-w_\infty(x)).
$$
Every summand is nonnegative, so $w_\infty(x)=1$ wherever
$\Lambda_x>0$.  At a vertex with value one, the equilibrium equation
gives $\sum_yP(x,y)w_\infty(y)=1$, forcing all its neighbours to have
value one.  Connectivity therefore gives $w_\infty\equiv1$.
Finally, $w^\omega(t+1,x)\ge w_c(t,x)$ proves convergence to one,
and $w^\omega\le v^\omega$ proves finiteness of both level times.
\end{proof}
\endgroup

\section{Auxiliary analytic and combinatorial estimates}
\label{sec:auxiliary}

The following inequalities identify a strictly positive spatial term in
the maximizing exponent.  In Lemma~\ref{lem:variational-regularity},
this forces the clock tilt to be positive whenever the full exponent is
nonnegative, allowing the local clock estimates to apply near the root.

\begin{lemma}
\label{lem:positive-log-derivatives}
For every $b\ge2$ and $z>0$,
$$
 \frac{zA_b'(z)}{A_b(z)}-\log A_b(z)>0,
 \qquad
 \frac{zC_b'(z)}{C_b(z)}-\log C_b(z)>0.
$$
\end{lemma}

\begin{proof}
From
$A_b(z)=b^{-1}e^z\left(1-e^{-\frac{b}{b-1}z}\right)$,
$$
 \frac{zA_b'(z)}{A_b(z)}-\log A_b(z)
 =\frac{\frac{b}{b-1}z}{e^{\frac{b}{b-1}z}-1}
  -\log\left(1-e^{-\frac{b}{b-1}z}\right)+\log b>0.
$$
For the second inequality, $C_b(0)=1$, $C_b'=A_b$, and
$$
 \frac{A_b(z)}{C_b(z)}
 =\frac{1-e^{-\frac{b}{b-1}z}}{1+(b-1)e^{-\frac{b}{b-1}z}}
$$
is strictly increasing.  Therefore,
$$
 \log C_b(z)=\int_0^z\frac{A_b(s)}{C_b(s)}\,\dd s
 <z\frac{A_b(z)}{C_b(z)}
 =\frac{zC_b'(z)}{C_b(z)}.
$$
This proves both inequalities.
\end{proof}

The next lemma gives the strict entropy bound used in the midpoint
analysis.

\begin{lemma}
\label{lem:bridge-entropy-convexity}
Let $X$ be the continuous-time walk on one $b$-ary coordinate, with
transition probabilities
$$
 p_s(x,y)=
 \begin{cases}
 \dfrac{1+(b-1)e^{-\beta s}}{b},&x=y,\\[6pt]
 \dfrac{1-e^{-\beta s}}{b},&x\ne y,
 \end{cases}
 \qquad \beta=\frac b{b-1}.
$$
For $0\le s\le z$, let $D_\sigma(s;z)$ be the relative entropy of
$\mathcal L(X_s\mid X_0=x,X_z=y)$ with respect to
$\mathcal L(X_s\mid X_0=x)$, where $\sigma=b-1$ if $x=y$ and
$\sigma=-1$ if $x\ne y$.  For every compact interval
$J\Subset(0,\infty)$ there is $c_J>0$ such that
$$
 D_\sigma(\alpha z;z)-\alpha D_\sigma(z;z)
 \le-c_J\alpha(1-\alpha)
$$
for $z\in J$, $0\le\alpha\le1$, and both values of $\sigma$.
\end{lemma}

\begin{proof}
We adapt the entropy argument in the proof of Proposition~1.5 and
Remark~3.2 of~\cite{MartinssonCartesian} to an arbitrary
duration $z>0$, allowing equal endpoints.  All expectations and entropies
in this proof are under the bridge from $x$ to $y$ in time $z$.

View the coordinate walk on $\mathbb Z/b\mathbb Z$.  Summing over its
increments before and after an interval of length $t$ shows that the
conditional law of the increment over that interval is independent of
its location.  More explicitly, for $0\le a\le z-t$ it assigns mass
\[
 \frac{p_t(0,w)p_{z-t}(0,y-x-w)}{p_z(x,y)}
\]
to $X_{a+t}-X_a=w$.  The definition of relative entropy therefore gives
\[
 \E\bigl[\log p_t(X_a,X_{a+t})\bigr]
 =\E\bigl[\log p_t(X_{z-t},y)\bigr]
 =D_\sigma(z-t;z)+\log p_z(x,y).
\]
For $0\le a\le v\le z$, expand the bridge conditional probability
$p_{v-a}(X_a,X_v)p_{z-v}(X_v,y)/p_{z-a}(X_a,y)$.
The conditional entropy of $X_v$ given $X_a$ is consequently
\[
 D_\sigma(a;z)-D_\sigma(z-v+a;z)-D_\sigma(v;z)-\log p_z(x,y).
\]
The bridge is Markov, so this conditional entropy decreases as $a$
increases.  Take $v=z-h$ and compare $a=s-h$ with $a=s$, where
$0\le s-h\le s+h\le z$.  Their difference is
\[
 D_\sigma(s-h;z)+D_\sigma(s+h;z)-2D_\sigma(s;z)\ge0.
\]
Thus $D_\sigma(\,\cdot\,;z)$ is convex by continuity.

Its endpoint values are $D_\sigma(0;z)=0$ and
$D_\sigma(z;z)=-\log p_z(x,y)$.  The same conditional-entropy identity
with $a=0$ and $v=z/2$ shows that
$D_\sigma(z/2;z)-\tfrac12D_\sigma(z;z)$ is minus half the Shannon
entropy of $X_{z/2}$.  This entropy has a positive minimum over $z\in J$
and both endpoint types, since all $b$ midpoint probabilities are
positive and continuous in $z$.  Convexity and the two chords joining
the midpoint to the endpoints now give
\[
 D_\sigma(\alpha z;z)-\alpha D_\sigma(z;z)
 \le-c_J\min\{\alpha,1-\alpha\}
 \le-c_J\alpha(1-\alpha),
\]
as required.
\end{proof}

The next estimate gives continuity of $\chi$ at zero mutation-count
density, including when the clock infimum occurs at the transform boundary.

\begin{lemma}\label{lem:smallnuexpansions}
Let $J\Subset(0,\infty)$ be compact and write
$\kappa_\partial=\lambda_+-1$.  As $\nu\downarrow0$, uniformly for $u\in J$,
\begin{equation}
 \chi(\nu,u)=\kappa_\partial u+
 O_J\bigl(\nu(1+|\log\nu|)\bigr).                   \label{eq:smallnuClock}
\end{equation}
\end{lemma}

\begin{proof}
Choose a finite $L>1+|\lambda_+|$ such that
$P([-L,\lambda_+])>0$, and put $D=\lambda_++L>0$.
For $s=\kappa-\kappa_\partial>0$,
$$
 \frac{P([-L,\lambda_+])}{s+D}
 \le H(\kappa_\partial+s)\le\frac1s.
$$
For all sufficiently small $\nu$, uniformly over $u\in J$,
the infimum of $su-\nu\log(s+D)$ over $s>0$ occurs as $s\downarrow0$.
For the upper bound choose $s=\nu/u$.  Hence,
$$
 \nu\log\frac{P([-L,\lambda_+])}{D}
 \le\chi(\nu,u)-\kappa_\partial u
 \le\nu\left(1+\log\frac u\nu\right),
$$
which proves \eqref{eq:smallnuClock}.
\end{proof}

\begin{lemma}\label{lem:onevariableassociation}
Let $\Lambda$ have any real-valued law.  If the bounded functions
$f_1,\ldots,f_k\ge0$ are all nondecreasing or are all nonincreasing, then
$$
 \E\!\left[\prod_{j=1}^kf_j(\Lambda)\right]
 \ge\prod_{j=1}^k\E\bigl[f_j(\Lambda)\bigr].
$$
\end{lemma}

\begin{proof}
For an independent copy $\Lambda'$ and two bounded functions $f,g$ with the
same monotonicity,
\begin{equation}
 \operatorname{Cov}\bigl(f(\Lambda),g(\Lambda)\bigr)
 =\frac12\E\!\Bigl[\bigl(f(\Lambda)-f(\Lambda')\bigr)
                   \bigl(g(\Lambda)-g(\Lambda')\bigr)\Bigr]\ge0.             \label{eq:associationidentity}
\end{equation}
The product of nonnegative functions with the same monotonicity again has
that monotonicity.  Applying \eqref{eq:associationidentity} to
$\prod_{j=1}^{k-1}f_j$ and $f_k$, and then inducting on $k$, proves the
claim.
\end{proof}

\noindent\textbf{Acknowledgement.}
The material in this paper is partly supported by the Air Force Office of
Scientific Research under award
number FA9550-20-1-0397 and ONR N000142412655. Support from NSF 2229012,
2312204, 2403007 is also gratefully acknowledged. ZZ gratefully acknowledges
support from a Jump Trading Fellowship. 
The authors used OpenAI GPT-5.6 and GPT-6 for assistance with
proof exploration, exposition, and editing.

\end{document}